\documentclass{amsart}
\usepackage{amssymb}
\usepackage{braket}
\usepackage{mathrsfs}
\usepackage[all]{xy}
\usepackage{graphicx}
\usepackage{color}

\newtheorem{thm}{Theorem}[section]
\newtheorem{cor}[thm]{Corollary}
\newtheorem{prop}[thm]{Proposition}
\theoremstyle{definition}
\newtheorem{dfn}[thm]{Definition}
\newtheorem{ex}[thm]{Example}
\newtheorem{claim}[thm]{Claim}
\newtheorem{lem}[thm]{Lemma}
\newtheorem{fact}[thm]{Fact}
\newtheorem{cond}[thm]{Condition}

\theoremstyle{remark}
\newtheorem{rem}[thm]{Remark}

\newcommand{\spmatrix}[2]{\big[\!\!\!\begin{array}{c}{}_{#1}\\ {}^{#2}\end{array}\!\!\!\!\big]}

\numberwithin{equation}{section}

\begin{document}
%\subjclass[2000]{,}

\title[A simultaneous generalization of mutation and recollement]{A simultaneous generalization of mutation and recollement of cotorsion pairs on a triangulated category}

\author{Hiroyuki NAKAOKA}
\address{Research and Education Assembly, Science and Engineering Area, Research Field in Science, Kagoshima University, 1-21-35 Korimoto, Kagoshima, 890-0065 Japan}

\email{nakaoka@sci.kagoshima-u.ac.jp}

\thanks{The author wishes to thank Professor Yann Palu for stimulating arguments, plenty of ideas, and useful comments and advices.}
\thanks{The author wishes to thank Professor Steffen Koenig and Professor Jorge Vit\'{o}ria for introducing him the notion of recollement, and for their advices.}
\thanks{The author wishes to thank Professor Osamu Iyama for introducing their results to him.}
\thanks{This work is supported by JSPS KAKENHI Grant Numbers 25800022,\, 24540085.}

\begin{abstract}
In this article, we introduce the notion of {\it concentric twin cotorsion pair} on a triangulated category. This notion contains the notions of $t$-structure, cluster tilting subcategory, co-$t$-structure and functorally finite rigid subcategory as examples. Moreover, a recollement of triangulated categories can be regarded as a special case of concentric twin cotorsion pair.

To any concentric twin cotorsion pair, we associate a pretriangulated subquotient category. This enables us to give a simultaneous generalization of the Iyama-Yoshino reduction and the recollement of cotorsion pairs. This allows us to give a generalized mutation on cotorsion pairs defined by the concentric twin cotorsion pair.
\end{abstract}

\maketitle

\tableofcontents

\section{Introduction and Preliminaries}

Let $\mathcal{C}$ be a triangulated category with the shift functor $[1]$, throughout this article.
A cotorsion pair $(\mathcal{U},\mathcal{V})$, essentially equal to the notion of a torsion pair (\cite[Definition 2.2]{IYo}) on $\mathcal{C}$, is a unifying notion of $t$-structure (\cite[D\'{e}finition 1.3.1]{BBD}) and cluster tilting subcategory (\cite[section 2.1]{KR}, \cite[Definition 3.1]{KZ}).
\begin{dfn}\label{DefCPC}
Let $\mathcal{U},\mathcal{V}\subseteq\mathcal{C}$ be full additive subcategories closed under isomorphisms and direct summands.
The pair $(\mathcal{U},\mathcal{V})$ is called a {\it cotorsion pair} on $\mathcal{C}$ if it satisfies the following.
\begin{itemize}
\item[{\rm (i)}] $\mathcal{C}=\mathcal{U}\ast\mathcal{V} [1]$.
\item[{\rm (ii)}] $\mathrm{Ext}^1(\mathcal{U},\mathcal{V})=0$, where $\mathrm{Ext}^1(X,Y)=\mathcal{C}(X,Y[1])$ for any $X,Y\in\mathcal{C}$.
\end{itemize}
Denote the class of cotorsion pairs on $\mathcal{C}$ by $\mathfrak{CP}(\mathcal{C})$.
In a cotorsion pair $(\mathcal{U},\mathcal{V})$, subcategories $\mathcal{U}$ and $\mathcal{V}$ determine each other, as right and left orthogonal categories. Indeed, we have $\mathcal{V}=\mathcal{U}[-1]^{\perp}$ and $\mathcal{U}={}^{\perp}\mathcal{V}[1]$. Here, $\mathcal{U}[-1]^{\perp}$ denotes the full subcategory of $\mathcal{C}$ consisting of objects $C\in\mathcal{C}$ satisfying $\mathcal{C}(\mathcal{U}[-1],C)=0$. Similarly for ${}^{\perp}\mathcal{V}[1]$. In particular, $\mathcal{U}$ and $\mathcal{V}$ are closed under extensions. A pair $(\mathcal{U},\mathcal{V})$ is cotorsion pair if and only if $(\mathcal{U}[-1],\mathcal{V})$ is {\it torsion pair}.
\end{dfn}
In the above definition, for any pair of subcategories $\mathcal{X},\mathcal{Y}\subseteq\mathcal{C}$, we denoted by $\mathcal{X}\ast\mathcal{Y}\subseteq\mathcal{C}$ the full subcategory of $\mathcal{C}$ consisting of those $C\in\mathcal{C}$ admitting a distinguished triangle
\begin{equation}\label{XCY}
X\to C\to Y\to X[1]
\end{equation}
with $X\in\mathcal{X}$ and $Y\in\mathcal{Y}$. Abbreviately, for any pair of objects $X,Y\in\mathcal{C}$,  we denote by $X\ast Y\subseteq\mathcal{C}$ the full subcategory of $\mathcal{C}$ consisting of those $C\in\mathcal{C}$ admitting a distinguished triangle $(\ref{XCY})$.

\medskip

The aim of this article is to give a simultaneous generalization of the following two constructions. The one is the {\it mutation} in a triangulated category, originally due to Iyama and Yoshino \cite{IYo}. Mutations in triangulated categories are investigated by several researchers, such as \cite{AI}, \cite{IYo}, \cite{ZZ1}. The key result is the following.
\begin{fact}\label{FactIYo}$($\cite[Theorem 4.2]{IYo}$)$
Let $\mathcal{I}\subseteq\mathcal{Z}\subseteq\mathcal{C}$ be full additive subcategories closed under isomorphisms and direct summands. Assume $\mathcal{Z}$ is closed under extensions. If $(\mathcal{Z},\mathcal{Z})$ is an $\mathcal{I}$-mutation pair (\cite[Definition 2.5]{IYo}), i.e., if it satisfies
\[ \mathcal{Z}\subseteq(\mathcal{Z}[-1]\ast\mathcal{I})\cap\mathcal{I}[-1]^{\perp}\quad\text{and}\quad\mathcal{Z}\subseteq(\mathcal{I}\ast\mathcal{Z}[1])\cap {}^{\perp}\mathcal{I}[1], \]
then the ideal quotient $\mathcal{Z}/\mathcal{I}$ becomes a triangulated category.
\end{fact}
More precisely, we give a generalization of the following version of mutation for cotorsion pairs, given by Zhou and Zhu.
\begin{fact}\label{FactZZ}$($\cite[Theorem 3.5]{ZZ1}$)$
Let $\mathcal{I}\subseteq\mathcal{C}$ be a functorially finite additive rigid subcategory, closed under isomorphisms and direct summands. Assume $\mathcal{I}[-1]^{\perp}={}^{\perp}\mathcal{I}[1]$ holds. Denote this equal subcategory by $\mathcal{Z}$. Then, $(\mathcal{Z},\mathcal{Z})$ is an $\mathcal{I}$-mutation pair, and there is a bijection between the following classes.
\begin{enumerate}
\item The class of cotorsion pairs $(\mathcal{U},\mathcal{V})$ on $\mathcal{C}$ satisfying $\mathcal{I}\subseteq \mathcal{U}\subseteq\mathcal{Z}$ $($or equivalently, $\mathcal{I}\subseteq\mathcal{V}\subseteq\mathcal{Z}$$)$.
\item The class of cotorsion pairs on $\mathcal{Z}/\mathcal{I}$.
\end{enumerate}
Through this bijection, {\it mutation} of cotorsion pairs in the class {\rm (1)}, is defined by pulling back the shift $\mathbb{Z}$-action on the class {\rm (2)}.
\end{fact}

The other is the {\it recollement} of cotorsion pairs. A recollement of triangulated categories is introduced in \cite{BBD}. Recollements are also investigated by several researchers, such as \cite{IKM}, \cite{J}, \cite{LV}. 
\begin{dfn}$($\cite[1.4.3]{BBD}$)$
Let
\begin{equation}\label{IntroRecoll}
\xy
(-20,0)*+{\mathcal{N}}="0";
(0,0)*+{\mathcal{C}}="2";
(20,0)*+{\mathcal{C}_{\mathcal{N}}}="4";
(-10,4)*+{\perp}="11";
(-10,-5)*+{\perp}="12";
(10,4)*+{\perp}="13";
(10,-5)*+{\perp}="14";
{\ar@{^(->}_{i_{\ast}} "0";"2"};
{\ar@/^1.80pc/^{i^{!}} "2";"0"};
{\ar@/_1.80pc/_{i^{\ast}} "2";"0"};
{\ar_{j^{\ast}} "2";"4"};
{\ar@/^1.80pc/^{j_{\ast}} "4";"2"};
{\ar@/_1.80pc/_{j_{!}} "4";"2"};
\endxy
\end{equation}
be a diagram of triangulated categories and triangle functors between them. This is called a {\it recollement} of triangulated categories, if it satisfies the following conditions.
\begin{enumerate}
\item $i^{\ast}\dashv i_{\ast}\dashv i^!$ and $j_!\dashv j^{\ast}\dashv j_{\ast}$ are adjoint triplets.
\item $i_{\ast},j_!,j_{\ast}$ are fully faithful.
\item $j^{\ast}\circ i_{\ast}=0$ holds.
\item For any $C\in\mathcal{C}$, the units and counits of the above adjoints give distinguished triangles
\[ i_{\ast}i^! C\to C\to j_{\ast}j^{\ast}C\to (i_{\ast}i^! C)[1]%
\quad\text{and}\quad%
j_{!}j^{\ast} C\to C\to i_{\ast}i^{\ast}C\to (j_!j^{\ast} C)[1]. \]
\end{enumerate}
\end{dfn}
If a recollement $(\ref{IntroRecoll})$ is given, we can {\it glue} $t$-structures on $\mathcal{C}_{\mathcal{N}}$ and $\mathcal{N}$, to obtain a $t$-structure on $\mathcal{C}$ (\cite[Th\'{e}or\`{e}me 1.4.10]{BBD}). More generally, we can glue cotorsion pairs in the same way (\cite[Theorem 3.1, Theorem 3.3]{C}). For the details, see Fact \ref{FactChen}.

\medskip

We note that both constructions give correspondences of cotorsion pairs on triangulated categories.
In this article, we will give a simultaneous generalization of these constructions, using {\it twin cotorsion pairs}.
\begin{dfn}$($\cite[Definition 2.7]{N2}$)$
Let $(\mathcal{S},\mathcal{T}),(\mathcal{U},\mathcal{V})$ be two cotorsion pairs on $\mathcal{C}$. The pair $\mathcal{P}=((\mathcal{S},\mathcal{T}),(\mathcal{U},\mathcal{V}))$ is called a {\it twin cotorsion pair $(${\it TCP} {\rm for short}$)$} on $\mathcal{C}$ if it satisfies $\mathrm{Ext}^1(\mathcal{S},\mathcal{V})=0$. Note that this condition is equivalent to $\mathcal{S}\subseteq\mathcal{U}$, and also to $\mathcal{V}\subseteq\mathcal{T}$.

Remark that any cotorsion pair $(\mathcal{U},\mathcal{V})$ on $\mathcal{C}$ can be regarded as a twin cotorsion pair $((\mathcal{U},\mathcal{V}),(\mathcal{U},\mathcal{V}))$. A twin cotorsion pair $\mathcal{P}=((\mathcal{S},\mathcal{T}),(\mathcal{U},\mathcal{V}))$ is said to be {\it degenerated to a single cotorsion pair} if it satisfies $\mathcal{S}=\mathcal{U}$ (equivalently, $\mathcal{T}=\mathcal{V}$).
\end{dfn}

\begin{ex}\label{Thefollowingareexamples} The following are examples of twin cotorsion pairs.
\begin{enumerate}
\item A cotorsion pair $(\mathcal{U},\mathcal{V})$ satisfies $\mathcal{U}[1]\subseteq\mathcal{U}$ if and only if $(\mathcal{U}[-1],\mathcal{V}[1])$ is a $t$-structure.
\item A cotorsion pair $(\mathcal{U},\mathcal{V})$ satisfies $\mathcal{U}=\mathcal{V}$ if and only if $\mathcal{U}\subseteq\mathcal{C}$ is a cluster tilting subcategory.
\item A cotorsion pair $(\mathcal{U},\mathcal{V})$ is called a co-$t$-structure if it satisfies $\mathcal{U}[-1]\subseteq\mathcal{U}$ (\cite[Definition 2.4]{P}).
\item A TCP $((\mathcal{S},\mathcal{T}),(\mathcal{U},\mathcal{V}))$ satisfies $\mathcal{S}=\mathcal{V}$ if and only if $\mathcal{S}\subseteq\mathcal{C}$ is a functorially finite rigid subcategory.
\item A TCP $((\mathcal{S},\mathcal{T}),(\mathcal{U},\mathcal{V}))$ satisfies $\mathcal{S}=\mathcal{V}$ and $\mathcal{U}=\mathcal{T}$ if and only if $\mathcal{Z}=\mathcal{U}$ and $\mathcal{I}=\mathcal{S}$ satisfy the assumption of Fact \ref{FactZZ}.
\end{enumerate}
\end{ex}
Remark that, in each of these examples, TCP $\mathcal{P}=((\mathcal{S},\mathcal{T}),(\mathcal{U},\mathcal{V}))$ satisfies the additional condition $\mathcal{S}\cap\mathcal{T}=\mathcal{U}\cap\mathcal{V}$.
In this article, we call such $\mathcal{P}$ a {\it concentric} TCP (Definition \ref{DefConcentric}). 

In section \ref{section_General}, we review some properties of (twin) cotorsion pairs. In section \ref{section_Concentric}, we define the notion of concentric twin cotorsion pair, and state basic properties of the associated subquotient category $\mathcal{Z}/\mathcal{I}$.
In section \ref{section_PreTr}, we construct a pretriangulated structure on $\mathcal{Z}/\mathcal{I}$, for any concentric TCP $\mathcal{P}$ (Theorem \ref{ThmPreTriangulated}).
In section \ref{section_Mutation}, under the assumption of some condition (Condition \ref{CondCDE}), we show that this pretriangulated category becomes indeed triangulated (Corollary \ref{CorSigmaOmega}). We construct a bijection between $\mathfrak{CP}(\mathcal{Z}/\mathcal{I})$ and the class
\[ \mathfrak{M}_{\mathcal{P}}=\{(\mathcal{A},\mathcal{B})\in\mathfrak{CP}(\mathcal{C})\mid\mathcal{A}=\mathcal{U}\cap(\mathcal{S}[-1]\ast\mathcal{A}),\ \mathcal{B}=\mathcal{T}\cap(\mathcal{B}\ast\mathcal{V}[1])\} \]
as in Theorem \ref{ThmBij} and Corollary \ref{CorABBij2}.
This enables us to generalize {\it mutation} onto $\mathfrak{M}_{\mathcal{P}}$ (Definition \ref{DefGeneralMut}).
Indeed, in the last section \ref{section_Typical}, we will see how these constructions recover the above mentioned two constructions, {\it mutation} and {\it recollement}.

\section{General properties of (twin) cotorsion pairs}\label{section_General}

In this section, we review basic properties of cotorsion pairs, from \cite{AN},\cite{N1},\cite{N2}, \cite{N3} and \cite{ZZ2}.

\begin{fact}\label{FactCPAdj}
For any $(\mathcal{U},\mathcal{V})\in\mathfrak{CP}(\mathcal{C})$, define
\begin{eqnarray*}
\mathcal{I}&=&\mathcal{U}\cap\mathcal{V},\\
\mathcal{C}^+&=&\mathcal{I}\ast\mathcal{V}[1]=\mathcal{V}\ast\mathcal{V}[1],\\
\mathcal{C}^-&=&\mathcal{U}[-1]\ast\mathcal{I}=\mathcal{U}[-1]\ast\mathcal{U},\\
\mathcal{H}&=&\mathcal{C}^+\cap\mathcal{C}^-
\end{eqnarray*}
and $\mathscr{A}_{(\mathcal{U},\mathcal{V})}=\mathcal{H}/\mathcal{I}$, the {\it heart} of $(\mathcal{U},\mathcal{V})$. Then $\mathscr{A}_{(\mathcal{U},\mathcal{V})}$ becomes an abelian category $($\cite[Theorem 6.4]{N1}$)$.

Moreover, the following holds. For any morphism $f$ in $\mathcal{C}$, let $\underline{f}$ denote its image in the ideal quotient $\mathcal{C}/\mathcal{I}$.
\begin{enumerate}
\item $($\cite[Proposition 3.1]{AN}$)$ The inclusion $\mathcal{U}/\mathcal{I}\hookrightarrow\mathcal{C}/\mathcal{I}$ has a right adjoint $\omega_{\mathcal{U}}$. For any $C\in\mathcal{C}$, if we decompose it into a distinguished triangle
\begin{equation}\label{TriaAdj1}
U_C\overset{u_C}{\longrightarrow}C\to V_C[1]\to U_C[1]\quad (U_C\in\mathcal{U},V_C\in\mathcal{V}),
\end{equation}
then $U_C$ satisfies $U_C\cong\omega_{\mathcal{U}}(C)$ in $\mathcal{U}/\mathcal{I}$, and $\underline{u}_C\colon U_C\to C$ gives the counit of the adjoint.
In particular, in $(\ref{TriaAdj1})$, $U_C$ is determined up to isomorphism in $\mathcal{U}/\mathcal{I}$.

Dually, the inclusion $\mathcal{V}/\mathcal{I}\hookrightarrow\mathcal{C}/\mathcal{I}$ has a left adjoint $\sigma_{\mathcal{V}}$. For any $C\in\mathcal{C}$, any distinguished triangle
\[ U[-1]\to C\to V\to U\quad(U\in\mathcal{U},V\in\mathcal{V}) \]
gives $V\cong\sigma_{\mathcal{V}}(C)$ in $\mathcal{V}/\mathcal{I}$.
\item  $($\cite[Propositions 3.4,\, 4.2]{AN}$)$ The inclusion $\mathcal{C}^+/\mathcal{I}\hookrightarrow\mathcal{C}/\mathcal{I}$ has a left adjoint $\tau^+_{(\mathcal{U},\mathcal{V})}$. 
This restricts to give the left adjoint functor $\mathcal{C}^-/\mathcal{I}\to\mathscr{A}_{(\mathcal{U},\mathcal{V})}$ of the inclusion $\mathscr{A}_{(\mathcal{U},\mathcal{V})}\hookrightarrow\mathcal{C}^-/\mathcal{I}$, which we denote by the same symbol $\tau^+_{(\mathcal{U},\mathcal{V})}$.
Explicitly, $Z=\tau^+_{(\mathcal{U},\mathcal{V})}(C)$ is given by the following.
\begin{itemize}
\item[-] Decompose $C$ into a distinguished triangle
\[ U\to C\to V[1]\to U[1]\quad(U\in\mathcal{U},\, V\in\mathcal{V}). \]
\item[-] Decompose $U$ into a distinguished triangle
\[ U^{\prime}[-1]\to U\to V^{\prime}\to U^{\prime}\quad(U^{\prime}\in\mathcal{U},\, V^{\prime}\in\mathcal{V}). \]
\item[-] By the octahedron axiom, we obtain the following commutative diagram made of distinguished triangles.
\[
\xy
(-14,18)*+{U^{\prime}[-1]}="0";
(-14.1,0)*+{U}="2";
(-14,-16)*+{V^{\prime}}="4";
(-3,0)*+{C}="6";
(2,-9)*+{Z}="8";
(20,0)*+{V[1]}="10";
(-8,-8)*+_{_{\circlearrowright}}="12";
(5,-3)*+_{_{\circlearrowright}}="14";
(-10,5)*+_{_{\circlearrowright}}="14";
{\ar_{} "0";"2"};
{\ar_{} "2";"4"};
{\ar^{} "0";"6"};
{\ar_{} "2";"6"};
{\ar_{z} "6";"8"};
{\ar^{} "6";"10"};
{\ar_{} "4";"8"};
{\ar_{} "8";"10"};
\endxy
\]
\end{itemize}
Dually, the inclusion $\mathcal{C}^-/\mathcal{I}\hookrightarrow\mathcal{C}/\mathcal{I}$ has a right adjoint $\tau^-_{(\mathcal{U},\mathcal{V})}$, which restricts to give a right adjoint functor $\mathcal{C}^-/\mathcal{I}\to\mathscr{A}_{(\mathcal{U},\mathcal{V})}$ of the inclusion $\mathscr{A}_{(\mathcal{U},\mathcal{V})}\hookrightarrow\mathcal{C}^-/\mathcal{I}$.

These functors satisfy $\tau^+_{(\mathcal{U},\mathcal{V})}\circ\tau^-_{(\mathcal{U},\mathcal{V})}\cong\tau^-_{(\mathcal{U},\mathcal{V})}\circ\tau^+_{(\mathcal{U},\mathcal{V})}$.

\item $($\cite[Theorem 5.7]{AN}$)$ Define $H_{(\mathcal{U},\mathcal{V})}\colon \mathcal{C}\to\mathscr{A}_{(\mathcal{U},\mathcal{V})}$ to be the composition
\[ H_{(\mathcal{U},\mathcal{V})}=\bigg( \mathcal{C}\to\mathcal{C}/\mathcal{I}\overset{\tau^+_{(\mathcal{U},\mathcal{V})}}{\longrightarrow}\mathcal{C}^+/\mathcal{I}\overset{\tau^-_{(\mathcal{U},\mathcal{V})}}{\longrightarrow}\mathscr{A}_{(\mathcal{U},\mathcal{V})}\bigg). \]
Then $H_{(\mathcal{U},\mathcal{V})}$ is cohomological, satisfying $H_{(\mathcal{U},\mathcal{V})}(\mathcal{U})=H_{(\mathcal{U},\mathcal{V})}(\mathcal{V})=0$.

\item  $($\cite[Corollary 5.8]{AN}$)$ For any $X\in\mathcal{C}$, the following are equivalent.
\begin{itemize}
\item[{\rm (i)}] $H_{(\mathcal{U},\mathcal{V})}(X)=0$.
\item[{\rm (ii)}] There exists a distinguished triangle
\begin{equation}\label{UXVU}
U\to X\overset{v}{\longrightarrow}V[1]\to U[1]\quad(U\in\mathcal{U},V\in\mathcal{V})
\end{equation}
where $v$ factors through some $V_0\in\mathcal{V}$.
\item[{\rm (iii)}] In any distinguished triangle $(\ref{UXVU})$, the morphism $v$ factors through some $V_0\in\mathcal{V}$.
\item[{\rm (iv)}] $\tau^+_{(\mathcal{U},\mathcal{V})}(X)\in\mathcal{V}/\mathcal{I}$.
\end{itemize}

\end{enumerate}
\end{fact}

\begin{rem}\label{RemHeartEquiv}
For any $(\mathcal{S},\mathcal{T}),(\mathcal{U},\mathcal{V})\in\mathfrak{CP}(\mathcal{C})$, the following are equivalent (\cite[section 6]{ZZ2},\cite[Corollary 4.3]{N3}).
\begin{enumerate}
\item $H_{(\mathcal{U},\mathcal{V})}(\mathcal{S})=H_{(\mathcal{U},\mathcal{V})}(\mathcal{T})=0$ and $H_{(\mathcal{S},\mathcal{T})}(\mathcal{U})=H_{(\mathcal{S},\mathcal{T})}(\mathcal{V})=0$ hold.
\item There is an equivalence $\mathscr{A}_{(\mathcal{S},\mathcal{T})}\overset{\simeq}{\longrightarrow}\mathscr{A}_{(\mathcal{U},\mathcal{V})}$ which makes the following diagram commutative up to natural isomorphism.
\[
\xy
(0,6)*+{\mathcal{C}}="0";
(0,-10)*+{}="1";
(-12,-8)*+{\mathscr{A}_{(\mathcal{S},\mathcal{T})}}="2";
(12,-8)*+{\mathscr{A}_{(\mathcal{U},\mathcal{V})}}="4";
{\ar_{H_{(\mathcal{S},\mathcal{T})}} "0";"2"};
{\ar^{H_{(\mathcal{U},\mathcal{V})}} "0";"4"};
{\ar^{\simeq} "2";"4"};
{\ar@{}|\circlearrowright "0";"1"};
\endxy
\]
In this case, we say $(\mathcal{S},\mathcal{T})$ and $(\mathcal{U},\mathcal{V})$ are {\it heart-equivalent}.
\end{enumerate}
\end{rem}

Let us apply the above remarks to a twin cotorsion pair.
\begin{rem}\label{RemTCPAdj}
For any twin cotorsion pair $\mathcal{P}=((\mathcal{S},\mathcal{T}),(\mathcal{U},\mathcal{V}))$, these $(\mathcal{S},\mathcal{T})$ and $(\mathcal{U},\mathcal{V})$ are heart-equivalent if and only if
\[ H_{(\mathcal{S},\mathcal{T})}(\mathcal{U})=0\ \ \text{and}\ \ H_{(\mathcal{U},\mathcal{V})}(\mathcal{T})=0 \]
hold.
\end{rem}

\begin{rem}\label{RemIntegral}
If a TCP $\mathcal{P}$ satisfies the above condition, then its {\it heart} $($\cite[Definition 2.8]{N2}$)$
\[ ((\mathcal{S}[-1]\ast\mathcal{S})\cap(\mathcal{V}\ast\mathcal{V}[1]))/(\mathcal{U}\cap\mathcal{T}) \]
becomes integral preabelian category. Indeed, the same proof as in \cite[Theorem 6.3]{N2} works.
\end{rem}

\section{Concentric twin cotorsion pair}\label{section_Concentric}

We introduce the notion of a {\it concentric twin cotorsion pair}, and state its basic properties.

\begin{dfn}\label{DefNiNf}
Let $\mathcal{P}=((\mathcal{S},\mathcal{T}),(\mathcal{U},\mathcal{V}))$ be a TCP. Let us define as
\[ \mathcal{N}^i=\mathcal{S}\ast\mathcal{V}[1],\quad\mathcal{N}^f=\mathcal{S}[-1]\ast\mathcal{V}, \]
and put $\mathcal{Z}=\mathcal{T}\cap\mathcal{U}$.
\end{dfn}

\begin{rem}\label{RemTCP_UNTN}
Let $\mathcal{P}=((\mathcal{S},\mathcal{T}),(\mathcal{U},\mathcal{V}))$ be any TCP. The following hold.
\begin{enumerate}
\item $\mathcal{U}\cap\mathcal{N}^i=\mathcal{S}$ and $\mathcal{T}\cap\mathcal{N}^f=\mathcal{V}$.
\item Let $U\in\mathcal{U},T\in\mathcal{T}$ be any pair of objects. If $f\in\mathcal{C}(U,T[1])$ factors through some $N\in\mathcal{N}^i$, then $f=0$.
\end{enumerate}
\end{rem}

\begin{dfn}\label{DefConcentric}
A TCP $\mathcal{P}=((\mathcal{S},\mathcal{T}),(\mathcal{U},\mathcal{V}))$ is called {\it concentric} if it satisfies $\mathcal{S}\cap\mathcal{T}=\mathcal{U}\cap\mathcal{V}$. We denote this equal subcategory by $\mathcal{I}$.
\end{dfn}

\begin{rem}\label{RemConcentric}
For any concentric $\mathcal{P}$, the following holds.
\begin{itemize}
\item[{\rm (i)}] $\mathcal{S}\cap\mathcal{V}=\mathcal{I}$.
\item[{\rm (ii)}] $\mathcal{S}\subseteq\mathcal{S}[-1]\ast\mathcal{I},\ \mathcal{V}\subseteq\mathcal{I}\ast\mathcal{V}[1]$.
\end{itemize}
\end{rem}

\subsection{Subquotient category $\mathcal{Z}/\mathcal{I}$}

Let $\mathcal{P}=((\mathcal{S},\mathcal{T}),(\mathcal{U},\mathcal{V}))$ be a concentric TCP throughout this section. We continue to use the symbols $\mathcal{Z}=\mathcal{U}\cap\mathcal{T}$ and $\mathcal{I}=\mathcal{S}\cap\mathcal{V}$.
\begin{rem}\label{RemQuot1}
Let $f\in\mathcal{Z}(X,Y)$ be a morphism which induces isomorphism $\underline{f}$ in $\mathcal{Z}/\mathcal{I}$. Then the following holds.
\begin{enumerate}
\item There exist $J\in\mathcal{I}$ and $j\in\mathcal{C}(X,J)$, such that $\begin{bmatrix}f\\ j\end{bmatrix}\colon X\to Y\oplus J$ becomes a section in $\mathcal{C}$.
\item There exist $I\in\mathcal{I}$ and $i\in\mathcal{C}(I,Y)$, such that $[f\ \, i]\colon X\oplus I\to Y$ becomes a retraction in $\mathcal{C}$.
\end{enumerate}
\end{rem}

\begin{lem}\label{Lem3-1}
Let $f\in\mathcal{Z}(X,Y)$ be a morphism which induces isomorphism $\underline{f}$ in $\mathcal{Z}/\mathcal{I}$. Then the following holds.
\begin{enumerate}
\item If $f$ is a section, then $\mathrm{Cone}(f)\in\mathcal{I}$.
\item If $f$ is a retraction, then $\mathrm{Cone}(f)[-1]\in\mathcal{I}$.
\end{enumerate}
\end{lem}
\begin{proof}
Since $\underline{f}$ is isomorphism, there exist $I\in\mathcal{I}$ and $i\in\mathcal{Z}(I,Y)$ such that $[f\  i]\colon X\oplus I\to Y$ becomes a retraction by Remark \ref{RemQuot1}.
By the octahedron axiom, we have the following commutative diagram made of distinguished triangles.
\[
\xy
(-20,0)*+{X}="0";
(-1,8)*+{X\oplus I}="2";
(16,15)*+{I}="4";
(3,0)*+{Y}="6";
(16.1,0)*+{\mathrm{Cone}(f)}="8";
(16,-19)*+{{}^{\exists}Q}="10";
(36,0)*+{X[1]}="12";
(9,6)*+_{_{\circlearrowright}}="21";
(-5,3)*+_{_{\circlearrowright}}="22";
(12,-6)*+_{_{\circlearrowright}}="23";
(-12,7)*+{\spmatrix{1}{0}}="30";
%
%{\ar^{\left[\begin{array}{c}1\\0\end{array}\right]} "0";"2"};
{\ar^{} "0";"2"};
{\ar^{[0\ 1]} "2";"4"};
{\ar_{f} "0";"6"};
{\ar^{[f\ i]} "2";"6"};
{\ar^{} "6";"8"};
{\ar_{0} "6";"10"};
{\ar^{} "4";"8"};
{\ar^{{}^{\exists}h} "8";"10"};
{\ar_{0} "8";"12"};
\endxy
\]
This shows $h=0$, and thus $\mathrm{Cone}(f)$ is a summand of $I\in\mathcal{I}$. Since $\mathcal{I}$ is closed under direct summands, this shows $\mathrm{Cone}(f)\in\mathcal{I}$.
{\rm (2)} can be shown dually.
\end{proof}

\begin{cor}\label{Cor3-2}
For any $f\in\mathcal{Z}(X,Y)$, the following are equivalent.
\begin{enumerate}
\item $\underline{f}$ is isomorphism in $\mathcal{Z}/\mathcal{I}$.
\item There exist $I,J\in\mathcal{I}$ and an isomorphism from $X\oplus I$ to $Y\oplus J$ in $\mathcal{C}$, of the form
\begin{equation}\label{fijk}
\left[\begin{array}{cc}f&i\\ j&k\end{array}\right]\colon X\oplus I\overset{\cong}{\longrightarrow}Y\oplus J.
\end{equation}
\item $\mathrm{Cone}(f)\in\mathcal{I}\ast\mathcal{I}[1]$.
\item $\mathrm{Cone}(f)\in\mathcal{N}^i$.
\end{enumerate}
\end{cor}
\begin{proof}
$(1)\Rightarrow(2)$ By Remark \ref{RemQuot1}, there exist $J\in\mathcal{I}$ and $j\in\mathcal{C}(X,J)$ such that $f^{\prime}=\begin{bmatrix}f\\ j\end{bmatrix}\colon X\to Y\oplus J$ becomes a section. Then by Lemma \ref{Lem3-1}, there exists a distinguished triangle
\begin{equation}\label{Xovfp}
X\overset{f^{\prime}}{\longrightarrow}Y\oplus J\overset{p}{\longrightarrow}I\to X[1]
\end{equation}
satisfying $I\in\mathcal{I}$.
If we take a section $\begin{bmatrix}i\\ k\end{bmatrix}\colon I\to Y\oplus J$ of $p$, this induces isomorphism $(\ref{fijk})$.

$(2)\Rightarrow(1)$ Remark that $f$ is equal to the composition of
\[
\xy
(-36,0)*+{X}="0";
(-14,0)*+{X\oplus I}="2";
(14,0)*+{Y\oplus J}="4";
(38,0)*+{Y}="6";
{\ar^(0.42){\left[\begin{array}{c}1\\ 0\end{array}\right]} "0";"2"};
{\ar^{\left[\begin{array}{cc}f&i\\ j&k\end{array}\right]}_{\cong} "2";"4"};
{\ar^(0.54){\left[\begin{array}{cc}1&0\end{array}\right]} "4";"6"};
\endxy.
\]
Since these morphisms become isomorphisms in $\mathcal{Z}/\mathcal{I}$, so is $\underline{f}$.

$(1)\Rightarrow(3)$ Let $Q[-1]\to X\overset{f}{\longrightarrow}Y\to Q$ be a distinguished triangle. By Remark \ref{RemQuot1} and Lemma \ref{Lem3-1}, we have a distinguished triangle $(\ref{Xovfp})$. By the octahedron axiom, we obtain a commutative diagram
\[
\xy
(14,18)*+{J}="0";
(14.1,0)*+{I}="2";
(14,-16)*+{Q}="4";
(3,0)*+{Y\oplus J}="6";
(-2,-9)*+{Y}="8";
(-20,0)*+{X}="10";
(8,-8)*+_{_{\circlearrowright}}="12";
(-5,-3)*+_{_{\circlearrowright}}="14";
(10,5)*+_{_{\circlearrowright}}="14";
{\ar^{} "0";"2"};
{\ar^{} "2";"4"};
{\ar_{} "0";"6"};
{\ar_{} "6";"2"};
{\ar^{[1\ 0]} "6";"8"};
{\ar^{f^{\prime}} "10";"6"};
{\ar_{} "8";"4"};
{\ar_{f} "10";"8"};
\endxy
\]
in which $J\to I\to Q\to J[1]$ is a distinguished triangle. This shows $Q\in\mathcal{I}\ast\mathcal{I} [1]$.

$(3)\Rightarrow(1)$ Let $Q[-1]\to X\overset{f}{\longrightarrow}Y\overset{g}{\longrightarrow}Q$ be a distinguished triangle. The following two cases are easy.
\begin{itemize}
\item[{\rm (i)}] If $Q\in\mathcal{I}$, then $\mathcal{C}(Q[-1],X)=0$, and thus $g$ has a section $s\colon Q\to Y$. This gives an isomorphism $[f\ s]\colon X\oplus Q\overset{\cong}{\longrightarrow}Y$ in $\mathcal{C}$. Thus
\[ \underline{f}=\underline{\begin{bmatrix}f&s\end{bmatrix}}\circ\underline{\begin{bmatrix}1\\0\end{bmatrix}}\colon X\to Y \]
is also isomorphism in $\mathcal{Z}/\mathcal{I}$.
\item[{\rm (ii)}] If $Q\in\mathcal{I}[1]$, we can show that $\underline{f}$ is isomorphism in $\mathcal{Z}/\mathcal{I}$, in a dual manner.
\end{itemize}
Generally, suppose $Q$ belongs to $\mathcal{I}\ast\mathcal{I}[1]$. Let
\[ J\to I\to Q\to J[1]\quad(I,J\in\mathcal{I}) \]
be a distinguished triangle.
By the octahedron axiom, we obtain a commutative diagram made of distinguished triangles
\[
\xy
(-20,0)*+{X}="0";
(-1,8)*+{{}^{\exists}Z}="2";
(16,15)*+{I}="4";
(3,0)*+{Y}="6";
(16.1,0)*+{Q}="8";
(16,-19)*+{J[1]}="10";
(9,6)*+_{_{\circlearrowright}}="12";
(-5,3)*+_{_{\circlearrowright}}="14";
(12,-6)*+_{_{\circlearrowright}}="14";
{\ar^{x} "0";"2"};
{\ar^{} "2";"4"};
{\ar_{f} "0";"6"};
{\ar^{y} "2";"6"};
{\ar^{} "6";"8"};
{\ar_{} "6";"10"};
{\ar^{} "4";"8"};
{\ar^{} "8";"10"};
\endxy
\]
satisfying $Z\in X\ast I\subseteq\mathcal{Z}$. By {\rm (i)} and {\rm (ii)}, morphisms $\underline{x}$ and $\underline{y}$ are isomorphisms in $\mathcal{Z}/\mathcal{I}$. Thus $\underline{f}=\underline{y}\circ\underline{x}$ is also isomorphism.

$(3)\Rightarrow(4)$ This is trivial.

$(4)\Rightarrow(3)$ Let $X\overset{f}{\longrightarrow}Y\to N\to X[1]$ and $S\to N\to V[1]\to S[1]$ be distinguished triangles satisfying $S\in\mathcal{S}$ and $V\in\mathcal{V}$. By the octahedron axiom, we obtain the following diagram made of distinguished triangles.
\[
\xy
(-8,23)*+{V}="-12";
(8,23)*+{V}="-14";
(-24,8)*+{X}="0";
(-8,8)*+{{}^{\exists}Q}="2";
(8,8)*+{S}="4";
(24,8)*+{X[1]}="6";
(-24,-7)*+{X}="10";
(-8,-7)*+{Y}="12";
(8,-7)*+{N}="14";
(24,-7)*+{X[1]}="16";
(-8,-22)*+{V[1]}="22";
(8,-22)*+{V[1]}="24";
{\ar@{=} "-12";"-14"};
{\ar_{} "-12";"2"};
{\ar^{} "-14";"4"};
{\ar^{} "0";"2"};
{\ar^{} "2";"4"};
{\ar^{} "4";"6"};
{\ar@{=} "0";"10"};
{\ar_{} "2";"12"};
{\ar^{} "4";"14"};
{\ar@{=} "6";"16"};
{\ar_{} "10";"12"};
{\ar^{} "12";"14"};
{\ar^{} "14";"16"};
{\ar_{} "12";"22"};
{\ar^{} "14";"24"};
{\ar@{=} "22";"24"};
{\ar@{}|\circlearrowright "-12";"4"};
{\ar@{}|\circlearrowright "0";"12"};
{\ar@{}|\circlearrowright "2";"14"};
{\ar@{}|\circlearrowright "4";"16"};
{\ar@{}|\circlearrowright "12";"24"};
\endxy
\]
By $\mathcal{C}(S,X[1])=0$ and $\mathcal{C}(Y,V[1])=0$, it follows $Q\cong V\oplus Y\cong S\oplus X$ in $\mathcal{C}$. In particular, we have $S\in\mathcal{S}\cap\mathcal{Z}=\mathcal{I}$ and $V\in\mathcal{V}\cap\mathcal{Z}=\mathcal{I}$. Thus $\mathrm{Cone} (f)\in\mathcal{I}\ast\mathcal{I}[1]$ holds.
\end{proof}

\begin{rem}\label{RemQuot2}
Let $p\colon\mathcal{C}\to\mathcal{C}/\mathcal{I}$ denote the residue functor. Then, the following is a mutually inverse bijective correspondence.
\begin{eqnarray*}
&\Set{ \mathcal{X}\subseteq\mathcal{C} | \begin{array}{l}\text{full additive subcategory,}\\%
\text{closed under isomorphisms and direct summands,}\\%
\text{containing}\ \mathcal{I}\end{array}}&\\
&p\downarrow\ \ \ \uparrow p^{-1}&\\
&\Set{ \mathscr{X}\subseteq\mathcal{C}/\mathcal{I} | \begin{array}{l}\text{full additive subcategory,}\\%
\text{closed under isomorphisms and direct summands,}\end{array}}.&
\end{eqnarray*}
Also remark that the corresponding $\mathcal{X}$ and $\mathscr{X}$ have the same class of objects $\mathrm{Ob}(\mathcal{X})=\mathrm{Ob}(\mathscr{X})$. In the sequel, we use the notation $\widetilde{\mathscr{X}}=p^{-1}(\mathscr{X})$, and $\overline{\mathcal{X}}=p(\mathcal{X})$.

%From this, we simply denote by $X\in\mathcal{X}$, to mean $X\in\mathrm{Ob}(\mathcal{X})$, or equivalently, $p(X)\in\mathrm{Ob}(\mathscr{X})$.
\end{rem}

\begin{claim}\label{ClaimConcentric}
For any concentric $\mathcal{P}$, the following holds.
\begin{enumerate}
\item $(\mathcal{S},\mathcal{T})$ is a $t$-structure $\Leftrightarrow$ $(\mathcal{U},\mathcal{V})$ is a $t$-structure.
\item $\mathcal{S}=\mathcal{T}$ (i.e. $\mathcal{T}\subseteq\mathcal{C}$ is a cluster tilting subcategory) $\Leftrightarrow$ $\mathcal{U}=\mathcal{V}$. In this case, $\mathcal{S}=\mathcal{T}=\mathcal{U}=\mathcal{V}$ holds.
\end{enumerate}
\end{claim}
\begin{proof}
{\rm (1)} $(\mathcal{S},\mathcal{T})$ is a $t$-structure $\Leftrightarrow$ $\mathcal{I}=0$ $\Leftrightarrow$ $(\mathcal{U},\mathcal{V})$ is a $t$-structure.

{\rm (2)} If $\mathcal{S}=\mathcal{T}$, then $\mathcal{S}=\mathcal{T}=\mathcal{I}=\mathcal{U}\cap\mathcal{V}$. Then $(\mathcal{I},\mathcal{I})$ and $(\mathcal{U},\mathcal{V})$ are cotorsion pairs satisfying $\mathcal{I}\subseteq\mathcal{U}$ and $\mathcal{I}\subseteq\mathcal{V}$. Thus it follows $(\mathcal{I},\mathcal{I})=(\mathcal{U},\mathcal{V})$.
\end{proof}

\subsection{Adjointness}
Let us apply the remarks in section \ref{section_General} to $\mathcal{P}$.
\begin{dfn}\label{DefTCPAdj}
Let $\mathcal{P}$ be a concentric TCP. By Fact \ref{FactCPAdj}, we have the following.
\begin{itemize}
\item[{\rm (i)}] Right adjoint functor $\omega_{\mathcal{U}}\colon\mathcal{C}/\mathcal{I}\to\mathcal{U}/\mathcal{I}$ of the inclusion. Let $\varepsilon_{\mathcal{U}}$ denote the counit of this adjunction.
\item[{\rm (ii)}] Left adjoint functor $\sigma_{\mathcal{T}}\colon\mathcal{C}/\mathcal{I}\to\mathcal{T}/\mathcal{I}$ of the inclusion. Let $\eta_{\mathcal{T}}$ denote the unit.
\end{itemize}
These restricts to yield the following.
\begin{itemize}
\item[{\rm (iii)}] Right adjoint functor $\omega\colon\mathcal{T}/\mathcal{I}\to\mathcal{Z}/\mathcal{I}$ of the inclusion. Let $\varepsilon$ denote the counit.
\item[{\rm (iv)}] Left adjoint functor $\sigma\colon\mathcal{U}/\mathcal{I}\to\mathcal{Z}/\mathcal{I}$ of the inclusion.  Let $\eta$ denote the unit.
\end{itemize}
We will use these notations in the rest of this article.
In particular, 
\[ \omega_{\mathcal{U}}|_{(\mathcal{U}/\mathcal{I})},\ \sigma_{\mathcal{T}}|_{(\mathcal{T}/\mathcal{I})},\ %
\omega|_{(\mathcal{Z}/\mathcal{I})},\ \sigma|_{(\mathcal{Z}/\mathcal{I})} \]
are natural isomorphisms. We can even choose the above adjoint functors so that these natural isomorphisms become identities.
\end{dfn}

\begin{lem}\label{LemSUU}
Let $\mathcal{P}$ be a concentric TCP.
\begin{enumerate}
\item For any distinguished triangle in $\mathcal{C}$
\[ S[-1]\to U\overset{u}{\longrightarrow}X\to S\quad(S\in\mathcal{S},U\in\mathcal{U}), \]
we have $X\in\mathcal{U}$, and $\sigma(\underline{u})\colon\sigma(U)\to\sigma(X)$ is isomorphism in $\mathcal{Z}/\mathcal{I}$.
\item For any distinguished triangle in $\mathcal{C}$
\[ V\to X\overset{t}{\longrightarrow}T\to V[1]\quad(V\in\mathcal{V},T\in\mathcal{T}) ,\]
we have $X\in\mathcal{T}$, and $\omega(\underline{t})\colon\omega(X)\to\omega(T)$ is isomorphism in $\mathcal{Z}/\mathcal{I}$.
\end{enumerate}
\end{lem}
\begin{proof}
{\rm (1)} $X\in U\ast S\subseteq\mathcal{U}\ast\mathcal{S}=\mathcal{U}$ is obvious.
If we decompose $X$ into a distinguished triangle in $\mathcal{C}$
\[ S_X[-1]\to X\to T_X\to S_X\quad(S_X\in\mathcal{S},T_X\in\mathcal{T}), \]
then by the octahedron axiom, we have the following commutative diagram made of distinguished triangles.
\[
\xy
(-20,0)*+{S[-1]}="0";
(-1,8)*+{E[-1]}="2";
(16,15)*+{S_X[-1]}="4";
(3,0)*+{U}="6";
(16.1,0)*+{X}="8";
(16,-19)*+{T_X}="10";
(9,6)*+_{_{\circlearrowright}}="12";
(-5,3)*+_{_{\circlearrowright}}="14";
(12,-6)*+_{_{\circlearrowright}}="14";
{\ar^{} "0";"2"};
{\ar^{} "2";"4"};
{\ar_{} "0";"6"};
{\ar^{} "2";"6"};
{\ar^{u} "6";"8"};
{\ar_{} "6";"10"};
{\ar^{} "4";"8"};
{\ar^{} "8";"10"};
\endxy
\]
which shows $E\in\mathcal{S}$, and thus gives $\sigma(U)\cong T_X\cong\sigma(X)$ in $\mathcal{Z}/\mathcal{I}$.
{\rm (2)} can be shown dually.
\end{proof}

\begin{lem}\label{LemNUU}
Let $\mathcal{P}$ be a concentric TCP.
\begin{enumerate}
\item For any distinguished triangle in $\mathcal{C}$
\[ N[-1]\to U_1\overset{u}{\longrightarrow}U_2\to N\quad(U_1,U_2\in\mathcal{U},N\in\mathcal{N}^i), \]
$\sigma(\underline{u})\colon\sigma(U_1)\to\sigma(U_2)$ is isomorphism in $\mathcal{Z}/\mathcal{I}$.
\item For any distinguished triangle in $\mathcal{C}$
\[ N[-1]\to T_1\overset{t}{\longrightarrow}T_2\to N\quad(T_1,T_2\in\mathcal{T},N\in\mathcal{N}^i) ,\]
$\omega(\underline{t})\colon\omega(T_1)\to\omega(T_2)$ is isomorphism in $\mathcal{Z}/\mathcal{I}$.
\end{enumerate}
\end{lem}
\begin{proof}
{\rm (1)} Decompose $N$ into a distinguished triangle in $\mathcal{C}$
\[ S\to N\to V[1]\to S[1]\quad(S\in\mathcal{S},V\in\mathcal{V}). \]
Then by the octahedron axiom and $\mathcal{C}(U_2,V[1])=0$, we obtain the following commutative diagram made of distinguished triangles.
\[
\xy
(-20,0)*+{U_1}="0";
(-1,8)*+{U_2\oplus V}="2";
(16,15)*+{S}="4";
(3,0)*+{U_2}="6";
(16.1,0)*+{N}="8";
(16,-19)*+{V[1]}="10";
(9,6)*+_{_{\circlearrowright}}="12";
(-5,3)*+_{_{\circlearrowright}}="14";
(12,-6)*+_{_{\circlearrowright}}="14";
{\ar^{u_1} "0";"2"};
{\ar^{} "2";"4"};
{\ar_{u} "0";"6"};
{\ar^{u_2} "2";"6"};
{\ar^{} "6";"8"};
{\ar_{0} "6";"10"};
{\ar^{} "4";"8"};
{\ar^{} "8";"10"};
\endxy
\qquad(u_2=[1\ \ 0]).
\]
From Lemma \ref{LemSUU}, it follows $U_2\oplus V\in\mathcal{U}$, and $\sigma(\underline{u}_1)\colon\sigma(U_1)\overset{\cong}{\longrightarrow}\sigma(U_2\oplus V)$ in $\mathcal{Z}/\mathcal{I}$.
In particular $V\in\mathcal{V}\cap\mathcal{U}=\mathcal{I}$, and thus $\sigma(\underline{u}_2)\colon\sigma(U_2\oplus V)\overset{\cong}{\longrightarrow}\sigma(U_2)$ in $\mathcal{Z}/\mathcal{I}$. As their composition, $\sigma(\underline{u})=\sigma(\underline{u}_2)\circ\sigma(\underline{u}_1)$ is isomorphism in $\mathcal{Z}/\mathcal{I}$. {\rm (2)} can be shown dually.
\end{proof}

\begin{prop}\label{PropNatural}
Let $\mathcal{P}$ be a concentric TCP. There is a $($unique$)$ natural transformation
\[ \mu\colon\sigma\circ\omega_{\mathcal{U}}\Rightarrow\omega\circ\sigma_{\mathcal{T}} \]
which makes the following diagram in $\mathcal{C}/\mathcal{I}$ commutative for any $X\in\mathcal{C}$.
\begin{equation}\label{PentaNat}
\xy
(-20,-2)*+{\omega_{\mathcal{U}}(X)}="0";
(-12,12)*+{\sigma\circ\omega_{\mathcal{U}}(X)}="2";
(12,12)*+{\omega\circ\sigma_{\mathcal{T}}(X)}="4";
(20,-2)*+{\sigma_{\mathcal{T}}(X)}="6";
(0,-14)*+{X}="8";
(0,16)*+{}="9";
{\ar^{(\eta\circ\omega_{\mathcal{U}})_X} "0";"2"};
{\ar^{\mu_X} "2";"4"};
{\ar^{(\varepsilon\circ\sigma_{\mathcal{T}})_X} "4";"6"};
{\ar_{(\varepsilon_{\mathcal{U}})_X} "0";"8"};
{\ar_{(\eta_{\mathcal{T}})_X} "8";"6"};
{\ar@{}|\circlearrowright "8";"9"};
\endxy
\end{equation}
\end{prop}
\begin{proof}
By the adjoint property, we have the following sequence of isomorphisms, for any $X\in\mathcal{C}$.
\begin{eqnarray*}
(\mathcal{C}/\mathcal{I})(\omega_{\mathcal{U}}(X),\sigma_{\mathcal{T}}(X))%
&\underset{\cong}{\overset{\varepsilon_{(\sigma_{\mathcal{T}}(X))}\circ-}{\longleftarrow}}&%
(\mathcal{U}/\mathcal{I})(\omega_{\mathcal{U}}(X),\omega(\sigma_{\mathcal{T}}(X)))\\
&\underset{\cong}{\overset{-\circ\eta_{(\omega_{\mathcal{U}}(X))}}{\longleftarrow}}&%
(\mathcal{Z}/\mathcal{I})(\sigma(\omega_{\mathcal{U}}(X)),\omega(\sigma_{\mathcal{T}}(X)))
\end{eqnarray*}
%\[
%\xy
%(-28,10)*+{(\mathcal{C}/\mathcal{I})(\omega_{\mathcal{U}}(X),\sigma_{\mathcal{T}}(X))}="0";
%(28,10)*+{(\mathcal{T}/\mathcal{I})(\sigma(\omega_{\mathcal{U}}(X)),\sigma_{\mathcal{T}}(X))}="2";
%(-28,-10)*+{(\mathcal{U}/\mathcal{I})(\omega_{\mathcal{U}}(X),\omega(\sigma_{\mathcal{T}}(X)))}="4";
%(28,-10)*+{(\mathcal{Z}/\mathcal{I})(\sigma(\omega_{\mathcal{U}}(X)),\omega(\sigma_{\mathcal{T}}(X)))}="6";
%%
%{\ar_{-\circ(\eta_{\sigma})_X}^{\cong} "2";"0"};
%{\ar_{(\varepsilon_{\omega})_X\circ-}^{\cong} "0";"4"};
%{\ar^{(\varepsilon_{\omega})_X\circ-}_{\cong} "2";"6"};
%{\ar^{-\circ(\eta_{\sigma})_X}_{\cong} "6";"4"};
%%
%{\ar@{}|\circlearrowright "0";"6"};
%\endxy
%\]
Thus $\mu_X\in(\mathcal{Z}/\mathcal{I})(\sigma(\omega_{\mathcal{U}}(X)),\omega(\sigma_{\mathcal{T}}(X)))$ is obtained as a morphism corresponding to $(\eta_{\mathcal{T}})_X\circ(\varepsilon_{\mathcal{U}})_X\in(\mathcal{C}/\mathcal{I})(\omega_{\mathcal{U}}(X),\sigma_{\mathcal{T}}(X))$.
\end{proof}

\begin{rem}
$(\ref{PentaNat})$ can be made commutative in $\mathcal{C}$, for any $X\in\mathcal{C}$. In fact, if we take distinguished triangles
\begin{eqnarray*}
U_X\overset{u_X}{\longrightarrow}X\to V_X[1]\to U_X[1]&& (U_X\in\mathcal{U},V_X\in\mathcal{V}),\\
S_U[-1]\to U_X\overset{z_U}{\longrightarrow}Z_U\to S_U&& (S_U\in\mathcal{S},Z_U\in\mathcal{Z}),\\
S_X[-1]\to X\overset{t_X}{\longrightarrow}T_X\to S_X&& (S_X\in\mathcal{S},T_X\in\mathcal{T}),\\
Z_T\overset{z_T}{\longrightarrow}T_X\to V_T[1]\to Z_T[1]&& (Z_T\in\mathcal{Z},V_T\in\mathcal{V}),
\end{eqnarray*}
then there exists a morphism $z\in\mathcal{Z}(Z_U,Z_T)$ which makes the following diagram commutative in $\mathcal{C}$.
\begin{equation}\label{PentaNat_inC}
\xy
(-16,-1)*+{U_X}="0";
(-8,8)*+{Z_U}="2";
(8,8)*+{Z_T}="4";
(16,-1)*+{T_X}="6";
(0,-9)*+{X}="8";
(0,10)*+{}="9";
{\ar^{z_U} "0";"2"};
{\ar^{z} "2";"4"};
{\ar^{z_T} "4";"6"};
{\ar_{u_X} "0";"8"};
{\ar_{t_X} "8";"6"};
{\ar@{}|\circlearrowright "8";"9"};
\endxy
\end{equation}
This gives
\[
\underline{t}_X=(\eta_{\mathcal{T}})_X,\ 
\underline{z}_U=\eta_{(U_X)},\ 
\underline{u}_X=(\varepsilon_{\mathcal{U}})_X,\ 
\underline{z}_T=\varepsilon_{(T_X)}
\]
%\begin{eqnarray*}
%\underline{t}_X=\eta_{\sigma_{\mathcal{T}}}&,&\underline{z}_U=\eta_{\sigma},\\
%\underline{u}_X=\varepsilon_{\omega_{\mathcal{U}}}&,&\underline{z}_T=\varepsilon_U
%\end{eqnarray*}
and $\underline{z}=\mu_X$.
\end{rem}

%As in Remark \ref{RemQuot2}, we use the following notation.
\begin{prop}\label{PropInverseImage}
Let $\mathcal{P}$ be a concentric TCP. Let $\mathscr{X}\subseteq\mathcal{C}/\mathcal{I}$ be a full additive subcategory, closed under isomorphisms and direct summands. Put $\mathcal{X}=\widetilde{\mathscr{X}}$ as in Remark \ref{RemQuot2}.
\begin{enumerate}
\item If $\mathscr{X}\subseteq\mathcal{U}/\mathcal{I}$, then $\widetilde{\omega_{\mathcal{U}}^{-1}(\mathscr{X})}=\mathcal{X}\ast\mathcal{V}[1]$.
\item If $\mathscr{X}\subseteq\mathcal{T}/\mathcal{I}$, then $\widetilde{\sigma_{\mathcal{T}}^{-1}(\mathscr{X})}=\mathcal{S}[-1]\ast\mathcal{X}$.
\item If $\mathscr{X}\subseteq\mathcal{Z}/\mathcal{I}$, then
\begin{eqnarray*}
&\widetilde{\sigma^{-1}(\mathscr{X})}=\mathcal{U}\cap(\mathcal{S}[-1]\ast\mathcal{X})=\mathcal{U}\cap(\mathcal{N}^f\ast\mathcal{X}),&\\
&\widetilde{\omega^{-1}(\mathscr{X})}=\mathcal{T}\cap(\mathcal{X}\ast\mathcal{V}[1])=\mathcal{T}\cap(\mathcal{X}\ast\mathcal{N}^i).&
\end{eqnarray*}
\end{enumerate}
These are full additive subcategories in $\mathcal{C}$, closed under isomorphisms and direct summands, containing $\mathcal{I}$.
\end{prop}
\begin{proof}
{\rm (1)} For any $C\in\mathcal{C}$, it satisfies $\omega_{\mathcal{U}}(C)\in\mathscr{X}$ if and only if a distinguished triangle in $\mathcal{C}$
\[ U\to C\to V[1]\to U[1]\quad (U\in\mathcal{U},V\in\mathcal{V}) \]
satisfies $U\in\mathcal{X}$. By Fact \ref{FactCPAdj} and Remark \ref{RemQuot2}, this does not depend on the choice of a distinguished triangle. This condition is equivalent to $C\in\mathcal{X}\ast\mathcal{V}[1]$.

{\rm (2)} can be shown dually.

{\rm (3)} $\widetilde{\sigma^{-1}(\mathscr{X})}=\mathcal{U}\cap(\mathcal{S}[-1]\ast\mathcal{X})$ can be shown in a similar way as in {\rm (2)}. Let us show $\mathcal{U}\cap(\mathcal{S}[-1]\ast\mathcal{X})\supseteq\mathcal{U}\cap(\mathcal{N}^f\ast\mathcal{X})$, since the converse is obvious.

Let $U\in\mathcal{U}\cap(\mathcal{N}^f\ast\mathcal{X})$ be any object. Then $U$ admits distinguished triangles
\begin{eqnarray*}
&N\to U\to X\to N[1],\quad S[-1]\to N\to V\to S&\\
&(X\in\mathcal{X},\, S\in\mathcal{S},\, V\in\mathcal{V}).&
\end{eqnarray*}
By the octahedron axiom, we obtain the following commutative diagram made of distinguished triangles.
\[
\xy
(-14,18)*+{S[-1]}="0";
(-14.1,0)*+{N}="2";
(-14,-16)*+{V}="4";
(-3,0)*+{U}="6";
(2,-9)*+{{}^{\exists}Y}="8";
(20,0)*+{X}="10";
(-8,-8)*+_{_{\circlearrowright}}="12";
(5,-3)*+_{_{\circlearrowright}}="14";
(-10,5)*+_{_{\circlearrowright}}="14";
{\ar_{} "0";"2"};
{\ar_{} "2";"4"};
{\ar^{} "0";"6"};
{\ar_{} "2";"6"};
{\ar_{} "6";"8"};
{\ar^{} "6";"10"};
{\ar_{} "4";"8"};
{\ar_{} "8";"10"};
\endxy
\]
Since $\mathcal{C}(X,V[1])=0$, we have $Y\cong X\oplus V$, and thus $X\oplus V\in U\ast S\subseteq\mathcal{U}$. This implies $V\in\mathcal{V}\cap\mathcal{U}=\mathcal{I}$, and thus $Y\in\mathcal{X}$. It follows $U\in\mathcal{U}\cap(\mathcal{S}[-1]\ast\mathcal{X})$.

\end{proof}

\begin{cor}\label{CorInverseImage}
Especially, the following are full additive subcategories in $\mathcal{C}$, closed under isomorphisms and direct summands.
\begin{eqnarray*}
&\mathcal{N}^i=\mathcal{S}\ast\mathcal{V}[1]=\widetilde{\omega_{\mathcal{U}}^{-1}(\overline{\mathcal{S}})}\ \ \ \ \subseteq\mathcal{C},&\\
&\mathcal{N}^f=\mathcal{S}[-1]\ast\mathcal{V}=\widetilde{\sigma_{\mathcal{T}}^{-1}(\overline{\mathcal{V}})}\ \ \subseteq\mathcal{C}.&
\end{eqnarray*}
\end{cor}

\section{Pretriangulated structure on $\mathcal{Z}/\mathcal{I}$} \label{section_PreTr}

Let $\mathcal{P}$ be a concentric TCP, throughout this section. The aim of this section is to give a pretriangulated structure on $\mathcal{Z}/\mathcal{I}$ (Theorem \ref{ThmPreTriangulated}).
\subsection{Right triangulated structure}
First, we will give a right triangulation of $\mathcal{Z}/\mathcal{I}$ (Proposition \ref{PropZIRightTria}).
\begin{dfn}\label{Def_poshi_neshi}
The functor $\langle1\rangle\colon\mathcal{Z}/\mathcal{I}\to\mathcal{U}/\mathcal{I}$ is defined in the following way, similarly as in \cite[section 2]{Ha} and \cite[Proposition 2.6, Definition 4.1]{IYo}.
\begin{itemize}
\item[{\rm (i)}] For any $Z\in\mathcal{Z}$, there is a distinguished triangle
\begin{equation}\label{Triaposhi1}
U[-1]\to Z\to I\to U\quad(U\in\mathcal{U},I\in\mathcal{I}).
\end{equation}
\item[{\rm (ii)}] For any (other) $Z^{\prime}\in\mathcal{Z}$ and any distinguished triangle
\begin{equation}\label{Triaposhi2}
U^{\prime}[-1]\to Z^{\prime}\to I^{\prime}\to U^{\prime}\quad(U^{\prime}\in\mathcal{U},I^{\prime}\in\mathcal{I}),
\end{equation}
we have a bijection
\[ (\mathcal{Z}/\mathcal{I})(Z,Z^{\prime})\overset{\cong}{\longrightarrow}(\mathcal{U}/\mathcal{I})(U,U^{\prime})\ ;\ \underline{z}\mapsto\underline{u}, \]
\end{itemize}
where $z$ and $u$ fit into a morphism of triangles $(z,i,u)$ from $(\ref{Triaposhi1})$ to $(\ref{Triaposhi2})$.
Using these {\rm (i),(ii)}, we define the following.
\begin{enumerate}
\item For each object $Z\in\mathrm{Ob}(\mathcal{Z}/\mathcal{I})$, choose $Z\langle1\rangle=U$ as in {\rm (i)}.
\item For any morphism $\underline{z}$ in $\mathcal{Z}/\mathcal{I}$, define the morphism $\underline{z}\langle1\rangle$ by $\underline{z}\langle1\rangle=\underline{u}$ as in {\rm (ii)}.
\end{enumerate}
Then we obtain a fully faithful functor $\langle1\rangle\colon\mathcal{Z}/\mathcal{I}\to\mathcal{U}/\mathcal{I}$, determined uniquely up to natural isomorphism.

The functor $\langle-1\rangle\colon\mathcal{Z}/\mathcal{I}\to\mathcal{T}/\mathcal{I}$ is defined dually.
Moreover, the same argument as in {\rm (ii)} shows the existence of a natural bijection
\begin{equation}\label{Eq7.3}
(\mathcal{U}/\mathcal{I})(Z\langle1\rangle,Z^{\prime})\cong(\mathcal{T}/\mathcal{I})(Z,Z^{\prime}\langle-1\rangle)
\end{equation}
for any $Z,Z^{\prime}\in\mathcal{Z}$.
\end{dfn}

\begin{dfn}\label{DefShifts}
Define as $\Sigma=\sigma\circ\langle1\rangle$, and $\Omega=\omega\circ\langle-1\rangle$.
\end{dfn}

\begin{prop}\label{PropShiftAdj}
For any concentric $\mathcal{P}$, we obtain an adjoint pair $\Sigma\dashv\Omega$.
\end{prop}
\begin{proof}
This follows from the adjointness of each of $\sigma,\omega$, and the bijection $(\ref{Eq7.3})$.
\end{proof}

\begin{dfn}\label{DefUTria}
Let $\mathcal{P}$ be a concentric TCP. We call a distinguished triangle in $\mathcal{C}$
\begin{equation}\label{UConicTria}
X\overset{f}{\longrightarrow}Y\overset{a}{\longrightarrow}U\overset{b}{\longrightarrow}X[1]
\end{equation}
a {\it $\mathcal{U}$-conic triangle} if it satisfies $X,Y\in\mathcal{Z}$ and $U\in\mathcal{U}$.
\end{dfn}

\begin{dfn}\label{DefUStan}
Let $\mathcal{P}$ be a concentric TCP.
To any $\mathcal{U}$-conic triangle $(\ref{UConicTria})$, we associate {\it standard right triangle}
\begin{equation}\label{UStanTria}
X\overset{\underline{f}}{\longrightarrow}Y\overset{\eta_U\circ\underline{a}}{\longrightarrow}\sigma(U)\overset{\sigma(\underline{c})}{\longrightarrow}\Sigma X
\end{equation}
in $\mathcal{Z}/\mathcal{I}$, as follows. (Here, $\eta$ is the unit for the adjoint, as in Definition \ref{DefTCPAdj} {\rm (iv)}.)
\begin{itemize}
\item[{\rm (i)}] Take a distinguished triangle in $\mathcal{C}$
\begin{equation}\label{UStT1}
X\overset{\iota_X}{\longrightarrow}I_X\overset{\wp_X}{\longrightarrow}U_X\overset{\gamma_X}{\longrightarrow}X[1]
\end{equation}
satisfying $I_X\in\mathcal{I},U_X\in\mathcal{U}$.
\item[{\rm (ii)}] By $\mathcal{C}(U[-1],I_X)=0$, there is a morphism of triangles in $\mathcal{C}$
\begin{equation}\label{UStT2}
\xy
(-19,6)*+{X}="0";
(-6,6)*+{Y}="2";
(7,6)*+{U}="4";
(21,6)*+{X[1]}="6";
(-19,-6)*+{X}="10";
(-6,-6)*+{I_X}="12";
(7,-6)*+{U_X}="14";
(21,-6)*+{X[1]}="16";
{\ar^{f} "0";"2"};
{\ar^{a} "2";"4"};
{\ar^{b} "4";"6"};
{\ar@{=} "0";"10"};
{\ar^{y} "2";"12"};
{\ar^{c} "4";"14"};
{\ar@{=} "6";"16"};
{\ar_{\iota_X} "10";"12"};
{\ar_{\wp_X} "12";"14"};
{\ar_{\gamma_X} "14";"16"};
{\ar@{}|\circlearrowright "0";"12"};
{\ar@{}|\circlearrowright "2";"14"};
{\ar@{}|\circlearrowright "4";"16"};
\endxy.
\end{equation}
\end{itemize}
This gives the morphism $c$ appearing in $(\ref{UStanTria})$.
\end{dfn}

\begin{claim}
For any $\mathcal{U}$-conic triangle $(\ref{UConicTria})$, its associated standard right triangle $(\ref{UStanTria})$ is uniquely determined up to isomorphism, independently of the choices of $(\ref{UStT1}),(\ref{UStT2})$.
\end{claim}
\begin{proof}
For a fixed $(\ref{UStT1})$, if
\[
\xy
(-19,6)*+{X}="0";
(-6,6)*+{Y}="2";
(7,6)*+{U}="4";
(21,6)*+{X[1]}="6";
(-19,-6)*+{X}="10";
(-6,-6)*+{I_X}="12";
(7,-6)*+{U_X}="14";
(21,-6)*+{X[1]}="16";
{\ar^{f} "0";"2"};
{\ar^{a} "2";"4"};
{\ar^{b} "4";"6"};
{\ar@{=} "0";"10"};
{\ar^{y^{\prime}} "2";"12"};
{\ar^{c^{\prime}} "4";"14"};
{\ar@{=} "6";"16"};
{\ar_{\iota_X} "10";"12"};
{\ar_{\wp_X} "12";"14"};
{\ar_{\gamma_X} "14";"16"};
{\ar@{}|\circlearrowright "0";"12"};
{\ar@{}|\circlearrowright "2";"14"};
{\ar@{}|\circlearrowright "4";"16"};
\endxy
\]
is another morphism of triangles, then $\gamma_X\circ(c-c^{\prime})=b-b=0$ implies that $c-c^{\prime}$ factors through $\wp_X$, and thus $\underline{c}=\underline{c}^{\prime}$. Thus for a fixed $(\ref{UStT1})$, the resulting right triangle is unique up to isomorphism.

Remark that, $\Sigma X=\sigma(U_X)$ is unique up to isomorphism, independently of the choice of $(\ref{UStanTria})$. More precisely, for any other distinguished triangle
\[ X\overset{\iota^{\prime}_X}{\longrightarrow}I^{\prime}_X\overset{\wp^{\prime}_X}{\longrightarrow}U^{\prime}_X\overset{\gamma^{\prime}_X}{\longrightarrow}X[1]\quad(I_X^{\prime}\in\mathcal{I},\, U_X^{\prime}\in\mathcal{U}), \]
there exists a morphism of triangles
\[
\xy
(-19,6)*+{X}="0";
(-6,6)*+{I_X}="2";
(7,6)*+{U_X}="4";
(21,6)*+{X[1]}="6";
(-19,-6)*+{X}="10";
(-6,-6)*+{I^{\prime}_X}="12";
(7,-6)*+{U^{\prime}_X}="14";
(21,-6)*+{X[1]}="16";
{\ar^{\iota_X} "0";"2"};
{\ar^{\wp_X} "2";"4"};
{\ar^{\gamma_X} "4";"6"};
{\ar@{=} "0";"10"};
{\ar^{i} "2";"12"};
{\ar^{u} "4";"14"};
{\ar@{=} "6";"16"};
{\ar_{\iota^{\prime}_X} "10";"12"};
{\ar_{\wp^{\prime}_X} "12";"14"};
{\ar_{\gamma^{\prime}_X} "14";"16"};
{\ar@{}|\circlearrowright "0";"12"};
{\ar@{}|\circlearrowright "2";"14"};
{\ar@{}|\circlearrowright "4";"16"};
\endxy
\]
which gives isomorphism $\underline{u}\in(\mathcal{U}/\mathcal{I})(U_X,U_X^{\prime})$.
Composing $(\ref{UStT2})$ with this, we obtain
\[
\xy
(-19,6)*+{X}="0";
(-6,6)*+{Y}="2";
(7,6)*+{U}="4";
(21,6)*+{X[1]}="6";
(-19,-6)*+{X}="10";
(-6,-6)*+{I^{\prime}_X}="12";
(7,-6)*+{U^{\prime}_X}="14";
(21,-6)*+{X[1]}="16";
{\ar^{f} "0";"2"};
{\ar^{a} "2";"4"};
{\ar^{b} "4";"6"};
{\ar@{=} "0";"10"};
{\ar^{i\circ y} "2";"12"};
{\ar^{u\circ c} "4";"14"};
{\ar@{=} "6";"16"};
{\ar_{\iota^{\prime}_X} "10";"12"};
{\ar_{\wp^{\prime}_X} "12";"14"};
{\ar_{\gamma^{\prime}_X} "14";"16"};
{\ar@{}|\circlearrowright "0";"12"};
{\ar@{}|\circlearrowright "2";"14"};
{\ar@{}|\circlearrowright "4";"16"};
\endxy.
\]
This induces an isomorphism of right triangles
\[
\xy
(-21,6)*+{X}="0";
(-8,6)*+{Y}="2";
(8,6)*+{\sigma(U)}="4";
(24,6)*+{\sigma(U_X)}="6";
(-21,-6)*+{X}="10";
(-8,-6)*+{Y}="12";
(8,-6)*+{\sigma(U)}="14";
(24,-6)*+{\sigma(U^{\prime}_X)}="16";
{\ar^{\underline{f}} "0";"2"};
{\ar^(0.46){\eta_U\circ\underline{a}} "2";"4"};
{\ar^{\sigma(\underline{c})} "4";"6"};
{\ar@{=} "0";"10"};
{\ar@{=} "2";"12"};
{\ar@{=} "4";"14"};
{\ar_{\cong}^{\sigma(\underline{u})} "6";"16"};
{\ar_{\underline{f}} "10";"12"};
{\ar_(0.46){\eta_U\circ\underline{a}} "12";"14"};
{\ar_{\sigma(\underline{u}\circ\underline{c})} "14";"16"};
{\ar@{}|\circlearrowright "0";"12"};
{\ar@{}|\circlearrowright "2";"14"};
{\ar@{}|\circlearrowright "4";"16"};
\endxy.
\]
\end{proof}

\begin{prop}\label{PropMorphStan}
For $i=1,2$, let $X_i\overset{f_i}{\longrightarrow}Y_i\overset{a_i}{\longrightarrow}U_i\overset{b_i}{\longrightarrow}X_i[1]$ be $\mathcal{U}$-conic triangles. For any morphism of triangles in $\mathcal{C}$
\[
\xy
(-19,6)*+{X_1}="0";
(-6,6)*+{Y_1}="2";
(7,6)*+{U_1}="4";
(21,6)*+{X_1[1]}="6";
(-19,-6)*+{X_2}="10";
(-6,-6)*+{Y_2}="12";
(7,-6)*+{U_2}="14";
(21,-6)*+{X_2[1]}="16";
{\ar^{f_1} "0";"2"};
{\ar^{a_1} "2";"4"};
{\ar^{b_1} "4";"6"};
{\ar_{x} "0";"10"};
{\ar^{y} "2";"12"};
{\ar^{u} "4";"14"};
{\ar^{x[1]} "6";"16"};
{\ar_{f_2} "10";"12"};
{\ar_{a_2} "12";"14"};
{\ar_{b_2} "14";"16"};
{\ar@{}|\circlearrowright "0";"12"};
{\ar@{}|\circlearrowright "2";"14"};
{\ar@{}|\circlearrowright "4";"16"};
\endxy,
\]
we have the following morphism between standard right triangles.
\begin{equation}\label{TTT}
\xy
(-21,6)*+{X_1}="0";
(-8,6)*+{Y_1}="2";
(9,6)*+{\sigma(U_1)}="4";
(25,6)*+{\Sigma X_1}="6";
(-21,-6)*+{X_2}="10";
(-8,-6)*+{Y_2}="12";
(9,-6)*+{\sigma(U_2)}="14";
(25,-6)*+{\Sigma X_2}="16";
{\ar^{\underline{f}_1} "0";"2"};
{\ar^(0.46){\eta_{U_1}\circ\underline{a}_1} "2";"4"};
{\ar^{\sigma(\underline{c}_1)} "4";"6"};
{\ar_{\underline{x}} "0";"10"};
{\ar^{\underline{y}} "2";"12"};
{\ar^{\sigma(\underline{u})} "4";"14"};
{\ar^{\Sigma\underline{x}} "6";"16"};
{\ar_{\underline{f}_2} "10";"12"};
{\ar_(0.46){\eta_{U_2}\circ\underline{a}_2} "12";"14"};
{\ar_{\sigma(\underline{c}_2)} "14";"16"};
{\ar@{}|\circlearrowright "0";"12"};
{\ar@{}|\circlearrowright "2";"14"};
{\ar@{}|\circlearrowright "4";"16"};
\endxy
\end{equation}
\end{prop}
\begin{proof}
For each $i=1,2$, let $X_i\overset{\iota_{X_i}}{\longrightarrow}I_{X_i}\overset{\wp_{X_i}}{\longrightarrow}U_{X_i}\overset{\gamma_{X_i}}{\longrightarrow}X_i[1]$ $(i=1,2)$ be a distinguished triangle in $\mathcal{C}$ satisfying $I_{X_i}\in\mathcal{I},\, U_{X_i}\in\mathcal{U}$, and let
\begin{equation}\label{c_i}
\xy
(-19,6)*+{X_i}="0";
(-6,6)*+{Y_i}="2";
(7,6)*+{U_i}="4";
(21,6)*+{X_i[1]}="6";
(-19,-6)*+{X_i}="10";
(-6,-6)*+{I_{X_i}}="12";
(7,-6)*+{U_{X_i}}="14";
(21,-6)*+{X_i[1]}="16";
{\ar^{f_i} "0";"2"};
{\ar^{a_i} "2";"4"};
{\ar^{b_i} "4";"6"};
{\ar@{=} "0";"10"};
{\ar^{y_i} "2";"12"};
{\ar^{c_i} "4";"14"};
{\ar@{=} "6";"16"};
{\ar_{\iota_{X_i}} "10";"12"};
{\ar_{\wp_{X_i}} "12";"14"};
{\ar_{\gamma_{X_i}} "14";"16"};
{\ar@{}|\circlearrowright "0";"12"};
{\ar@{}|\circlearrowright "2";"14"};
{\ar@{}|\circlearrowright "4";"16"};
\endxy
\end{equation}
be a morphism of triangles.
It suffices to show the commutativity of the right square of $(\ref{TTT})$
\begin{equation}\label{RightTTT}
\xy
(-8,6)*+{\sigma(U_1)}="0";
(8,6)*+{\Sigma X_1}="2";
(-8,-6)*+{\sigma(U_2)}="4";
(8,-6)*+{\Sigma X_2}="6";
{\ar^{\sigma(\underline{c}_1)} "0";"2"};
{\ar_{\sigma(\underline{u})} "0";"4"};
{\ar^{\Sigma\underline{x}} "2";"6"};
{\ar_{\sigma(\underline{c}_2)} "4";"6"};
{\ar@{}|\circlearrowright "0";"6"};
\endxy.
\end{equation}
Take a morphism of triangles
\[
\xy
(-19,6)*+{X_1}="0";
(-6,6)*+{I_{X_1}}="2";
(7,6)*+{U_{X_1}}="4";
(21,6)*+{X_1[1]}="6";
(-19,-6)*+{X_2}="10";
(-6,-6)*+{I_{X_2}}="12";
(7,-6)*+{U_{X_2}}="14";
(21,-6)*+{X_2[1]}="16";
{\ar^{\iota_{X_1}} "0";"2"};
{\ar^{\wp_{X_1}} "2";"4"};
{\ar^{\gamma_{X_1}} "4";"6"};
{\ar_{x} "0";"10"};
{\ar^{} "2";"12"};
{\ar^{s} "4";"14"};
{\ar^{x[1]} "6";"16"};
{\ar_{\iota_{X_2}} "10";"12"};
{\ar_{\wp_{X_2}} "12";"14"};
{\ar_{\gamma_{X_2}} "14";"16"};
{\ar@{}|\circlearrowright "0";"12"};
{\ar@{}|\circlearrowright "2";"14"};
{\ar@{}|\circlearrowright "4";"16"};
\endxy
\]
which gives $\underline{s}=\underline{x}\langle1\rangle$, and thus $\sigma(\underline{s})=\Sigma\underline{x}$. Then we have
\begin{eqnarray*}
\gamma_{X_2}\circ(s\circ c_1-c_2\circ u)&=&x[1]\circ\gamma_{X_1}\circ c_1-\gamma_{X_2}\circ c_2\circ u\\
&=&x[1]\circ b_1-b_2\circ u\ =\ 0.
\end{eqnarray*}
Thus $s\circ c_1-c_2\circ u$ factors through $\wp_{X_2}$, which means $\underline{s}\circ\underline{c}_1=\underline{c}_2\circ\underline{u}$, i.e., the following diagram is commutative in $\mathcal{U}/\mathcal{I}$.
\[
\xy
(-7,6)*+{U_1}="0";
(7,6)*+{U_{X_1}}="2";
(-7,-6)*+{U_2}="4";
(7,-6)*+{U_{X_2}}="6";
{\ar^{\underline{c}_1} "0";"2"};
{\ar_{\underline{u}} "0";"4"};
{\ar^{\underline{s}} "2";"6"};
{\ar_{\underline{c}_2} "4";"6"};
{\ar@{}|\circlearrowright "0";"6"};
\endxy.
\]
Applying $\sigma$, we obtain the commutative square $(\ref{RightTTT})$.
\end{proof}

\begin{lem}\label{LemInvariantStan}
Let $\mathcal{P}$ be a concentric TCP. Let $f\in\mathcal{Z}(X,Y)$ be any morphism, and let $X\overset{\iota_X}{\longrightarrow}I_X\overset{\wp_X}{\longrightarrow}U_X\overset{\gamma_X}{\longrightarrow}X[1]$ be a distinguished triangle as in $(\ref{UStT1})$. Then, we have the following.
\begin{enumerate}
\item If we complete $\begin{bmatrix}f\\ \iota_X\end{bmatrix}\in\mathcal{C}(X,Y\oplus I_X)$ into a distinguished triangle in $\mathcal{C}$
\begin{equation}\label{CanonUTria}
X\overset{\scriptstyle\begin{bmatrix}f\\ \iota_X\end{bmatrix}}{\longrightarrow}Y\oplus I_X\overset{a_f}{\longrightarrow}C_f\overset{b_f}{\longrightarrow}X[1],
\end{equation}
then it is a $\mathcal{U}$-conic triangle. Remark that, the associated standard right triangle
\begin{equation}\label{AssocRightCanon}
X\overset{\begin{bmatrix}\underline{f}\\ \underline{\iota_X}\end{bmatrix}}{\longrightarrow}Y\oplus I_X\overset{\eta_{C_f}\circ\underline{a}_f}{\longrightarrow}\sigma(C_f)\overset{\sigma(\underline{c}_f)}{\longrightarrow}\Sigma X
\end{equation}
is isomorphic to
\[
X\overset{\underline{f}}{\longrightarrow}Y\overset{\eta_{C_f}\circ\underline{d}_f}{\longrightarrow}\sigma(C_f)\overset{\sigma(\underline{c}_f)}{\longrightarrow}\Sigma X
\]
as right triangles in $\mathcal{Z}/\mathcal{I}$, where we express $a_f$ as $a_f=[d_f\ e_f]\colon Y\oplus I_X\to C_f$. Here $c_f$ denotes the morphism obtained in diagram $(\ref{ShiftedOcta})$ in the proof below.
\item Let $X\overset{f^{\prime}}{\longrightarrow}Y\overset{a^{\prime}}{\longrightarrow}U^{\prime}\overset{b^{\prime}}{\longrightarrow}X[1]$ be any $\mathcal{U}$-conic triangle satisfying $\underline{f}=\underline{f}^{\prime}$, and let
\[ X\overset{\underline{f}^{\prime}}{\longrightarrow}Y\overset{\eta_{U^{\prime}}\circ\underline{a}^{\prime}}{\longrightarrow}\sigma(U^{\prime})\overset{\sigma(\underline{c}^{\prime})}{\longrightarrow}\Sigma X \]
be the associated standard right triangle. Then, it is isomorphic to $(\ref{AssocRightCanon})$ as right triangles in $\mathcal{Z}/\mathcal{I}$.
\end{enumerate}
\end{lem}
\begin{proof}
{\rm (1)} By the octahedron axiom, we have the following commutative diagram made of distinguished triangles
\begin{equation}\label{ShiftedOcta}
\xy
(-8,23)*+{Y}="-12";
(8,23)*+{Y}="-14";
(-24,8)*+{X}="0";
(-26,8)*+{}="1";
(-8,8)*+{Y\oplus I_X}="2";
(8,8)*+{C_f}="4";
(24,8)*+{X[1]}="6";
(-24,-7)*+{X}="10";
(-8,-7)*+{I_X}="12";
(8,-7)*+{U_X}="14";
(24,-7)*+{X[1]}="16";
(24,-20)*+{}="17";
(-8,-22)*+{Y[1]}="22";
(8,-22)*+{Y[1]}="24";
(19,-22)*+{}="25";
(28,-22)*+{Y[1]\oplus I_X[1]}="28";
{\ar@{=} "-12";"-14"};
{\ar_{\spmatrix{1}{0}} "-12";"2"};
{\ar^{d_f} "-14";"4"};
{\ar^(0.42){\spmatrix{f}{\iota_X}} "0";"2"};
{\ar^{a_f} "2";"4"};
{\ar^{b_f} "4";"6"};
{\ar@{=} "0";"10"};
{\ar_{[0\ 1]} "2";"12"};
{\ar^{{}^{\exists}c_f} "4";"14"};
{\ar@{=} "6";"16"};
{\ar_{\iota_X} "10";"12"};
{\ar_{\wp_X} "12";"14"};
{\ar_{\gamma_X} "14";"16"};
{\ar_{0} "12";"22"};
{\ar^{{}^{\exists}q} "14";"24"};
{\ar^(0.56){\spmatrix{f[1]}{\iota_X[1]}} "16";"17"};
{\ar@{=} "22";"24"};
{\ar_(0.56){-\spmatrix{1}{0}} "24";"25"};
{\ar@{}|\circlearrowright "-12";"4"};
{\ar@{}|\circlearrowright "1";"12"};
{\ar@{}|\circlearrowright "2";"14"};
{\ar@{}|\circlearrowright "4";"16"};
{\ar@{}|\circlearrowright "12";"24"};
{\ar@{}|\circlearrowright "16";"24"};
\endxy,
\end{equation}
which shows $C_f\in\mathcal{U}$.

{\rm (2)} By $\underline{f}=\underline{f}^{\prime}$, their difference $f-f^{\prime}$ factors through $\iota_X$. Namely, there exists $j^{\prime}\in\mathcal{C}(I_X,Y)$ satisfying $f-f^{\prime}=j^{\prime}\circ\iota_X$. Thus the following diagram is commutative in $\mathcal{C}$.
\[
\xy
(-9,7)*+{X}="0";
(9,7)*+{Y\oplus I_X}="2";
(-9,-7)*+{X}="4";
(9,-7)*+{Y}="6";
{\ar^{\spmatrix{f}{\iota_X}} "0";"2"};
{\ar@{=} "0";"4"};
{\ar^{[1\ j^{\prime}]} "2";"6"};
{\ar_{f^{\prime}} "4";"6"};
{\ar@{}|\circlearrowright "0";"6"};
\endxy.
\]
Remark that, 
\[ I_X\overset{\spmatrix{-j^{\prime}}{1}}{\longrightarrow}Y\oplus I_X\overset{[1\ j^{\prime}]}{\longrightarrow}Y\overset{0}{\longrightarrow}I_X[1] \]
is a distinguished triangle in $\mathcal{C}$. 
By the octahedron axiom, we have
\[
\xy
(-8,23)*+{I_X}="-12";
(8,23)*+{I_X}="-14";
(-24,8)*+{X}="0";
(-26,8)*+{}="1";
(-8,8)*+{Y\oplus I_X}="2";
(8,8)*+{C_f}="4";
(24,8)*+{X[1]}="6";
(-24,-7)*+{X}="10";
(-8,-7)*+{Y}="12";
(8,-7)*+{U^{\prime}}="14";
(24,-7)*+{X[1]}="16";
(-8,-22)*+{I_X[1]}="22";
(8,-22)*+{I_X[1]}="24";
{\ar@{=} "-12";"-14"};
{\ar_{} "-12";"2"};
{\ar^{} "-14";"4"};
{\ar^(0.4){\spmatrix{f}{\iota_X}} "0";"2"};
{\ar^{a_f} "2";"4"};
{\ar^{b_f} "4";"6"};
{\ar@{=} "0";"10"};
{\ar_{[1\ j^{\prime}]} "2";"12"};
{\ar^{{}^{\exists}u^{\prime}} "4";"14"};
{\ar@{=} "6";"16"};
{\ar_{f^{\prime}} "10";"12"};
{\ar^{a^{\prime}} "12";"14"};
{\ar^{b^{\prime}} "14";"16"};
{\ar_{0} "12";"22"};
{\ar^{0} "14";"24"};
{\ar@{=} "22";"24"};
{\ar@{}|\circlearrowright "-12";"4"};
{\ar@{}|\circlearrowright "1";"12"};
{\ar@{}|\circlearrowright "2";"14"};
{\ar@{}|\circlearrowright "4";"16"};
{\ar@{}|\circlearrowright "12";"24"};
\endxy.
\]
By Lemma \ref{LemNUU}, $\sigma(\underline{u}^{\prime})$ is isomorphism.
Thus it gives an isomorphism of associated standard right triangles
\[
\xy
(-25,7)*+{X}="0";
(-8,7)*+{Y\oplus I_X}="2";
(8,7)*+{\sigma(C_f)}="4";
(25,7)*+{\Sigma X}="6";
(-25,-7)*+{X}="10";
(-8,-7)*+{Y}="12";
(8,-7)*+{\sigma(U^{\prime})}="14";
(11,-7)*+{}="15";
(25,-7)*+{\Sigma X}="16";
(28,-7)*+{}="17";
{\ar^{\spmatrix{\underline{f}}{\underline{\iota}_X}} "0";"2"};
{\ar^{} "2";"4"};
{\ar^{} "4";"6"};
{\ar@{=} "0";"10"};
{\ar^{[1\ \underline{j}^{\prime}]}_{\cong} "2";"12"};
{\ar^{\sigma(\underline{u}^{\prime})}_{\cong} "4";"14"};
{\ar@{=} "6";"16"};
{\ar_{\underline{f}^{\prime}} "10";"12"};
{\ar_{} "12";"14"};
{\ar_{} "14";"16"};
{\ar@{}|\circlearrowright "0";"12"};
{\ar@{}|\circlearrowright "2";"15"};
{\ar@{}|\circlearrowright "4";"17"};
\endxy
\]
by Proposition \ref{PropMorphStan}.
\end{proof}

\begin{prop}\label{PropUniqueRightComplete}
Let $\mathcal{P}$ be a concentric TCP. Let $X,Y\in\mathcal{Z}$ be any pair of objects. Let
\begin{eqnarray*}
&X\overset{f}{\longrightarrow}Y\overset{a}{\longrightarrow}U\overset{b}{\longrightarrow}X[1]&\\
&X\overset{f^{\prime}}{\longrightarrow}Y\overset{a^{\prime}}{\longrightarrow}U^{\prime}\overset{b^{\prime}}{\longrightarrow}X[1]&
\end{eqnarray*}
be $\mathcal{U}$-conic triangles, satisfying $\underline{f}=\underline{f}^{\prime}$. Then there exists a morphism $\delta\in(\mathcal{Z}/\mathcal{I})(\sigma(U),\sigma(U^{\prime}))$ which gives an isomorphism of right triangles
\begin{equation}\label{RightTriaIsomXY}
\xy
(-26,7)*+{X}="0";
(-9,7)*+{Y}="2";
(9,7)*+{\sigma(U)}="4";
(26,7)*+{\Sigma X}="6";
(-26,-7)*+{X}="10";
(-9,-7)*+{Y}="12";
(9,-7)*+{\sigma(U^{\prime})}="14";
(26,-7)*+{\Sigma X}="16";
{\ar^{\underline{f}} "0";"2"};
{\ar^(0.46){\eta_U\circ\underline{a}} "2";"4"};
{\ar^{\sigma(\underline{c})} "4";"6"};
{\ar@{=} "0";"10"};
{\ar@{=} "2";"12"};
{\ar^{\delta}_{\cong} "4";"14"};
{\ar@{=} "6";"16"};
{\ar_{\underline{f}^{\prime}} "10";"12"};
{\ar_(0.46){\eta_{U^{\prime}}\circ\underline{a}^{\prime}} "12";"14"};
{\ar_{\sigma(\underline{c}^{\prime})} "14";"16"};
{\ar@{}|\circlearrowright "0";"12"};
{\ar@{}|\circlearrowright "2";"14"};
{\ar@{}|\circlearrowright "4";"16"};
\endxy
\end{equation}
in $\mathcal{Z}/\mathcal{I}$.
\end{prop}
\begin{proof}
Let $X\overset{\iota_X}{\longrightarrow}I_X\overset{\wp_X}{\longrightarrow}U_X\overset{\gamma_X}{\longrightarrow}X[1]$ be a distinguished right triangle as in $(\ref{UStT1})$. Applying Lemma \ref{LemInvariantStan} {\rm (1)} to $f$, we obtain a $\mathcal{U}$-conic triangle $(\ref{CanonUTria})$.
By Lemma \ref{LemInvariantStan} {\rm (2)}, we have isomorphisms of right triangles
\[
\xy
(-30,7)*+{X}="0";
(-11,7)*+{Y\oplus I_X}="2";
(11,7)*+{\sigma(C_f)}="4";
(31,7)*+{\Sigma X}="6";
(-30,-7)*+{X}="10";
(-11,-7)*+{Y}="12";
(11,-7)*+{\sigma(U)}="14";
(31,-7)*+{\Sigma X}="16";
{\ar^{\spmatrix{\underline{f}}{\underline{\iota}_X}} "0";"2"};
{\ar^{\eta_{C_f}\circ\underline{a}_f} "2";"4"};
{\ar^{\sigma(\underline{c}_f)} "4";"6"};
{\ar@{=} "0";"10"};
{\ar^{[1\ \underline{j}]}_{\cong} "2";"12"};
{\ar^{\sigma(\underline{u})}_{\cong} "4";"14"};
{\ar@{=} "6";"16"};
{\ar_{\underline{f}} "10";"12"};
{\ar_{\eta_U\circ\underline{a}} "12";"14"};
{\ar_{\sigma(\underline{c})} "14";"16"};
{\ar@{}|\circlearrowright "0";"12"};
{\ar@{}|\circlearrowright "2";"14"};
{\ar@{}|\circlearrowright "4";"16"};
\endxy
\]
and
\[
\xy
(-30,7)*+{X}="0";
(-11,7)*+{Y\oplus I_X}="2";
(11,7)*+{\sigma(C_f)}="4";
(31,7)*+{\Sigma X}="6";
(-30,-7)*+{X}="10";
(-11,-7)*+{Y}="12";
(11,-7)*+{\sigma(U^{\prime})}="14";
(31,-7)*+{\Sigma X}="16";
{\ar^{\spmatrix{\underline{f}}{\underline{\iota}_X}} "0";"2"};
{\ar^{\eta_{C_f}\circ\underline{a}_f} "2";"4"};
{\ar^{\sigma(\underline{c}_f)} "4";"6"};
{\ar@{=} "0";"10"};
{\ar^{[1\ \underline{j}^{\prime}]}_{\cong} "2";"12"};
{\ar^{\sigma(\underline{u}^{\prime})}_{\cong} "4";"14"};
{\ar@{=} "6";"16"};
{\ar_{\underline{f}^{\prime}} "10";"12"};
{\ar_{\eta_{U^{\prime}}\circ\underline{a}^{\prime}} "12";"14"};
{\ar_{\sigma(\underline{c}^{\prime})} "14";"16"};
{\ar@{}|\circlearrowright "0";"12"};
{\ar@{}|\circlearrowright "2";"14"};
{\ar@{}|\circlearrowright "4";"16"};
\endxy.
\]
Since $[1\ \underline{j}^{\prime}]\circ [1\ \underline{j}]^{-1}=\mathrm{id}_Y$, if we put $\delta=\sigma(\underline{u}^{\prime})\circ\sigma(\underline{u})^{-1}$, then it gives an isomorphism of right triangles $(\ref{RightTriaIsomXY})$.
\end{proof}

\begin{dfn}\label{DefDistRight}
Let $\mathcal{P}$ be a concentric TCP. We define a {\it distinguished right triangle} to be a right triangle isomorphic to the standard one associated to some $\mathcal{U}$-conic triangle in $\mathcal{C}$.
\end{dfn}

\begin{lem}\label{LemUUUTria}
Let $\mathcal{P}$ be a concentric TCP. Let
\[ U_1\overset{\ell}{\longrightarrow}U_2\overset{m}{\longrightarrow}U_3\overset{n}{\longrightarrow}U_1[1] \]
be a distinguished triangle in $\mathcal{C}$ satisfying $U_1,U_2,U_3\in\mathcal{U}$. Then there exist distinguished triangles in $\mathcal{C}$
\begin{eqnarray*}
&S_1[-1]\to U_1\overset{z_1}{\longrightarrow}Z_1\to S_1,&\\
&S_2[-1]\to U_2\overset{z_2}{\longrightarrow}Z_2\to S_2,&\\
&S_3[-1]\to U_3\overset{u_3}{\longrightarrow}U_3^{\prime}\to S_3,&\\
&Z_1\overset{x}{\longrightarrow}Z_2\overset{y}{\longrightarrow}U_3^{\prime}\overset{z}{\longrightarrow}Z_1[1]&
\end{eqnarray*}
satisfying $S_i\in\mathcal{S}\, (i=1,2,3)$, $Z_1,Z_2\in\mathcal{Z}\, (i=1,2)$ and $U_3^{\prime}\in\mathcal{U}$, which give the following morphism of triangles in $\mathcal{C}$.
\begin{equation}\label{MorphUUUTria}
\xy
(-19,6)*+{U_1}="0";
(-6,6)*+{U_2}="2";
(7,6)*+{U_3}="4";
(21,6)*+{U_1[1]}="6";
(-19,-6)*+{Z_1}="10";
(-6,-6)*+{Z_2}="12";
(7,-6)*+{U_3^{\prime}}="14";
(21,-6)*+{Z_1[1]}="16";
{\ar^{\ell} "0";"2"};
{\ar^{m} "2";"4"};
{\ar^{n} "4";"6"};
{\ar_{z_1} "0";"10"};
{\ar^{z_2} "2";"12"};
{\ar^{u_3} "4";"14"};
{\ar^{z_1[1]} "6";"16"};
{\ar_{x} "10";"12"};
{\ar_{y} "12";"14"};
{\ar_{z} "14";"16"};
{\ar@{}|\circlearrowright "0";"12"};
{\ar@{}|\circlearrowright "2";"14"};
{\ar@{}|\circlearrowright "4";"16"};
\endxy
\end{equation}

Thus it gives a distinguished right triangle
\begin{equation}\label{RightTria_m}
\sigma(U_1)\overset{\sigma(\underline{\ell})}{\longrightarrow}\sigma(U_2)\overset{\sigma(\eta_{U_3}\circ\underline{m})}{\longrightarrow}\sigma(U_3)\to\Sigma(\sigma U_1)
\end{equation}
in $\mathcal{Z}/\mathcal{I}$.
\end{lem}
\begin{proof}
Decompose $U_1$ into a distinguished triangle
\[ S_1[-1]\overset{s_1}{\longrightarrow}U_1\overset{z_1}{\longrightarrow}Z_1\to S_1\quad(S_1\in\mathcal{S},Z_1\in\mathcal{Z}) \]
in $\mathcal{C}$. By the octahedron axiom, we obtain a commutative diagram
\[
\xy
(-27,23)*+{S_1[-1]}="-12";
(-9,23)*+{S_1[-1]}="-14";
(-27,8)*+{U_1}="2";
(-9,8)*+{U_2}="4";
(9,8)*+{U_3}="6";
(27,8)*+{U_1[1]}="8";
(-27,-7)*+{Z_1}="12";
(-9,-7)*+{{}^{\exists}U_2^{\prime}}="14";
(9,-7)*+{U_3}="16";
(27,-7)*+{Z_1[1]}="18";
{\ar@{=} "-12";"-14"};
{\ar_{s_1} "-12";"2"};
{\ar^{} "-14";"4"};
{\ar^{\ell} "2";"4"};
{\ar^{m} "4";"6"};
{\ar^{n} "6";"8"};
{\ar_{z_1} "2";"12"};
{\ar^{{}^{\exists}u_2} "4";"14"};
{\ar@{=} "6";"16"};
{\ar^{z_1[1]} "8";"18"};
{\ar_{{}^{\exists}x^{\prime}} "12";"14"};
{\ar_{{}^{\exists}y^{\prime}} "14";"16"};
{\ar_{{}^{\exists}k} "16";"18"};
{\ar@{}|\circlearrowright "-12";"4"};
{\ar@{}|\circlearrowright "2";"14"};
{\ar@{}|\circlearrowright "4";"16"};
{\ar@{}|\circlearrowright "6";"18"};
\endxy
\]
where
\[ S_1[-1]\to U_2\overset{u_2}{\longrightarrow}U_2^{\prime}\to S_1,\quad Z_1\overset{x^{\prime}}{\longrightarrow}U_2^{\prime}\overset{y^{\prime}}{\longrightarrow}U_3\overset{k}{\longrightarrow}Z_1[1] \]
are distinguished triangles. It follows $U_2^{\prime}\in\mathcal{U}$.

Decompose $U_2^{\prime}$ into a distinguished triangle
\[ S_3[-1]\to U_2^{\prime}\overset{z_2^{\prime}}{\longrightarrow}Z_2\to S_3\quad(S_3\in\mathcal{S},Z_2\in\mathcal{Z}) \]
in $\mathcal{C}$. Put $x=z_2^{\prime}\circ x^{\prime}$, and complete it into a distinguished triangle
\[ Z_1\overset{x}{\longrightarrow}Z_2\overset{y}{\longrightarrow}U_3^{\prime}\overset{z}{\longrightarrow}Z_1[1]. \]
By the octahedron axiom, we obtain
\[
\xy
(-8,23)*+{S_3[-1]}="-12";
(8,23)*+{S_3[-1]}="-14";
(-24,8)*+{Z_1}="0";
(-8,8)*+{U_2^{\prime}}="2";
(8,8)*+{U_3}="4";
(24,8)*+{Z_1[1]}="6";
(-24,-7)*+{Z_1}="10";
(-8,-7)*+{Z_2}="12";
(8,-7)*+{U_3^{\prime}}="14";
(24,-7)*+{Z_1[1]}="16";
(-8,-22)*+{S_3}="22";
(8,-22)*+{S_3}="24";
{\ar@{=} "-12";"-14"};
{\ar_{} "-12";"2"};
{\ar^{} "-14";"4"};
{\ar^{x^{\prime}} "0";"2"};
{\ar^{y^{\prime}} "2";"4"};
{\ar^{k} "4";"6"};
{\ar@{=} "0";"10"};
{\ar_{z_2^{\prime}} "2";"12"};
{\ar^{{}^{\exists}u_3} "4";"14"};
{\ar@{=} "6";"16"};
{\ar_{x} "10";"12"};
{\ar_{y} "12";"14"};
{\ar_{z} "14";"16"};
{\ar_{} "12";"22"};
{\ar^{} "14";"24"};
{\ar@{=} "22";"24"};
{\ar@{}|\circlearrowright "-12";"4"};
{\ar@{}|\circlearrowright "0";"12"};
{\ar@{}|\circlearrowright "2";"14"};
{\ar@{}|\circlearrowright "4";"16"};
{\ar@{}|\circlearrowright "12";"24"};
\endxy
\]
which shows $U_3^{\prime}\in\mathcal{U}$.

Put $z_2=z_2^{\prime}\circ u_2\colon U_2\to Z_2$, and complete it into a distinguished triangle
\[ S_2[-1]\to U_2\overset{z_2}{\longrightarrow}Z_2\to S_2 \]
in $\mathcal{C}$. Then the octahedron axiom gives the following commutative diagram made of distinguished triangles
\[
\xy
(-20,16)*+{S_1[-1]}="0";
(0.5,15)*+{{}^{\exists}S_2[-1]}="2";
(17,14)*+{S_3[-1]}="4";
(-2,4)*+{U_2}="6";
(6.1,-0.6)*+{U_2^{\prime}}="8";
(-5.8,-14.2)*+{Z_2}="10";
(7.5,8.5)*+_{_{\circlearrowright}}="12";
(-5.5,11.5)*+_{_{\circlearrowright}}="14";
(-0.5,-4.3)*+_{_{\circlearrowright}}="14";
{\ar^{} "0";"2"};
{\ar^{} "2";"4"};
{\ar_{} "0";"6"};
{\ar^{} "2";"6"};
{\ar^{u_2} "6";"8"};
{\ar_{z_2} "6";"10"};
{\ar^{} "4";"8"};
{\ar^{z_2^{\prime}} "8";"10"};
\endxy
\]
which shows $S_2\in\mathcal{S}$. Commutativity of $(\ref{MorphUUUTria})$ can be checked easily.

Take the standard right triangle
\begin{equation}\label{RightTria_v}
Z_1\overset{\underline{x}}{\longrightarrow}Z_2\overset{\eta_{U_3^{\prime}}\circ\underline{y}}{\longrightarrow}\sigma(U_3^{\prime})\overset{\sigma(\underline{v})}{\longrightarrow}\Sigma Z_1
\end{equation}
associated to $Z_1\overset{x}{\longrightarrow}Z_2\overset{y}{\longrightarrow}U_3^{\prime}\overset{z}{\longrightarrow}Z_1[1]$. Remark that we have $\underline{x}=\sigma(\underline{\ell})$. Here,
\begin{equation}\label{S_T_A_R}
\xy
(-19,6)*+{Z_1}="0";
(-6,6)*+{Z_2}="2";
(7,6)*+{U_3^{\prime}}="4";
(21,6)*+{Z_1[1]}="6";
(-19,-6)*+{Z_1}="10";
(-6,-6)*+{I_{Z_1}}="12";
(7,-6)*+{U_{Z_1}}="14";
(21,-6)*+{Z_1[1]}="16";
{\ar^{x} "0";"2"};
{\ar^{y} "2";"4"};
{\ar^{z} "4";"6"};
{\ar@{=} "0";"10"};
{\ar "2";"12"};
{\ar^{v} "4";"14"};
{\ar@{=} "6";"16"};
{\ar_{\iota_{Z_1}} "10";"12"};
{\ar_{\wp_{Z_1}} "12";"14"};
{\ar_{\gamma_{Z_1}} "14";"16"};
{\ar@{}|\circlearrowright "0";"12"};
{\ar@{}|\circlearrowright "2";"14"};
{\ar@{}|\circlearrowright "4";"16"};
\endxy
\end{equation}
is a morphism of triangles, with $I_{Z_1}\in\mathcal{I},U_{Z_1}\in\mathcal{U}$.
Remark that we have a commutative square
\[
\xy
(-14,6)*+{U_2}="0";
(0,6)*+{U_3}="2";
(16,6)*+{\sigma(U_3)}="4";
(-14,-6)*+{Z_2}="10";
(0,-6)*+{U_3^{\prime}}="12";
(16,-6)*+{\sigma(U_3^{\prime})}="14";
{\ar^{\underline{m}} "0";"2"};
{\ar^{\eta_{U_3}} "2";"4"};
{\ar_{\underline{z}_2} "0";"10"};
{\ar^{\underline{u}_3} "2";"12"};
{\ar^{\sigma(\underline{u}_3)} "4";"14"};
{\ar_{\underline{y}} "10";"12"};
{\ar_{\eta_{U_3^{\prime}}} "12";"14"};
{\ar@{}|\circlearrowright "0";"12"};
{\ar@{}|\circlearrowright "2";"14"};
\endxy
\]
in $\mathcal{U}/\mathcal{I}$, and $\sigma(\underline{u}_3)$ is isomorphism by Lemma \ref{LemSUU}. This shows $\sigma(\underline{u}_3)^{-1}\circ\eta_{U_3^{\prime}}\circ\underline{y}=\sigma(\eta_{U_3}\circ\underline{m})$.
Composing the isomorphism $\sigma(\underline{u}_3)$ with $(\ref{RightTria_v})$, we obtain distinguished right triangle $(\ref{RightTria_m})$.
\end{proof}

For the notion of a {\it right triangulated category} below, see for example \cite[Chapter {\rm II}]{BR}.
\begin{prop}\label{PropZIRightTria}
Let $\mathcal{P}$ be any concentric TCP. With the above structures, $\mathcal{Z}/\mathcal{I}$ is a right triangulated category.
\end{prop}
\begin{proof}
{\rm (TR1)} By definition, the class of distinguished right triangles is closed under isomorphisms.

For any $X\in\mathcal{C}$, since $X\overset{\mathrm{id}}{\longrightarrow}X\to 0\to X[1]$ is a $\mathcal{U}$-conic triangle in $\mathcal{C}$, it follows that $X\overset{\mathrm{id}_X}{\longrightarrow}X\to0\to\Sigma X$ is a distinguished right triangle in $\mathcal{Z}/\mathcal{I}$.

For any morphism $\alpha=\underline{f}\in(\mathcal{Z}/\mathcal{I})(X,Y)$, existence of a distinguished right triangle $X\overset{\alpha}{\longrightarrow}Y\to Z\to \Sigma X$ follows from Lemma \ref{LemInvariantStan} {\rm (1)}.

\medskip

{\rm (TR2)} Take any distinguished right triangle. By Lemma \ref{LemInvariantStan}, replacing by an isomorphism of right triangles in $\mathcal{Z}/\mathcal{I}$, we may assume it is a standard right triangle associated to $(\ref{CanonUTria})$. If we draw diagram $(\ref{ShiftedOcta})$ and take a distinguished triangle
\[ S[-1]\to C_f\overset{z}{\longrightarrow}Z\to S\quad(S\in\mathcal{S},Z\in\mathcal{Z}), \]
we may assume $\sigma(C_f)=Z$ and $\underline{z}=\eta_{C_f}$ in $\mathcal{Z}/\mathcal{I}$. Then the associated standard right triangle is isomorphic to
\[ X\overset{\underline{f}}{\longrightarrow}Y\overset{\underline{z}\circ\underline{d}_f}{\longrightarrow}Z\overset{\sigma(\underline{c}_f)}{\longrightarrow}\Sigma X. \]
Also remark that the commutativity of $(\ref{ShiftedOcta})$ especially implies $-q=f[1]\circ\gamma_X$.

Complete $z\circ d_f$ into a distinguished triangle
\begin{equation}\label{UcalShifted}
Y\overset{z\circ d_f}{\longrightarrow}Z\overset{d_1}{\longrightarrow}Q\overset{d_2}{\longrightarrow}Y[1]
\end{equation}
in $\mathcal{C}$. By the octahedron axiom, we obtain the following commutative diagram made of distinguished triangles in $\mathcal{C}$.
\[
\xy
(-8,23)*+{S[-1]}="-12";
(8,23)*+{S[-1]}="-14";
(-24,8)*+{Y}="0";
(-8,8)*+{C_f}="2";
(8,8)*+{U_X}="4";
(24,8)*+{Y[1]}="6";
(-24,-7)*+{Y}="10";
(-8,-7)*+{Z}="12";
(8,-7)*+{Q}="14";
(24,-7)*+{Y[1]}="16";
(-8,-22)*+{S}="22";
(8,-22)*+{S}="24";
{\ar@{=} "-12";"-14"};
{\ar_{} "-12";"2"};
{\ar^{} "-14";"4"};
{\ar^{d_f} "0";"2"};
{\ar^{c_f} "2";"4"};
{\ar^{q} "4";"6"};
{\ar@{=} "0";"10"};
{\ar_{z} "2";"12"};
{\ar^{{}^{\exists}r} "4";"14"};
{\ar@{=} "6";"16"};
{\ar_{z\circ d_f} "10";"12"};
{\ar_{d_1} "12";"14"};
{\ar_{d_2} "14";"16"};
{\ar_{} "12";"22"};
{\ar^{} "14";"24"};
{\ar@{=} "22";"24"};
{\ar@{}|\circlearrowright "-12";"4"};
{\ar@{}|\circlearrowright "0";"12"};
{\ar@{}|\circlearrowright "2";"14"};
{\ar@{}|\circlearrowright "4";"16"};
{\ar@{}|\circlearrowright "12";"24"};
\endxy
\]
This shows $Q\in\mathcal{U}$, and thus $(\ref{UcalShifted})$ is a $\mathcal{U}$-conic triangle.
It suffices to show that its associated standard right triangle is isomorphic to
\[ Y\overset{\underline{z}\circ \underline{d}_f}{\longrightarrow}Z\overset{\sigma(\underline{c}_f)}{\longrightarrow}\Sigma X\overset{-\Sigma\underline{f}}{\longrightarrow}\Sigma Y. \]
As in Definition \ref{DefUStan}, let $Y\overset{\iota_Y}{\longrightarrow}I_Y\overset{\wp_Y}{\longrightarrow}U_Y\overset{\gamma_Y}{\longrightarrow}Y[1]$ be a distinguished triangle in $\mathcal{C}$ satisfying $I_Y\in\mathcal{I},\, U_Y\in\mathcal{U}$, and take a morphism of triangles
\[
\xy
(-19,6)*+{Y}="0";
(-6,6)*+{Z}="2";
(7,6)*+{Q}="4";
(21,6)*+{Y[1]}="6";
(-19,-6)*+{Y}="10";
(-6,-6)*+{I_Y}="12";
(7,-6)*+{U_Y}="14";
(21,-6)*+{Y[1]}="16";
{\ar^{z\circ d_f} "0";"2"};
{\ar^{d_1} "2";"4"};
{\ar^{d_2} "4";"6"};
{\ar@{=} "0";"10"};
{\ar^{g} "2";"12"};
{\ar^{h} "4";"14"};
{\ar@{=} "6";"16"};
{\ar_{\iota_Y} "10";"12"};
{\ar_{\wp_Y} "12";"14"};
{\ar_{\gamma_Y} "14";"16"};
{\ar@{}|\circlearrowright "0";"12"};
{\ar@{}|\circlearrowright "2";"14"};
{\ar@{}|\circlearrowright "4";"16"};
\endxy
\]
in $\mathcal{C}$. By definition, the standard right triangle associated to $(\ref{UcalShifted})$ is given by
\[ Y\overset{\underline{z}\circ\underline{d}_f}{\longrightarrow}Z\overset{\eta_{Q}\circ\underline{d}_1}{\longrightarrow}\sigma(Q)\overset{\sigma(\underline{h})}{\longrightarrow}\Sigma Y. \]
Since $f[1]\circ\gamma_X=-q=-d_2\circ r=-\gamma_Y\circ h\circ r$, we obtain a morphism of triangles
\[
\xy
(-21,7)*+{X}="0";
(-7,7)*+{I_X}="2";
(8,7)*+{U_X}="4";
(23,7)*+{X[1]}="6";
(-21,-7)*+{Y}="10";
(-7,-7)*+{I_Y}="12";
(8,-7)*+{U_Y}="14";
(23,-7)*+{Y[1]}="16";
{\ar^{\iota_X} "0";"2"};
{\ar^{\wp_X} "2";"4"};
{\ar^{\gamma_X} "4";"6"};
{\ar_{f} "0";"10"};
{\ar^{} "2";"12"};
{\ar|*+{_{-h\circ r}} "4";"14"};
{\ar^{f[1]} "6";"16"};
{\ar_{\iota_Y} "10";"12"};
{\ar_{\wp_Y} "12";"14"};
{\ar_{\gamma_Y} "14";"16"};
{\ar@{}|\circlearrowright "0";"12"};
{\ar@{}|\circlearrowright "2";"14"};
{\ar@{}|\circlearrowright "4";"16"};
\endxy
\]
by {\rm (TR2)} and {\rm (TR3)} for $\mathcal{C}$. This gives $-\underline{h}\circ\underline{r}=\underline{f}\langle1\rangle$, and thus $\sigma(\underline{h})\circ\sigma(\underline{r})=-\Sigma\underline{f}$.

By Lemma \ref{LemSUU}, $\sigma(\underline{r})$ becomes isomorphism. This gives an isomorphism of right triangles
\[
\xy
(-23,7)*+{Y}="0";
(-8,7)*+{Z}="2";
(8,7)*+{\Sigma X}="4";
(24,7)*+{\Sigma Y}="6";
(-23,-7)*+{Y}="10";
(-8,-7)*+{Z}="12";
(8,-7)*+{\sigma(Q)}="14";
(24,-7)*+{\Sigma Y}="16";
{\ar^{\underline{z}\circ\underline{d}_f} "0";"2"};
{\ar^{\sigma(\underline{c}_f)} "2";"4"};
{\ar^{-\Sigma\underline{f}} "4";"6"};
{\ar@{=} "0";"10"};
{\ar@{=} "2";"12"};
{\ar_{\cong}^{\sigma(\underline{r})} "4";"14"};
{\ar@{=} "6";"16"};
{\ar_{\underline{z}\circ\underline{d}_f} "10";"12"};
{\ar_{\eta_{Q}\circ\underline{d}_1} "12";"14"};
{\ar_{\sigma(\underline{h})} "14";"16"};
{\ar@{}|\circlearrowright "0";"12"};
{\ar@{}|\circlearrowright "2";"14"};
{\ar@{}|\circlearrowright "4";"16"};
\endxy.
\]

\medskip

{\rm (TR3)} Replacing by isomorphisms of right triangles, this is also reduced to the case of standard right triangles. Let $X\overset{f}{\longrightarrow}Y\overset{a}{\longrightarrow}U\overset{b}{\longrightarrow}X[1]$ and $X^{\prime}\overset{f^{\prime}}{\longrightarrow}Y^{\prime}\overset{a^{\prime}}{\longrightarrow}U^{\prime}\overset{b^{\prime}}{\longrightarrow}X^{\prime}[1]$ be $\mathcal{U}$-conic triangles in $\mathcal{C}$. Suppose we are given a commutative square
\[
\xy
(-7,6)*+{X}="0";
(7,6)*+{Y}="2";
(-7,-6)*+{X^{\prime}}="4";
(7,-6)*+{Y^{\prime}}="6";
{\ar^{\underline{f}} "0";"2"};
{\ar_{\underline{x}} "0";"4"};
{\ar^{\underline{y}} "2";"6"};
{\ar_{\underline{f}^{\prime}} "4";"6"};
{\ar@{}|\circlearrowright "0";"6"};
\endxy
\]
in $\mathcal{Z}/\mathcal{I}$. Let $X\overset{\iota_X}{\longrightarrow}I_X\overset{\wp_X}{\longrightarrow}U_X\overset{\gamma_X}{\longrightarrow}X[1]$ be an distinguished triangle as in $(\ref{UStT1})$. Then $f^{\prime}\circ x-y\circ f$ factors through $\iota_X$. Namely, there exists $j\in\mathcal{C}(I_X,Y^{\prime})$ satisfying
\[ f^{\prime}\circ x=y\circ f+j\circ\iota_X. \]
Complete $\begin{bmatrix}f\\ \iota_X\end{bmatrix}$ into a distinguished triangle
\[ X\overset{\spmatrix{f}{\iota_X}}{\longrightarrow}Y\oplus I_X\overset{a_f}{\longrightarrow}C_f\overset{b_f}{\longrightarrow}X[1]. \]
Then by {\rm (TR3)} for $\mathcal{C}$, we obtain a morphism of triangles
\[
\xy
(-25,7)*+{X}="0";
(-9,7)*+{Y\oplus I_X}="2";
(8,7)*+{C_f}="4";
(23,7)*+{X[1]}="6";
(-25,-7)*+{X^{\prime}}="10";
(-9,-7)*+{Y^{\prime}}="12";
(8,-7)*+{U^{\prime}}="14";
(23,-7)*+{X^{\prime}[1]}="16";
{\ar^{\spmatrix{f}{\iota_X}} "0";"2"};
{\ar^(0.56){a_f} "2";"4"};
{\ar^{b_f} "4";"6"};
{\ar_{x} "0";"10"};
{\ar^{[y\ j]} "2";"12"};
{\ar^{{}^{\exists}u} "4";"14"};
{\ar^{x[1]} "6";"16"};
{\ar_{f^{\prime}} "10";"12"};
{\ar_{a^{\prime}} "12";"14"};
{\ar_{b^{\prime}} "14";"16"};
{\ar@{}|\circlearrowright "0";"12"};
{\ar@{}|\circlearrowright "2";"14"};
{\ar@{}|\circlearrowright "4";"16"};
\endxy
\]
in $\mathcal{C}$. Now {\rm (TR3)} for $\mathcal{Z}/\mathcal{I}$ follows from Proposition \ref{PropMorphStan} and Lemma \ref{LemInvariantStan}.

\medskip

{\rm (TR4)}
Replacing by isomorphisms of right triangles, it suffices to show for standard right triangles associated to $\mathcal{U}$-conic triangles.

Let $X,Y,Z\in\mathcal{Z}$ be arbitrary objects, let
\begin{eqnarray}
&X\overset{f}{\longrightarrow}Y\overset{a_1}{\longrightarrow}U_1\overset{b_1}{\longrightarrow}X[1],&\\
&X\overset{h}{\longrightarrow}Z\overset{a_2}{\longrightarrow}U_2\overset{b_2}{\longrightarrow}X[1],&\label{DTria_h}\\
&Y\overset{g}{\longrightarrow}Z\overset{a_3}{\longrightarrow}U_3\overset{b_3}{\longrightarrow}Y[1]&
\end{eqnarray}
be $\mathcal{U}$-conic triangles, and suppose $\underline{h}=\underline{g}\circ\underline{f}$ holds.

Complete $g\circ f$ into a distinguished triangle
\begin{equation}\label{DTria_gf}
X\overset{g\circ f}{\longrightarrow}Z\to Q\to X[1].
\end{equation}
Then by the octahedron axiom in $\mathcal{C}$, it follows $Q\in\mathcal{U}$. By Proposition \ref{PropUniqueRightComplete}, the standard right triangles associated to $(\ref{DTria_h})$ and $(\ref{DTria_gf})$ are isomorphic as right triangles.

Thus we may replace $(\ref{DTria_h})$ by $(\ref{DTria_gf})$, and assume that $h=g\circ f$ holds from the beginning.
By the octahedron axiom in $\mathcal{C}$, we obtain a distinguished triangle
\[ U_1\overset{\ell}{\longrightarrow}U_2\overset{m}{\longrightarrow}U_3\overset{n}{\longrightarrow}U_1[1] \]
which makes the following diagram commutative in $\mathcal{C}$.
\begin{equation}\label{Octa_13A}
\xy
(-24,23)*+{X}="0";
(-8,23)*+{Y}="2";
(8,23)*+{U_1}="4";
(24,23)*+{X[1]}="6";
(-24,8)*+{X}="10";
(-8,8)*+{Z}="12";
(8,8)*+{U_2}="14";
(24,8)*+{X[1]}="16";
(-8,-7)*+{U_3}="22";
(8,-7)*+{U_3}="24";
(24,-7)*+{Y[1]}="26";
(-8,-22)*+{Y[1]}="32";
(8,-22)*+{U_1[1]}="34";
{\ar^{f} "0";"2"};
{\ar^{a_1} "2";"4"};
{\ar^{b_1} "4";"6"};
{\ar@{=} "0";"10"};
{\ar_{g} "2";"12"};
{\ar^{\ell} "4";"14"};
{\ar@{=} "6";"16"};
{\ar_{h} "10";"12"};
{\ar_{a_2} "12";"14"};
{\ar_{b_2} "14";"16"};
{\ar_{a_3} "12";"22"};
{\ar^{m} "14";"24"};
{\ar^{f[1]} "16";"26"};
{\ar@{=} "22";"24"};
{\ar_{b_3} "24";"26"};
{\ar_{b_3} "22";"32"};
{\ar^{n} "24";"34"};
{\ar_{a_1[1]} "32";"34"};
{\ar@{}|\circlearrowright "0";"12"};
{\ar@{}|\circlearrowright "2";"14"};
{\ar@{}|\circlearrowright "4";"16"};
{\ar@{}|\circlearrowright "12";"24"};
{\ar@{}|\circlearrowright "14";"26"};
{\ar@{}|\circlearrowright "22";"34"};
\endxy
\end{equation}
By Lemma \ref{LemUUUTria}, we obtain distinguished triangles in $\mathcal{C}$
\begin{eqnarray*}
&S_1[-1]\to U_1\overset{z_1}{\longrightarrow}Z_1\to S_1,&\\
&S_2[-1]\to U_2\overset{z_2}{\longrightarrow}Z_2\to S_2,&\\
&S_3[-1]\to U_3\overset{u_3}{\longrightarrow}U_3^{\prime}\to S_3,&\\
&Z_1\overset{x}{\longrightarrow}Z_2\overset{y}{\longrightarrow}U_3^{\prime}\overset{z}{\longrightarrow}Z_1[1],&\\
&(S_1,S_2,S_3\in\mathcal{S},\, Z_1,Z_2\in\mathcal{Z},\, U_3^{\prime}\in\mathcal{U}),&
\end{eqnarray*}
which give a morphism of triangles $(\ref{MorphUUUTria})$.
By composition, we obtain a morphism of triangles
\begin{equation}\label{MorphYZU}
\xy
(-21,6)*+{Y}="0";
(-7,6)*+{Z}="2";
(8,6)*+{U_3}="4";
(23,6)*+{Y[1]}="6";
(-21,-7)*+{Z_1}="10";
(-7,-7)*+{Z_2}="12";
(8,-7)*+{U_3^{\prime}}="14";
(23,-7)*+{Z_1[1]}="16";
{\ar^{g} "0";"2"};
{\ar^{a_3} "2";"4"};
{\ar^{b_3} "4";"6"};
{\ar_{z_1\circ a_1} "0";"10"};
{\ar|*+{_{z_2\circ a_2}} "2";"12"};
{\ar^{u_3} "4";"14"};
{\ar^{(z_1\circ a_1)[1]} "6";"16"};
{\ar_{x} "10";"12"};
{\ar_{y} "12";"14"};
{\ar_{z} "14";"16"};
{\ar@{}|\circlearrowright "0";"12"};
{\ar@{}|\circlearrowright "2";"14"};
{\ar@{}|\circlearrowright "4";"16"};
\endxy.
\end{equation}
Take morphisms of distinguished triangles in $\mathcal{C}$
\[
\xy
(-21,6)*+{X}="0";
(-7,6)*+{Y}="2";
(8,6)*+{U_1}="4";
(23,6)*+{X[1]}="6";
(-21,-6)*+{X}="10";
(-7,-6)*+{I_X}="12";
(8,-6)*+{U_X}="14";
(23,-6)*+{X[1]}="16";
{\ar^{f} "0";"2"};
{\ar^{a_1} "2";"4"};
{\ar^{b_1} "4";"6"};
{\ar@{=} "0";"10"};
{\ar "2";"12"};
{\ar^{c_1} "4";"14"};
{\ar@{=} "6";"16"};
{\ar_{\iota_X} "10";"12"};
{\ar_{\wp_X} "12";"14"};
{\ar_{\gamma_X} "14";"16"};
{\ar@{}|\circlearrowright "0";"12"};
{\ar@{}|\circlearrowright "2";"14"};
{\ar@{}|\circlearrowright "4";"16"};
\endxy\ \ ,
\ \ \quad
\xy
(-21,6)*+{X}="0";
(-7,6)*+{Z}="2";
(8,6)*+{U_2}="4";
(23,6)*+{X[1]}="6";
(-21,-6)*+{X}="10";
(-7,-6)*+{I_X}="12";
(8,-6)*+{U_X}="14";
(23,-6)*+{X[1]}="16";
{\ar^{h} "0";"2"};
{\ar^{a_2} "2";"4"};
{\ar^{b_2} "4";"6"};
{\ar@{=} "0";"10"};
{\ar "2";"12"};
{\ar^{c_2} "4";"14"};
{\ar@{=} "6";"16"};
{\ar_{\iota_X} "10";"12"};
{\ar_{\wp_X} "12";"14"};
{\ar_{\gamma_X} "14";"16"};
{\ar@{}|\circlearrowright "0";"12"};
{\ar@{}|\circlearrowright "2";"14"};
{\ar@{}|\circlearrowright "4";"16"};
\endxy,
\]
\[
\xy
(-21,6)*+{Y}="0";
(-7,6)*+{Z}="2";
(8,6)*+{U_3}="4";
(23,6)*+{Y[1]}="6";
(-21,-6)*+{Y}="10";
(-7,-6)*+{I_Y}="12";
(8,-6)*+{U_Y}="14";
(23,-6)*+{Y[1]}="16";
{\ar^{g} "0";"2"};
{\ar^{a_3} "2";"4"};
{\ar^{b_3} "4";"6"};
{\ar@{=} "0";"10"};
{\ar "2";"12"};
{\ar^{c_3} "4";"14"};
{\ar@{=} "6";"16"};
{\ar_{\iota_Y} "10";"12"};
{\ar_{\wp_Y} "12";"14"};
{\ar_{\gamma_Y} "14";"16"};
{\ar@{}|\circlearrowright "0";"12"};
{\ar@{}|\circlearrowright "2";"14"};
{\ar@{}|\circlearrowright "4";"16"};
\endxy
\]
with $I_X,I_Y\in\mathcal{I},\, U_X,U_Y\in\mathcal{U}$.

Applying Proposition \ref{PropMorphStan} to $(\ref{Octa_13A})$ and $(\ref{MorphYZU})$, we obtain morphisms of right triangles
\[
\xy
(-24,6)*+{X}="0";
(-9,6)*+{Y}="2";
(9,6)*+{\sigma(U_1)}="4";
(27,6)*+{\Sigma X}="6";
(-24,-6)*+{X}="10";
(-9,-6)*+{Z}="12";
(9,-6)*+{\sigma(U_2)}="14";
(27,-6)*+{\Sigma X}="16";
{\ar^{\underline{f}} "0";"2"};
{\ar^(0.46){\eta_{U_1}\circ\underline{a}_1} "2";"4"};
{\ar^{\sigma(\underline{c}_1)} "4";"6"};
{\ar@{=} "0";"10"};
{\ar^{\underline{g}} "2";"12"};
{\ar^{\sigma(\underline{\ell})} "4";"14"};
{\ar@{=} "6";"16"};
{\ar_{\underline{h}} "10";"12"};
{\ar_(0.46){\eta_{U_2}\circ\underline{a}_2} "12";"14"};
{\ar_{\sigma(\underline{c}_2)} "14";"16"};
{\ar@{}|\circlearrowright "0";"12"};
{\ar@{}|\circlearrowright "2";"14"};
{\ar@{}|\circlearrowright "4";"16"};
\endxy,
\]
\[
\xy
(-24,7)*+{Y}="0";
(-9,7)*+{Z}="2";
(9,7)*+{\sigma(U_3)}="4";
(11,7)*+{}="5";
(27,7)*+{\Sigma Y}="6";
(-24,-6)*+{Z_1}="10";
(-9,-6)*+{Z_2}="12";
(9,-6)*+{\sigma(U_3^{\prime})}="14";
(27,-6)*+{\Sigma Z_1}="16";
{\ar^{\underline{g}} "0";"2"};
{\ar^(0.46){\eta_{U_3}\circ\underline{a}_3} "2";"4"};
{\ar^{\sigma(\underline{c}_3)} "4";"6"};
{\ar_{\underline{z}_1\circ\underline{a}_1} "0";"10"};
{\ar|*+{_{\underline{z}_2\circ\underline{a}_2}} "2";"12"};
{\ar^{\sigma(\underline{u}_3)}_{\cong} "4";"14"};
{\ar^{\Sigma(\underline{z}_1\circ\underline{a}_1)} "6";"16"};
{\ar_{\underline{x}} "10";"12"};
{\ar_(0.46){\eta_{U_3^{\prime}}\circ\underline{y}} "12";"14"};
{\ar_{\sigma(\underline{v})} "14";"16"};
{\ar@{}|\circlearrowright "0";"12"};
{\ar@{}|\circlearrowright "2";"14"};
{\ar@{}|\circlearrowright "5";"16"};
\endxy,
\]
where $v$ is the morphism appearing in $(\ref{S_T_A_R})$.
This gives the following commutative diagram of distinguished right triangles in $\mathcal{Z}/\mathcal{I}$.
\[
\xy
(-28,23)*+{X}="0";
(-10,23)*+{Y}="2";
(10,23)*+{\sigma(U_1)}="4";
(30,23)*+{\Sigma X}="6";
(-28,8)*+{X}="10";
(-10,8)*+{Z}="12";
(10,8)*+{\sigma(U_2)}="14";
(30,8)*+{\Sigma X}="16";
(-10,-7)*+{\sigma(U_3)}="22";
(10,-7)*+{\sigma(U_3)}="24";
%(24,-7)*+{B[1]}="26";
%
(-10,-22)*+{\Sigma Y}="32";
(10,-22)*+{\Sigma(\sigma U_1)}="34";
{\ar^{\underline{f}} "0";"2"};
{\ar^{\eta_{U_1}\circ\underline{a}_1} "2";"4"};
{\ar^{\sigma(\underline{c}_1)} "4";"6"};
{\ar@{=} "0";"10"};
{\ar_{\underline{g}} "2";"12"};
{\ar^{\sigma(\underline{\ell})} "4";"14"};
{\ar@{=} "6";"16"};
{\ar_{\underline{h}} "10";"12"};
{\ar_{\eta_{U_2}\circ\underline{a}_2} "12";"14"};
{\ar_{\sigma(\underline{c}_2)} "14";"16"};
{\ar_{\eta_{U_3}\circ\underline{a}_3} "12";"22"};
{\ar^{\sigma(\eta_{U_3}\circ\underline{m})=\sigma(\underline{u}_3)^{-1}\circ\eta_{U_3^{\prime}}\circ\underline{y}} "14";"24"};
{\ar@{=} "22";"24"};
{\ar_{\sigma(\underline{c}_3)} "22";"32"};
{\ar^{\sigma(\underline{v})\circ\sigma(\underline{u}_3)^{-1}} "24";"34"};
{\ar_{\Sigma(\eta_{U_1}\circ\underline{a}_1)} "32";"34"};
{\ar@{}|\circlearrowright "0";"12"};
{\ar@{}|\circlearrowright "2";"14"};
{\ar@{}|\circlearrowright "4";"16"};
{\ar@{}|\circlearrowright "12";"24"};
%{\ar@{}|\circlearrowright "14";"26"};
{\ar@{}|\circlearrowright "22";"34"};
\endxy
\]
It remains to show the commutativity of
\begin{equation}\label{Comm13A}
\xy
(-9,7)*+{\sigma(U_2)}="0";
(9,7)*+{\Sigma X}="2";
(-9,-6)*+{\sigma(U_3)}="4";
(9,-6)*+{\Sigma Y}="6";
{\ar^{\sigma(\underline{c}_2)} "0";"2"};
{\ar_{\sigma(\eta_{U_3}\circ\underline{m})} "0";"4"};
{\ar^{\Sigma\underline{f}} "2";"6"};
{\ar_{\sigma(\underline{c}_3)} "4";"6"};
{\ar@{}|\circlearrowright "0";"6"};
\endxy.
\end{equation}
By definition, if we take a morphism of triangles
\[
\xy
(-21,6)*+{X}="0";
(-7,6)*+{I_X}="2";
(8,6)*+{U_X}="4";
(23,6)*+{X[1]}="6";
(-21,-7)*+{Y}="10";
(-7,-7)*+{I_Y}="12";
(8,-7)*+{U_Y}="14";
(23,-7)*+{Y[1]}="16";
{\ar^{\iota_X} "0";"2"};
{\ar^{\wp_X} "2";"4"};
{\ar^{\gamma_X} "4";"6"};
{\ar_{f} "0";"10"};
{\ar "2";"12"};
{\ar^{u_f} "4";"14"};
{\ar^{f[1]} "6";"16"};
{\ar_{\iota_Y} "10";"12"};
{\ar_{\wp_Y} "12";"14"};
{\ar_{\gamma_Y} "14";"16"};
{\ar@{}|\circlearrowright "0";"12"};
{\ar@{}|\circlearrowright "2";"14"};
{\ar@{}|\circlearrowright "4";"16"};
\endxy
\]
in $\mathcal{C}$, this gives $\Sigma\underline{f}=\sigma(\underline{u}_f)$.

By the functoriality of $\sigma$, to show the commutativity of $(\ref{Comm13A})$, it suffices to show the commutativity of
\[
\xy
(-7,6)*+{U_2}="0";
(8,6)*+{U_X}="2";
(-7,-6)*+{U_3}="4";
(8,-6)*+{U_Y}="6";
{\ar^{\underline{c}_2} "0";"2"};
{\ar_{\underline{m}} "0";"4"};
{\ar^{\underline{u}_f} "2";"6"};
{\ar_{\underline{c}_3} "4";"6"};
{\ar@{}|\circlearrowright "0";"6"};
\endxy
\]
in $\mathcal{U}/\mathcal{I}$. This follows from the existence of the distinguished triangle $Y\overset{\iota_Y}{\longrightarrow}I_Y\overset{\wp_Y}{\longrightarrow}U_Y\overset{\gamma_Y}{\longrightarrow}Y[1]$, and the equality
\[ \gamma_Y\circ (c_3\circ m)=b_3\circ m=f[1]\circ b_2%
=f[1]\circ\gamma_X\circ c_2=\gamma_Y\circ (u_f\circ c_2) \]
in $\mathcal{C}$.
%\begin{eqnarray*}
%\gamma_Y\circ (c_3\circ m)&=&b_3\circ m\ =\ f[1]\circ b_2\\
%&=&f[1]\circ\gamma_X\circ c_2\ =\ \gamma_Y\circ (u_f\circ c_2).
%\end{eqnarray*}
%it reduces to show the commutativity of
%\[
%\xy
%(-10,7)*+{U_2}="0";
%(10,7)*+{U_X}="2";
%(-10,-7)*+{U_3}="4";
%(10,-7)*+{Y[1]}="6";
%%
%{\ar^{c_2} "0";"2"};
%{\ar_{m} "0";"4"};
%{\ar^{\gamma_Y\circ u_f} "2";"6"};
%{\ar_{\gamma_Y\circ c_3} "4";"6"};
%%
%{\ar@{}|\circlearrowright "0";"6"};
%\endxy.
%\]
%This follows from
\end{proof}

\subsection{Pretriangulated structure}
So far we have given a right triangulation of $\mathcal{Z}/\mathcal{I}$. Dually, the following construction gives a left triangulation of $\mathcal{Z}/\mathcal{I}$.
\begin{dfn}\label{DefLeftTriangle}
Let $\mathcal{P}$ be a concentric TCP. To any distinguished triangle in $\mathcal{C}$
\[ P[-1]\overset{k}{\longrightarrow}T\overset{\ell}{\longrightarrow}Z\overset{m}{\longrightarrow}P \]
satisfying $T\in\mathcal{T}$ and $Z,P\in\mathcal{Z}$, we associate {\it standard left triangle}
\[ \Omega P\overset{\omega(\underline{n})}{\longrightarrow}\omega(T)\overset{\underline{\ell}\circ\varepsilon_T}{\longrightarrow}Z\overset{\underline{m}}{\longrightarrow}P \]
by the following {\rm (i)$^{\prime}$,(ii)$^{\prime}$}.
\begin{itemize}
\item[{\rm (i)$^{\prime}$}] Take a distinguished triangle in $\mathcal{C}$
\[ T_P\overset{\iota^{\prime}_P}{\longrightarrow}I^{\prime}_P\overset{\wp_P^{\prime}}{\longrightarrow}P\overset{\gamma_P^{\prime}}{\longrightarrow}T_P[1] \]
satisfying $T_P\in\mathcal{T},\, I_P^{\prime}\in\mathcal{I}$. Remark that this gives $T_P\cong P\langle-1\rangle$ in $\mathcal{T}/\mathcal{I}$, and thus $\omega(T_P)\cong\Omega P$ in $\mathcal{Z}/\mathcal{I}$.
\item[{\rm (ii)$^{\prime}$}] By $\mathcal{C}(I^{\prime}_P,T[1])=0$, we obtain a morphism of triangles in $\mathcal{C}$
\begin{equation}\label{Left}
\xy
(-21,6)*+{T_P}="0";
(-7,6)*+{I^{\prime}_P}="2";
(8,6)*+{P}="4";
(23,6)*+{T_P[1]}="6";
(-21,-7)*+{T}="10";
(-7,-7)*+{Z}="12";
(8,-7)*+{P}="14";
(23,-7)*+{T[1]}="16";
{\ar^{\iota_P^{\prime}} "0";"2"};
{\ar^{\wp_P^{\prime}} "2";"4"};
{\ar^{\gamma_P^{\prime}} "4";"6"};
{\ar_{n} "0";"10"};
{\ar "2";"12"};
{\ar@{=} "4";"14"};
{\ar^{n[1]} "6";"16"};
{\ar_{\ell} "10";"12"};
{\ar_{m} "12";"14"};
{\ar_{-k[1]} "14";"16"};
{\ar@{}|\circlearrowright "0";"12"};
{\ar@{}|\circlearrowright "2";"14"};
{\ar@{}|\circlearrowright "4";"16"};
\endxy.
\end{equation}
\end{itemize}
A left triangle in $\mathcal{Z}/\mathcal{I}$ isomorphic to a standard one, is called a {\it distinguished left triangle}.
\end{dfn}

\begin{prop}\label{PropLeftTriangulated}
Let $\mathcal{P}$ be any concentric TCP. Then, $\mathcal{Z}/\mathcal{I}$ becomes a left triangulated category with the functor $\Omega$ and the class of distinguished left triangles given in Definition \ref{DefLeftTriangle}.
\end{prop}
\begin{proof}
This is dual to Proposition \ref{PropZIRightTria}.
\end{proof}

\begin{thm}\label{ThmPreTriangulated}
Let $\mathcal{P}$ be any concentric TCP. With the right and left triangulation given in Propositions \ref{PropZIRightTria} and \ref{PropLeftTriangulated}, the category $\mathcal{Z}/\mathcal{I}$ becomes a pretriangulated category in the sense of \cite[Definition 1.1]{BR}.
\end{thm}
\begin{proof}
By Propositions \ref{PropShiftAdj}, \ref{PropZIRightTria}, \ref{PropLeftTriangulated}, it remains to show the following {\rm (1),(2)} for any standard right triangle
\begin{equation}\label{PreRight}
X\overset{\underline{f}}{\longrightarrow}Y\overset{\eta_U\circ\underline{a}}{\longrightarrow}\sigma(U)\overset{\sigma(\underline{c})}{\longrightarrow}\Sigma X
\end{equation}
and any standard left triangle
\begin{equation}\label{PreLeft}
\Omega P\overset{\omega(\underline{n})}{\longrightarrow}\omega(T)\overset{\underline{\ell}\circ\varepsilon_T}{\longrightarrow}Z\overset{\underline{m}}{\longrightarrow}P.
\end{equation}
\begin{enumerate}
\item Let $\underline{x}\in(\mathcal{Z}/\mathcal{I})(X,\Omega P),\underline{y}\in(\mathcal{Z}/\mathcal{I})(Y,\omega(T))$ be any pair of morphisms satisfying $\underline{y}\circ\underline{f}=\omega(\underline{n})\circ\underline{x}$. Let $\alpha\in(\mathcal{Z}/\mathcal{I})(\Sigma X,P)$ be the morphism corresponding to $\underline{x}$ by the adjointness $\Sigma\dashv\Omega$.
Then, there exists a morphism $\sigma(U)\to Z$ in $\mathcal{Z}/\mathcal{I}$, which makes the following diagram commutative.
\[
\xy
(-21,6)*+{X}="0";
(-7,6)*+{Y}="2";
(8,6)*+{\sigma(U)}="4";
(23,6)*+{\Sigma X}="6";
(-21,-7)*+{\Omega P}="10";
(-7,-7)*+{\omega(T)}="12";
(8,-7)*+{Z}="14";
(23,-7)*+{P}="16";
{\ar^{\underline{f}} "0";"2"};
{\ar^{\eta_U\circ\underline{a}} "2";"4"};
{\ar^{\sigma(\underline{c})} "4";"6"};
{\ar_{\underline{x}} "0";"10"};
{\ar^{\underline{y}} "2";"12"};
{\ar "4";"14"};
{\ar^{\alpha} "6";"16"};
{\ar_{\omega(\underline{n})} "10";"12"};
{\ar_{\underline{\ell}\circ\varepsilon_T} "12";"14"};
{\ar_{\underline{m}} "14";"16"};
{\ar@{}|\circlearrowright "0";"12"};
{\ar@{}|\circlearrowright "2";"14"};
{\ar@{}|\circlearrowright "4";"16"};
\endxy
\]
\item Let $\underline{z}\in(\mathcal{Z}/\mathcal{I})(\sigma(U),Z),\underline{p}\in(\mathcal{Z}/\mathcal{I})(\Sigma X,P)$ be any pair of morphisms satisfying $\underline{m}\circ\underline{z}=\underline{p}\circ\sigma(\underline{c})$. Let $\beta\in(\mathcal{Z}/\mathcal{I})(X,\Omega P)$ be the morphism corresponding to $\underline{p}$ by the adjointness $\Sigma\dashv\Omega$.
Then, there exists a morphism $Y\to\omega(T)$ in $\mathcal{Z}/\mathcal{I}$, which makes the following diagram commutative.
\[
\xy
(-21,6)*+{X}="0";
(-7,6)*+{Y}="2";
(8,6)*+{\sigma(U)}="4";
(23,6)*+{\Sigma X}="6";
(-21,-7)*+{\Omega P}="10";
(-7,-7)*+{\omega(T)}="12";
(8,-7)*+{Z}="14";
(23,-7)*+{P}="16";
{\ar^{\underline{f}} "0";"2"};
{\ar^{\eta_U\circ\underline{a}} "2";"4"};
{\ar^{\sigma(\underline{c})} "4";"6"};
{\ar_{\beta} "0";"10"};
{\ar^{} "2";"12"};
{\ar^{\underline{z}} "4";"14"};
{\ar^{\underline{p}} "6";"16"};
{\ar_{\omega(\underline{n})} "10";"12"};
{\ar_{\underline{\ell}\circ\varepsilon_T} "12";"14"};
{\ar_{\underline{m}} "14";"16"};
{\ar@{}|\circlearrowright "0";"12"};
{\ar@{}|\circlearrowright "2";"14"};
{\ar@{}|\circlearrowright "4";"16"};
\endxy
\]
\end{enumerate}

\smallskip

Since {\rm (2)} can be shown in a similar way, we only show {\rm (1)}. Let
\[ X\overset{\iota_X}{\longrightarrow}I_X\overset{\wp_X}{\longrightarrow}U_X\overset{\gamma_X}{\longrightarrow}X[1],\quad %
T_P\overset{\iota_P^{\prime}}{\longrightarrow}I^{\prime}_P\overset{\wp_P^{\prime}}{\longrightarrow}P\overset{\gamma_P^{\prime}}{\longrightarrow}T_P[1] \]
be distinguished triangles in $\mathcal{C}$ satisfying $I_X,I_P^{\prime}\in\mathcal{I},\, U_X\in\mathcal{U},\, T_P\in\mathcal{T}$. Let $(\ref{PreRight})$ be the standard right triangle obtained in Definition \ref{DefUStan}, and $(\ref{PreLeft})$ the standard left triangle obtained in Definition \ref{DefLeftTriangle}, respectively.

By Lemma \ref{LemInvariantStan}, by an isomorphism of right triangles, we may replace $(\ref{PreRight})$ by $(\ref{AssocRightCanon})$
\[ X\overset{\begin{bmatrix}\underline{f}\\ \underline{\iota_X}\end{bmatrix}}{\longrightarrow}Y\oplus I_X\overset{\eta_{C_f}\circ\underline{a}_f}{\longrightarrow}\sigma(C_f)\overset{\sigma(\underline{c}_f)}{\longrightarrow}\Sigma X, \]
obtained from $(\ref{ShiftedOcta})$. Take a morphism of distinguished triangles in $\mathcal{C}$
\[
\xy
(-23,6)*+{S_{C_f}[-1]}="0";
(-7,6)*+{C_f}="2";
(8,6)*+{Z_{C_f}}="4";
(23,6)*+{S_{C_f}}="6";
(-23,-7)*+{S_{U_X}[-1]}="10";
(-7,-7)*+{U_X}="12";
(8,-7)*+{Z_{U_X}}="14";
(23,-7)*+{S_{U_X}}="16";
{\ar^{} "0";"2"};
{\ar^{z_{C_f}} "2";"4"};
{\ar^{} "4";"6"};
{\ar_{} "0";"10"};
{\ar^{c_f} "2";"12"};
{\ar^{u} "4";"14"};
{\ar^{} "6";"16"};
{\ar_{} "10";"12"};
{\ar_{z_{U_X}} "12";"14"};
{\ar_{} "14";"16"};
{\ar@{}|\circlearrowright "0";"12"};
{\ar@{}|\circlearrowright "2";"14"};
{\ar@{}|\circlearrowright "4";"16"};
\endxy
\]
satisfying $S_{C_f},S_{U_X}\in\mathcal{S}$ and $Z_{C_f},Z_{U_X}\in\mathcal{Z}$. We may assume $Z_{C_f}=\sigma(C_f),\, Z_{U_X}=\sigma(U_X)$ and $\underline{z}_{C_f}=\eta_{C_f},\, \underline{z}_{U_X}=\eta_{U_X},\, \underline{u}=\sigma(\underline{c}_f)$.

Take morphisms of distinguished triangles $(\ref{Left})$ and
\[
\xy
(-21,6)*+{V_{T_P}}="0";
(-7,6)*+{Z_{T_P}}="2";
(8,6)*+{T_P}="4";
(25,6)*+{V_{T_P}[1]}="6";
(-21,-7)*+{V_T}="10";
(-7,-7)*+{Z_T}="12";
(8,-7)*+{T}="14";
(25,-7)*+{V_T[1]}="16";
{\ar^{} "0";"2"};
{\ar^{z^{\prime}_{T_P}} "2";"4"};
{\ar^{} "4";"6"};
{\ar_{} "0";"10"};
{\ar^{p} "2";"12"};
{\ar^{n} "4";"14"};
{\ar^{} "6";"16"};
{\ar_{} "10";"12"};
{\ar_{z_T^{\prime}} "12";"14"};
{\ar_{} "14";"16"};
{\ar@{}|\circlearrowright "0";"12"};
{\ar@{}|\circlearrowright "2";"14"};
{\ar@{}|\circlearrowright "4";"16"};
\endxy
\]
satisfying $V_T,V_{T_P}\in\mathcal{V}$ and $Z_T,Z_{T_P}\in\mathcal{Z}$.
We may assume $Z_{T_P}=\omega(T_P)=\Omega P,\, Z_T=\omega(T)$ and $\underline{z}_{T_P}^{\prime}=\varepsilon_{T_P},\, \underline{z}_T^{\prime}=\varepsilon_T,\, \underline{p}=\omega(\underline{n})$.

Complete $n$ into a distinguished triangle
\[ T_P\overset{n}{\longrightarrow}T\overset{q}{\longrightarrow}Q\overset{r}{\longrightarrow}T_P[1] \]
in $\mathcal{C}$. Applying the octahedron axiom to
\begin{eqnarray*}
&P[-1]\overset{-\gamma_P^{\prime}[1]}{\longrightarrow}T_P\overset{\iota_P^{\prime}}{\longrightarrow}I_P^{\prime}\overset{\wp_P^{\prime}}{\longrightarrow}P,&\\
&P[-1]\overset{k}{\longrightarrow}T\overset{\ell}{\longrightarrow}Z\overset{m}{\longrightarrow}P,&\\
&T_P\overset{n}{\longrightarrow}T\overset{q}{\longrightarrow}Q\overset{r}{\longrightarrow}T_P[1],&
\end{eqnarray*}
we obtain a distinguished triangle
\[ I_P^{\prime}\overset{{}^{\exists}i^{\prime}}{\longrightarrow}Z\overset{{}^{\exists}q^{\prime}}{\longrightarrow}Q\overset{\iota_P^{\prime}[1]\circ r}{\longrightarrow}I_P^{\prime} [1] \]
which makes the following diagram commutative.
\[
\xy
(-26,23)*+{P[-1]}="0";
(-8,23)*+{T_P}="2";
(8,23)*+{I^{\prime}_P}="4";
(24,23)*+{P}="6";
(-26,8)*+{P[-1]}="10";
(-8,8)*+{T}="12";
(8,8)*+{Z}="14";
(24,8)*+{P}="16";
(-8,-7)*+{Q}="22";
(8,-7)*+{Q}="24";
(24,-7)*+{T_P[1]}="26";
(-8,-22)*+{T_P[1]}="32";
(8,-22)*+{I_P^{\prime}[1]}="34";
{\ar^{-\gamma_P^{\prime}[1]} "0";"2"};
{\ar^{\iota_P^{\prime}} "2";"4"};
{\ar^{\wp_P^{\prime}} "4";"6"};
{\ar@{=} "0";"10"};
{\ar_{n} "2";"12"};
{\ar^{i^{\prime}} "4";"14"};
{\ar@{=} "6";"16"};
{\ar_{k} "10";"12"};
{\ar_{\ell} "12";"14"};
{\ar_{m} "14";"16"};
{\ar_{q} "12";"22"};
{\ar^{q^{\prime}} "14";"24"};
{\ar^{\gamma_P^{\prime}} "16";"26"};
{\ar@{=} "22";"24"};
{\ar_{r} "24";"26"};
{\ar_{r} "22";"32"};
{\ar^{\iota_P^{\prime}[1]\circ r} "24";"34"};
{\ar_{\iota_P^{\prime}[1]} "32";"34"};
{\ar@{}|\circlearrowright "0";"12"};
{\ar@{}|\circlearrowright "2";"14"};
{\ar@{}|\circlearrowright "4";"16"};
{\ar@{}|\circlearrowright "12";"24"};
{\ar@{}|\circlearrowright "14";"26"};
{\ar@{}|\circlearrowright "22";"34"};
\endxy
\]
By $\underline{y}\circ\underline{f}=\omega(\underline{n})\circ\underline{x}=\underline{p}\circ\underline{x}$, the difference $p\circ x-y\circ f$ factors through $\iota_X$. Namely, there exists $j\in\mathcal{C}(I_X,Z_T)$ which makes the following diagram commutative in $\mathcal{C}$.
\[
\xy
(-9,7)*+{X}="0";
(9,7)*+{Y\oplus I_X}="2";
(-9,-7)*+{Z_{T_P}}="4";
(9,-7)*+{Z_T}="6";
{\ar^(0.46){\spmatrix{f}{\iota_X}} "0";"2"};
{\ar_{x} "0";"4"};
{\ar^{[y\ j]} "2";"6"};
{\ar_{p} "4";"6"};
{\ar@{}|\circlearrowright "0";"6"};
\endxy
\]
By {\rm (TR3)} in $\mathcal{C}$, we obtain $e\in\mathcal{C}(C_f,Q)$ which gives the following morphism of triangles in $\mathcal{C}$.
\[
\xy
(-25,7)*+{X}="0";
(-9,7)*+{Y\oplus I_X}="2";
(8,7)*+{C_f}="4";
(23,7)*+{X[1]}="6";
(-25,-7)*+{T_P}="10";
(-9,-7)*+{T}="12";
(8,-7)*+{Q}="14";
(23,-7)*+{T_P[1]}="16";
{\ar^(0.46){\spmatrix{f}{\iota_X}} "0";"2"};
{\ar^(0.56){a_f} "2";"4"};
{\ar^{b_f} "4";"6"};
{\ar_{z_{T_P}^{\prime}\circ x} "0";"10"};
{\ar|*+{_{z_T^{\prime}\circ [y\ j]}} "2";"12"};
{\ar^{e} "4";"14"};
{\ar^{(z_{T_P}^{\prime}\circ x)[1]} "6";"16"};
{\ar_{n} "10";"12"};
{\ar_{q} "12";"14"};
{\ar_{r} "14";"16"};
{\ar@{}|\circlearrowright "0";"12"};
{\ar@{}|\circlearrowright "2";"14"};
{\ar@{}|\circlearrowright "4";"16"};
\endxy
\]
Then by $\mathcal{C}(C_f,I_P^{\prime}[1])=0$ and $\mathcal{C}(S_{C_f}[-1],Z)=0$, we obtain $e^{\prime}\in\mathcal{C}(C_f,Z)$ and $e^{\prime\prime}\in\mathcal{C}(Z_{C_f},Z)$ which makes the following diagram commutative in $\mathcal{C}$.
\[
\xy
(-20,8)*+{S_{C_f}[-1]}="-2";
(-20,-8)*+{I_P^{\prime}[1]}="-4";
(0,8)*+{C_f}="0";
(16,8)*+{Z_{C_f}}="2";
(6,0)*+{}="3";
(0,-8)*+{Q}="4";
(8,2)*+{}="5";
(16,-8)*+{Z}="6";
{\ar^{z_{C_f}} "0";"2"};
{\ar_{e} "0";"4"};
{\ar^{e^{\prime\prime}} "2";"6"};
{\ar^{q^{\prime}} "6";"4"};
{\ar|*+{_{e^{\prime}}} "0";"6"};
{\ar "-2";"0"};
{\ar "4";"-4"};
{\ar@{}|\circlearrowright "2";"3"};
{\ar@{}|\circlearrowright "4";"5"};
\endxy
\]
It remains to show the commutativity of the following diagram.
\begin{equation}\label{CommToShowPreTr}
\xy
(-20,8)*+{Y\oplus I_X}="0";
(0,8)*+{Z_{C_f}}="2";
(16,8)*+{\Sigma X}="4";
(-20,-8)*+{Z_T}="10";
(0,-8)*+{Z}="12";
(16,-8)*+{P}="14";
{\ar^{\underline{z}_{C_f}\circ\underline{a}_f} "0";"2"};
{\ar^{\underline{u}} "2";"4"};
{\ar_{\underline{[y\ j]}} "0";"10"};
{\ar^{\underline{e}^{\prime\prime}} "2";"12"};
{\ar^{\alpha} "4";"14"};
{\ar_{\underline{\ell}\circ\underline{z}^{\prime}_T} "10";"12"};
{\ar_{\underline{m}} "12";"14"};
{\ar@{}|\circlearrowright "0";"12"};
{\ar@{}|\circlearrowright "2";"14"};
\endxy
\end{equation}

\medskip

Commutativity of the left square in $(\ref{CommToShowPreTr})$ follows from the equality
\[ q^{\prime}\circ(\ell\circ z_T^{\prime}\circ [y\ j])=q\circ z_T^{\prime}\circ [y\ j]=e\circ a_f=q^{\prime}\circ(e^{\prime\prime}\circ z_{C_f}\circ a_f) \]
and the existence of the distinguished triangle $I_P^{\prime}\to Z\overset{q^{\prime}}{\longrightarrow}Q\to I_P^{\prime}[1]$ in $\mathcal{C}$.
%\[
%\xy
%(-12,10)*+{Y\oplus I_X}="0";
%(12,10)*+{Z_{C_f}}="2";
%(-12,-10)*+{Z_T}="4";
%(2,-10)*+{Z}="6";
%(12,-1)*+{Z}="8";
%(12,-10)*+{Q}="10";
%%
%{\ar^{z_{C_f}\circ a_f} "0";"2"};
%{\ar_{[y\ j]} "0";"4"};
%{\ar^{e^{\prime\prime}} "2";"8"};
%{\ar_{\ell\circ z_T^{\prime}} "4";"6"};
%{\ar_{q^{\prime}} "6";"10"};
%{\ar^{q^{\prime}} "8";"10"};
%%
%{\ar@{}|\circlearrowright "0";"10"};
%\endxy
%\]

Let us show the commutativity of the right square in $(\ref{CommToShowPreTr})$.
As in Proposition \ref{PropShiftAdj}, the morphism $\alpha$ is obtained in the following way.
\begin{itemize}
\item[-] There is a morphism of triangles in $\mathcal{C}$
\[
\xy
(-23,7)*+{X}="0";
(-7,7)*+{I_X}="2";
(8,7)*+{U_X}="4";
(23,7)*+{X[1]}="6";
(-23,-7)*+{T_P}="10";
(-7,-7)*+{I_P^{\prime}}="12";
(8,-7)*+{P}="14";
(23,-7)*+{T_P[1]}="16";
{\ar^{\iota_X} "0";"2"};
{\ar^{\wp_X} "2";"4"};
{\ar^{\gamma_X} "4";"6"};
{\ar_{z_{T_P}^{\prime}\circ x} "0";"10"};
{\ar "2";"12"};
{\ar^{{}^{\exists}g} "4";"14"};
{\ar^{(z_{T_P}^{\prime}\circ x)[1]} "6";"16"};
{\ar_{\iota_P^{\prime}} "10";"12"};
{\ar_{\wp_P^{\prime}} "12";"14"};
{\ar_{\gamma_P^{\prime}} "14";"16"};
{\ar@{}|\circlearrowright "0";"12"};
{\ar@{}|\circlearrowright "2";"14"};
{\ar@{}|\circlearrowright "4";"16"};
\endxy
\]
\item[-] By $\mathcal{C}(S_{U_X}[-1],P)=0$, there exists $h\in\mathcal{C}(Z_{U_X},P)$ which makes the following diagram commutative.
\[
\xy
(-26,6)*+{S_{U_X}[-1]}="0";
(-8,6)*+{U_X}="2";
(2,-2)*+{}="3";
(8,6)*+{Z_{U_X}}="4";
(24,6)*+{S_{U_X}}="6";
(-8,-10)*+{P}="8";
{\ar^{} "0";"2"};
{\ar^{z_{U_X}} "2";"4"};
{\ar^{} "4";"6"};
{\ar_{g} "2";"8"};
{\ar^{h} "4";"8"};
{\ar@{}|\circlearrowright "2";"3"};
\endxy
\]
This gives $\alpha=\underline{h}$.
\end{itemize}
By the adjointness of $\sigma$ as in Definition \ref{DefTCPAdj} {\rm (iv)}, it suffices to show the commutativity of
\[
\xy
(-9,7)*+{C_f}="0";
(10,7)*+{Z_{U_X}}="2";
(-9,-7)*+{Z}="4";
(10,-7)*+{P}="6";
{\ar^{\underline{z}_{U_X}\circ\underline{c}_f} "0";"2"};
{\ar_{\underline{e}^{\prime}} "0";"4"};
{\ar^{\underline{h}} "2";"6"};
{\ar_{\underline{m}} "4";"6"};
{\ar@{}|\circlearrowright "0";"6"};
\endxy
\]
in $\mathcal{U}/\mathcal{I}$. This follows from the equation
\begin{eqnarray*}
\gamma_P^{\prime}\circ (h\circ z_{U_X}\circ c_f)&=&\gamma_P^{\prime}\circ g\circ c_f\ =\ (z_{T_P}^{\prime}\circ x)[1]\circ\gamma_X\circ c_f\\
&=&(z_{T_P}^{\prime}\circ x)[1]\circ b_f\ =\ r\circ e\\
&=&r\circ q^{\prime}\circ e^{\prime}\ =\ \gamma_P^{\prime}\circ (m\circ e^{\prime})
\end{eqnarray*}
and the existence of the distinguished triangle $T_P\overset{\iota_P^{\prime}}{\longrightarrow}I_P^{\prime}\overset{\wp_P^{\prime}}{\longrightarrow}P\overset{\gamma_P^{\prime}}{\longrightarrow}T_P[1]$ in $\mathcal{C}$.
%\[
%\xy
%(-12,10)*+{C_f}="0";
%(12,10)*+{Z_{U_X}}="2";
%(-12,-10)*+{Z}="4";
%(2,-10)*+{P}="6";
%(12,-1)*+{P}="8";
%(12,-10)*+{T_P[1]}="10";
%%
%{\ar^{z_{U_X}\circ c_f} "0";"2"};
%{\ar_{e^{\prime}} "0";"4"};
%{\ar^{h} "2";"8"};
%{\ar_{m} "4";"6"};
%{\ar_{\gamma_P^{\prime}} "6";"10"};
%{\ar^{\gamma_P^{\prime}} "8";"10"};
%%
%{\ar@{}|\circlearrowright "0";"10"};
%\endxy
%\]
\end{proof}

\section{Condition for mutation} \label{section_Mutation}

\subsection{Condition {\rm (I)+(II)}} 
In this section, we give a generalization of the bijection between cotorsion pairs given in \cite[Theorem3.5]{ZZ1}, under suitable conditions {\rm (I),(II)} (Theorem \ref{ThmBij}). This allows us to define {\it mutation} with respect to $\mathcal{P}$ (Definition \ref{DefGeneralMut}).
\begin{dfn}\label{DefMutable}
Let $\mathcal{P}=((\mathcal{S},\mathcal{T}),(\mathcal{U},\mathcal{V}))$ be a concentric TCP.
Define the class of {\it mutable cotorsion pairs on} $\mathcal{C}$ {\it with respect to} $\mathcal{P}$ to be
\[ \mathfrak{M}_{\mathcal{P}}=\Set{ (\mathcal{A},\mathcal{B})\in\mathfrak{CP}(\mathcal{C})| \begin{array}{c}\mathcal{S}\subseteq\mathcal{A}\subseteq\mathcal{U}\\\mathcal{V}\subseteq\mathcal{B}\subseteq\mathcal{T}\end{array},
\mathrm{Ext}^1_{\mathcal{Z}/\mathcal{I}}(\sigma(\mathcal{A}/\mathcal{I}),\omega(\mathcal{B}/\mathcal{I}))=0 }. \]
Here, we define $\mathrm{Ext}^1_{\mathcal{Z}/\mathcal{I}}$ by $\mathrm{Ext}^1_{\mathcal{Z}/\mathcal{I}}(X,Y)=(\mathcal{Z}/\mathcal{I})(X,\Sigma Y)$ for any $X,Y\in\mathcal{Z}/\mathcal{I}$.
Remark that $\mathcal{S}\subseteq\mathcal{A}$ implies $\mathcal{B}\subseteq\mathcal{T}$, and $\mathcal{V}\subseteq\mathcal{B}$ implies $\mathcal{A}\subseteq\mathcal{U}$.
For another description of $\mathfrak{M}_{\mathcal{P}}$ under the assumption of the following {\rm (I)+(II)}, see Corollary \ref{CorABBij2}.
\end{dfn}

\begin{cond}\label{CondCDE}
Let $\mathcal{P}=((\mathcal{S},\mathcal{T}),(\mathcal{U},\mathcal{V}))$ be a concentric TCP. We consider the following conditions.
\begin{itemize}
\item[{\rm (I)}] For any $X\in\mathcal{T}\ast\mathcal{U}$, the morphism
\[ \mu_X\colon\sigma(\omega_{\mathcal{U}}(X))\to\omega(\sigma_{\mathcal{T}}(X)) \]
given in Proposition \ref{PropNatural} is isomorphism in $\mathcal{Z}/\mathcal{I}$.
\item[{\rm (II)}] $\mathcal{U}\cap\mathcal{N}^f=\mathcal{S}$ and $\mathcal{T}\cap\mathcal{N}^i=\mathcal{V}$ hold.
\end{itemize}
As for examples, see section \ref{section_Typical}.
\end{cond}

\begin{rem}\label{RemCDE}
The following holds for {\it any} concentric $\mathcal{P}$.
\begin{itemize}
\item For any $X$ in $\mathcal{T}\cup\mathcal{U}$, the morphism $\mu_X$ is isomorphism in $\mathcal{Z}/\mathcal{I}$.
\item $\mathcal{U}\cap\mathcal{N}^i=\mathcal{S}$ and $\mathcal{T}\cap\mathcal{N}^f=\mathcal{V}$ hold (Remark \ref{RemTCP_UNTN} {\rm (1)}).
\end{itemize}
\end{rem}

\begin{claim}\label{ClaimCot}
Let $\mathcal{P}$ be a concentric TCP satisfying {\rm (II)}. Then, $(\mathcal{S},\mathcal{T})$ is a co-$t$-structure if and only if $(\mathcal{U},\mathcal{V})$ is a co-$t$-structure.
\end{claim}
\begin{proof}
By duality, it suffices to show
\[ \mathcal{S}[-1]\subseteq\mathcal{S}\ \Rightarrow\ \mathcal{V}[1]\subseteq\mathcal{V}. \]

Suppose $\mathcal{S}[-1]\subseteq\mathcal{S}$ holds. Then $\mathcal{C}(\mathcal{S}[-1],\mathcal{V}[1])=0$ implies $\mathcal{V}[1]\subseteq\mathcal{T}$. From condition {\rm (II)}, it follows $\mathcal{V}[1]=\mathcal{T}\cap\mathcal{V}[1]\subseteq\mathcal{T}\cap\mathcal{N}^i=\mathcal{V}$.
\end{proof}

\begin{lem}\label{LemCondC1}
Let $\mathcal{P}$ be a concentric TCP satisfying {\rm (I)+(II)}.
Then, we have the following.
\begin{enumerate}
\item $(\mathcal{T}\ast\mathcal{U})\cap(\mathcal{I}\ast\mathcal{V}[1])\subseteq\mathcal{N}^f$.
\item $(\mathcal{T}\ast\mathcal{U})\cap(\mathcal{S}[-1]\ast\mathcal{I})\subseteq\mathcal{N}^i$.
\end{enumerate}
\end{lem}
\begin{proof}
{\rm (1)} For any $X\in(\mathcal{T}\ast\mathcal{U})\cap(\mathcal{I}\ast\mathcal{V}[1])$, we have $\sigma(\omega_{\mathcal{U}}(X))\cong\omega(\sigma_{\mathcal{T}}(X))$ by {\rm (I)}.
Since $X\in\mathcal{I}\ast\mathcal{V}[1]=\widetilde{\omega_{\mathcal{U}}^{-1}(0)}$, it follows $\omega(\sigma_{\mathcal{T}}(X))=0$ in $\mathcal{Z}/\mathcal{I}$, which means
\[ \sigma_{\mathcal{T}}(X)\in\omega^{-1}(0)=\overline{\mathcal{T}\cap(\mathcal{I}\ast\mathcal{V}[1])}\subseteq\overline{\mathcal{T}\cap(\mathcal{S}\ast\mathcal{V}[1])}=\overline{\mathcal{T}\cap\mathcal{N}^i}=\overline{\mathcal{V}} \]
by {\rm (II)}. Thus we have
\[ X\in\widetilde{\sigma_{\mathcal{T}}^{-1}(\overline{\mathcal{V}})}=\mathcal{S}[-1]\ast\mathcal{V}=\mathcal{N}^f. \]
{\rm (2)} can be shown dually.
\end{proof}

\begin{lem}\label{LemCondC2}
Let $\mathcal{P}$ be a concentric TCP satisfying {\rm (I)+(II)}. Then the following holds.
\begin{enumerate}
\item $\tau^+_{(\mathcal{U},\mathcal{V})}(\overline{\mathcal{T}})\subseteq\overline{\mathcal{N}^f}$.
\item $\tau^-_{(\mathcal{S},\mathcal{T})}(\overline{\mathcal{U}})\subseteq\overline{\mathcal{N}^i}$.
\end{enumerate}
\end{lem}
\begin{proof}
{\rm (1)} We use Fact \ref{FactCPAdj} {\rm (2)}. For any $T\in\mathcal{T}$, take diagram
\begin{equation}\label{2TriaT_1}
\xy
(-14,18)*+{U_T[-1]}="0";
(-14.1,0)*+{Z_T}="2";
(-3,0)*+{T}="4";
(-14,-16)*+{I_T}="6";
(2,-9)*+{M_T}="8";
(20,0)*+{V_T[1]}="10";
(-8,-8)*+_{_{\circlearrowright}}="12";
(5,-3)*+_{_{\circlearrowright}}="14";
(-10,5)*+_{_{\circlearrowright}}="14";
{\ar_{b_T} "0";"2"};
{\ar_{} "2";"6"};
{\ar^{c_T=a_T\circ b_T} "0";"4"};
{\ar_{a_T} "2";"4"};
{\ar_{} "4";"8"};
{\ar^{} "4";"10"};
{\ar_{} "6";"8"};
{\ar_{} "8";"10"};
\endxy
\end{equation}
obtained from distinguished triangles
\begin{eqnarray*}
&Z_T\overset{a_T}{\longrightarrow} T\to V_T[1]\to Z_T[1]&(Z_T\in\mathcal{Z},V_T\in\mathcal{V}),\\
&U_T[-1]\overset{b_T}{\longrightarrow} Z_T\to I_T\to U_T&(I_T\in\mathcal{I},U_T\in\mathcal{U}),\\
&U_T[-1]\overset{c_T}{\longrightarrow} T\to {}^{\exists}M_T\to U_T&
\end{eqnarray*}
by the octahedron axiom. $V_T\to I_T\to M_T\to V_T[1]$ also becomes a distinguished triangle in $\mathcal{C}$. This $M_T$ gives $\tau^+_{(\mathcal{U},\mathcal{V})}(T)\cong M_T$ in $\mathcal{C}/\mathcal{I}$. It follows $M_T\in(\mathcal{T}\ast\mathcal{U})\cap(\mathcal{I}\ast\mathcal{V}[1])\subseteq\mathcal{N}^f$ by Lemma \ref{LemCondC1}. Dually for {\rm (2)}.
\end{proof}

\begin{rem}\label{Rem35Revised}
Let $(\ref{2TriaT_1})$ as in the proof of Lemma \ref{LemCondC2}. If we decompose $M_T$ into a distinguished triangle
\[ S_{M_T}[-1]\to M_T\to V_{M_T}\to S_{M_T}\quad(S_{M_T}\in\mathcal{S},V_{M_T}\in\mathcal{V}) \]
in $\mathcal{C}$, we also have a commutative diagram in $\mathcal{C}$
\begin{equation}\label{2TriaT_2}
\xy
(-14,18)*+{S_{M_T}[-1]}="0";
(-14.1,0)*+{M_T}="2";
(-3,0)*+{U_T}="4";
(-14,-16)*+{V_{M_T}}="6";
(2,-9)*+{{}^{\exists}U_{M_T}}="8";
(20,0)*+{T[1]}="10";
(-8,-8)*+_{_{\circlearrowright}}="12";
(5,-3)*+_{_{\circlearrowright}}="14";
(-10,5)*+_{_{\circlearrowright}}="14";
{\ar_{} "0";"2"};
{\ar_{} "2";"6"};
{\ar^{} "0";"4"};
{\ar_{} "2";"4"};
{\ar_{u_T} "4";"8"};
{\ar^{c_T[1]} "4";"10"};
{\ar_{} "6";"8"};
{\ar_{d_T} "8";"10"};
\endxy
\end{equation}
by the octahedron axiom, where
\[ S_{M_T}[-1]\to U_T\overset{u_T}{\longrightarrow} U_{M_T}\to S_{M_T},\quad T\to V_{M_T}\to U_{M_T}\overset{d_T}{\longrightarrow} T[1] \]
are distinguished triangles. This $U_{M_T}$ satisfies $U_{M_T}\in U_T\ast S_{M_T}\subseteq\mathcal{U}$.
\end{rem}

The following can be shown without assuming condition {\rm (I)}. We omit its details, since we do not use this result in this article.
\begin{rem}\label{RemCond2}
If a concentric $\mathcal{P}$ satisfies {\rm (1),(2)} in Lemma \ref{LemCondC2}, then the following holds.
\begin{enumerate}
\item $\mathcal{P}$ satisfies {\rm (II)}.
\item $\Omega\circ\Sigma\cong\mathrm{Id}_{\mathcal{Z}/\mathcal{I}}$, $\Sigma\circ\Omega\cong\mathrm{Id}_{\mathcal{Z}/\mathcal{I}}$.
\end{enumerate}
\end{rem}

\begin{prop}\label{PropShiftCompatT}
Let $\mathcal{P}$ be a concentric TCP satisfying {\rm (I)+(II)}.
For any $T\in\mathcal{T}$, the diagrams $(\ref{2TriaT_1}),(\ref{2TriaT_2})$ induce the isomorphism
\[ \sigma(\underline{u}_T)\colon\Sigma(\omega T)\overset{\cong}{\longrightarrow}\sigma(\omega_{\mathcal{U}}(T[1])) \]
in $\mathcal{Z}/\mathcal{I}$. This isomorphism is natural in $T$. Namely, the following diagram is commutative up to natural isomorphism.
\[
\xy
(-15,15)*+{\mathcal{T}}="0";
(0,15)*+{\mathcal{T}[1]}="2";
(15,15)*+{\mathcal{C}}="4";
(-15,0)*+{\mathcal{T}/\mathcal{I}}="6";
(-15,-15)*+{\mathcal{Z}/\mathcal{I}}="8";
(15,5)*+{\mathcal{C}/\mathcal{I}}="10";
(15,-5)*+{\mathcal{U}/\mathcal{I}}="12";
(15,-15)*+{\mathcal{Z}/\mathcal{I}}="14";
{\ar^(0.46){[1]} "0";"2"};
{\ar@{^(->} "2";"4"};
{\ar_{p_{\mathcal{T}}} "0";"6"};
{\ar_{\omega} "6";"8"};
{\ar_{} "4";"10"};
{\ar^{\omega_{\mathcal{U}}} "10";"12"};
{\ar^{\sigma} "12";"14"};
{\ar_{\Sigma} "8";"14"};
{\ar@{}|\circlearrowright "0";"14"};
\endxy
\]
Here, $p_{\mathcal{T}}$ denotes the residue functor.
Dually, $\omega(\sigma_{\mathcal{T}}(U[-1]))\cong\Omega(\sigma U)$ holds for any $U\in\mathcal{U}$.
\end{prop}
\begin{proof}
From diagram $(\ref{2TriaT_1})$, we have
\[ U_T\cong Z_T\langle1\rangle\cong(\omega T)\langle1\rangle \]
in $\mathcal{U}/\mathcal{I}$.
Applying Lemma \ref{LemSUU} to the distinguished triangle
\[ S_{M_T}[-1]\to U_T\overset{u_T}{\longrightarrow}U_{M_T}\to S_{M_T} \]
in $(\ref{2TriaT_2})$, it follows $\sigma(\underline{u}_T)$ is isomorphism in $\mathcal{Z}/\mathcal{I}$.

Since distinguished triangle
\[ U_{M_T}\to T[1]\to V_{M_T}[1]\to U_{M_T}[1] \]
gives $\omega_{\mathcal{U}}(T[1])\cong U_{M_T}$ in $\mathcal{U}/\mathcal{I}$, we obtain
\[  \Sigma(\omega T)=\sigma((\omega T)\langle1\rangle)\cong\sigma(U_T)\underset{\sigma(\underline{u}_T)}{\overset{\cong}{\longrightarrow}}\sigma(U_{M_T})\cong\sigma(\omega_{\mathcal{U}}(T[1])). \]

As for the naturality, since the other parts are obvious, it suffices to show that $\underline{u}_T\colon U_T\to U_{M_T}$ is natural in $T\in\mathcal{T}$.
Let $f\in\mathcal{T}(T,T^{\prime})$ be any morphism in $\mathcal{T}$, and take diagrams $(\ref{2TriaT_1}),(\ref{2TriaT_2})$ similarly for $T^{\prime}$. By the definition of $\langle1\rangle$ and $\omega$, the morphism $(\omega\underline{f})\langle1\rangle$ corresponds to $\underline{g}\colon U_T\to U_{T^{\prime}}$, where $g$ is a morphism given by the following morphisms of triangles in $\mathcal{C}$.
\[
\xy
(-23,7)*+{V_T}="0";
(-7,7)*+{Z_T}="2";
(8,7)*+{T}="4";
(23,7)*+{V_T[1]}="6";
(-23,-7)*+{V_{T^{\prime}}}="10";
(-7,-7)*+{Z_{T^{\prime}}}="12";
(8,-7)*+{T^{\prime}}="14";
(23,-7)*+{V_{T^{\prime}}[1]}="16";
{\ar^{} "0";"2"};
{\ar^{a_T} "2";"4"};
{\ar^{} "4";"6"};
{\ar_{} "0";"10"};
{\ar_{{}^{\exists}z} "2";"12"};
{\ar^{f} "4";"14"};
{\ar^{} "6";"16"};
{\ar_{} "10";"12"};
{\ar_{a_{T^{\prime}}} "12";"14"};
{\ar_{} "14";"16"};
{\ar@{}|\circlearrowright "0";"12"};
{\ar@{}|\circlearrowright "2";"14"};
{\ar@{}|\circlearrowright "4";"16"};
\endxy
\ \ ,\ \ 
\xy
(-23,7)*+{U_T[-1]}="0";
(-7,7)*+{Z_T}="2";
(8,7)*+{I_T}="4";
(23,7)*+{U_T}="6";
(-23,-7)*+{U_{T^{\prime}}[-1]}="10";
(-7,-7)*+{Z_{T^{\prime}}}="12";
(8,-7)*+{I_{T^{\prime}}}="14";
(23,-7)*+{U_{T^{\prime}}}="16";
{\ar^(0.6){b_T} "0";"2"};
{\ar^{} "2";"4"};
{\ar^{} "4";"6"};
{\ar_{g[-1]} "0";"10"};
{\ar_{z} "2";"12"};
{\ar^{} "4";"14"};
{\ar^{g} "6";"16"};
{\ar_(0.6){b_{T^{\prime}}} "10";"12"};
{\ar_{} "12";"14"};
{\ar_{} "14";"16"};
{\ar@{}|\circlearrowright "0";"12"};
{\ar@{}|\circlearrowright "2";"14"};
{\ar@{}|\circlearrowright "4";"16"};
\endxy
\]
On the other hand, by the definition of $\omega_{\mathcal{U}}$, the morphism $\omega_{\mathcal{U}}(\underline{f[1]})$ corresponds to $\underline{h}\colon U_{M_T}\to U_{M_{T^{\prime}}}$, where $h$ is a morphism appearing in the following morphism of triangles in $\mathcal{C}$.
\[
\xy
(-23,7)*+{V_{M_T}}="0";
(-7,7)*+{U_{M_T}}="2";
(8,7)*+{T[1]}="4";
(25,7)*+{V_{M_T}[1]}="6";
(-23,-7)*+{V_{M_{T^{\prime}}}}="10";
(-7,-7)*+{U_{M_{T^{\prime}}}}="12";
(8,-7)*+{T^{\prime} [1]}="14";
(25,-7)*+{V_{M_{T^{\prime}}}[1]}="16";
{\ar^{} "0";"2"};
{\ar^{d_T} "2";"4"};
{\ar^{} "4";"6"};
{\ar_{} "0";"10"};
{\ar_{h} "2";"12"};
{\ar^{f[1]} "4";"14"};
{\ar^{} "6";"16"};
{\ar_{} "10";"12"};
{\ar_{d_{T^{\prime}}} "12";"14"};
{\ar_{} "14";"16"};
{\ar@{}|\circlearrowright "0";"12"};
{\ar@{}|\circlearrowright "2";"14"};
{\ar@{}|\circlearrowright "4";"16"};
\endxy
\]
Thus we have
\begin{eqnarray*}
d_{T^{\prime}}\circ(h\circ u_T-u_{T^{\prime}}\circ g)&=&f[1]\circ d_T\circ u_T-a_{T^{\prime}}[1]\circ b_{T^{\prime}}[1]\circ g\\
&=&f[1]\circ c_T[1]-f[1]\circ a_T[1]\circ b_T[1]\ =\ 0,
\end{eqnarray*}
which implies that $h\circ u_T-u_{T^{\prime}}\circ g$ factors through $V_{M_{T^{\prime}}}$. Since $(\mathcal{C}/\mathcal{I})(U_T,V_{M_{T^{\prime}}})=0$ by $V_{M_{T^{\prime}}}\in\mathcal{V}\subseteq\mathcal{I}\ast\mathcal{V}[1]$, this means that
\[
\xy
(-7,6)*+{U_T}="0";
(7,6)*+{U_{M_T}}="2";
(-7,-6)*+{U_{T^{\prime}}}="4";
(7,-6)*+{U_{M_{T^{\prime}}}}="6";
{\ar^{\underline{u}_T} "0";"2"};
{\ar_{\underline{g}} "0";"4"};
{\ar^{\underline{h}} "2";"6"};
{\ar_{\underline{u}_{T^{\prime}}} "4";"6"};
{\ar@{}|\circlearrowright "0";"6"};
\endxy
\]
is commutative in $\mathcal{U}/\mathcal{I}$. This shows the naturality.
\end{proof}

\begin{cor}\label{CorRevise}
Let $\mathcal{P}$ be a concentric TCP satisfying {\rm (I)+(II)}. Let $\Gamma\colon\mathcal{T}\to\mathcal{U}/\mathcal{I}$ be the composition of functors
\[ \mathcal{T}\overset{[1]}{\longrightarrow}\mathcal{T} [1]\hookrightarrow\mathcal{C}\to\mathcal{C}/\mathcal{I}\overset{\omega_{\mathcal{U}}}{\longrightarrow}\mathcal{U}/\mathcal{I}. \]
Then $\Gamma$ induces a functor $\overline{\Gamma}\colon\mathcal{T}/\mathcal{I}\to\mathcal{U}/\mathcal{I}$ which makes the following diagram of functors commutative up to natural isomorphisms.
\begin{equation}\label{DiagRevise}
\xy
(-9,12)*+{\mathcal{T}}="-2";
(-9,0)*+{\mathcal{T}/\mathcal{I}}="0";
(4,10)*+{}="1";
(9,0)*+{\mathcal{U}/\mathcal{I}}="2";
(-9,-12)*+{\mathcal{Z}/\mathcal{I}}="4";
(9,-12)*+{\mathcal{Z}/\mathcal{I}}="6";
{\ar_{p_{\mathcal{T}}} "-2";"0"};
{\ar@/^0.40pc/^{\Gamma} "-2";"2"};
{\ar^{\overline{\Gamma}} "0";"2"};
{\ar_{\omega} "0";"4"};
{\ar^{\sigma} "2";"6"};
{\ar_{\Sigma} "4";"6"};
{\ar@{}|\circlearrowright "0";"1"};
{\ar@{}|\circlearrowright "0";"6"};
\endxy
\end{equation}
\end{cor}
\begin{proof}
By definition, $\omega_{\mathcal{U}}(T[1])$ is given by $U$ appearing in a distinguished triangle
\[ T\to V\to U\to T[1]\quad(U\in\mathcal{U},V\in\mathcal{V}) \]
for each $T\in\mathcal{T}$. If $T\in\mathcal{I}$, then the morphism $U\to T[1]$ is zero, and $U$ becomes a direct summand of $V$, which implies $U\in\mathcal{I}$. Thus the additive functor $\Gamma$ satisfies $\Gamma(\mathcal{I})=0$. This shows the existence of $\overline{\Gamma}$ which makes the upper half of $(\ref{DiagRevise})$ commutative. The lower half is obtained by Proposition \ref{PropShiftCompatT} and the property of the residue functor $p_{\mathcal{T}}\colon\mathcal{T}\to\mathcal{T}/\mathcal{I}$.
\end{proof}

\begin{cor}\label{CorSigmaOmega}
Let $\mathcal{P}$ be a concentric TCP satisfying {\rm (I)+(II)}. Then $\mathcal{Z}/\mathcal{I}$ is triangulated. In particular, we have $\mathrm{Ext}^1_{\mathcal{Z}/\mathcal{I}}(X,Y)=(\mathcal{Z}/\mathcal{I})(X,\Sigma Y)\cong(\mathcal{Z}/\mathcal{I})(\Omega X,Y)$ for any $X,Y\in\mathcal{Z}/\mathcal{I}$.
\end{cor}
\begin{proof}
By Theorem \ref{ThmPreTriangulated}, it remains to show that $\Sigma$  and $\Omega$ are quasi-inverses to each other.

Let us show $\Sigma\circ\Omega\cong\mathrm{Id}$. Remark that $Z\langle-1\rangle$ is given by $T$ appearing in a distinguished triangle
\[ T\to I\to Z\to T[1]\quad(I\in\mathcal{I},T\in\mathcal{T}), \]
for each $Z\in\mathcal{Z}/\mathcal{I}$. This in turn gives $Z\cong\omega_{\mathcal{U}}(T[1])$, and we obtain an isomorphism $\overline{\Gamma}(Z\langle-1\rangle)\cong Z$ in $\mathcal{U}/\mathcal{I}$, which is natural in $Z\in\mathcal{Z}/\mathcal{I}$. Thus by Corollary \ref{CorRevise} we have the following diagram of functors, commutative up to natural isomorphisms.
\[
\xy
(-22,13)*+{\mathcal{Z}/\mathcal{I}}="-2";
(22,13)*+{\mathcal{Z}/\mathcal{I}}="-4";
(-20,-3)*+{}="-1";
(-9,1)*+{\mathcal{T}/\mathcal{I}}="0";
(9,1)*+{\mathcal{U}/\mathcal{I}}="2";
(20,-3)*+{}="3";
(-9,-14)*+{\mathcal{Z}/\mathcal{I}}="4";
(9,-14)*+{\mathcal{Z}/\mathcal{I}}="6";
(0,12)*+{}="10";
(0,7)*+{}="11";
{\ar@/^0.40pc/^{\mathrm{Id}} "-2";"-4"};
{\ar^{\langle-1\rangle} "-2";"0"};
{\ar@/_0.80pc/_{\Omega} "-2";"4"};
{\ar@/^0.80pc/^{\mathrm{Id}} "-4";"6"};
{\ar^{\overline{\Gamma}} "0";"2"};
{\ar@{^(->}_{} "-4";"2"};
{\ar_{\omega} "0";"4"};
{\ar^{\sigma} "2";"6"};
{\ar_{\Sigma} "4";"6"};
{\ar@{}|\circlearrowright "10";"11"};
{\ar@{}|\circlearrowright "-1";"0"};
{\ar@{}|\circlearrowright "2";"3"};
{\ar@{}|\circlearrowright "0";"6"};
\endxy
\]
This shows $\Sigma\circ\Omega\cong\mathrm{Id}$. Dually for $\Omega\circ\Sigma\cong\mathrm{Id}$.
\end{proof}

\subsection{Reduction-bijection}

In this section, we will give a bijection between $\mathfrak{M}_{\mathcal{P}}$ and $\mathfrak{CP}(\mathcal{Z}/\mathcal{I})$, when $\mathcal{P}$ satisfies {\rm (I)+(II)}. Its construction involves the following functors.
\begin{dfn}\label{DefTildes}
Let $\mathcal{P}$ be a concentric TCP. Put
\[ \overline{\sigma}=\sigma\circ p_{\mathcal{U}}\colon\mathcal{U}\to\mathcal{Z}/\mathcal{I},\quad\overline{\omega}=\omega\circ p_{\mathcal{T}}\colon\mathcal{T}\to\mathcal{Z}/\mathcal{I}. \]
Here, $p_{\mathcal{U}}\colon\mathcal{U}\to\mathcal{U}/\mathcal{I}$ and $p_{\mathcal{T}}\colon\mathcal{T}\to\mathcal{T}/\mathcal{I}$ are the residue functors.
\end{dfn}

Let us begin with some lemmas.
\begin{lem}\label{LemExtInj}
Let $\mathcal{P}$ be a concentric TCP satisfying {\rm (I)+(II)}.
Let $\mathcal{B}\subseteq\mathcal{C}$ be a full additive subcategory, closed under isomorphisms, direct summands and extension, satisfying $\mathcal{V}\subseteq\mathcal{B}\subseteq\mathcal{T}$ (for example, $\mathcal{B}=\mathcal{V}$ or $\mathcal{B}=\mathcal{T}$).

Then for any $B\in\mathcal{B}$, there exists a distinguished triangle
\begin{equation}\label{Tria_ExtInj}
B_0\to R\to N\to B_0[1]
\end{equation}
with  $B_0\in\mathcal{B}, N\in\mathcal{N}^i$ and $R\in\mathcal{T}\cap(\mathcal{B}\ast\mathcal{V}[1])$, which satisfies the following for any $U\in\mathcal{U}$.
\begin{enumerate}
\item $\mathrm{Ext}^1_{\mathcal{C}}(U,B)\cong\mathrm{Ext}^1_{\mathcal{C}}(U,B_0)$.
\item $\mathrm{Ext}^1_{\mathcal{Z}/\mathcal{I}}(\sigma(U),\omega(B))\cong\mathrm{Ext}^1_{\mathcal{C}}(U,R)$.
\end{enumerate}
\end{lem}
\begin{proof}
Applying Lemma \ref{LemCondC2} and Remark \ref{Rem35Revised} to $B\in\mathcal{B}\subseteq\mathcal{T}$, we obtain diagrams as in $(\ref{2TriaT_1})$ and $(\ref{2TriaT_2})$:
\begin{eqnarray*}
&\xy
(-14,18)*+{U_B[-1]}="0";
(-14.1,0)*+{Z_B}="2";
(-3,0)*+{B}="4";
(-14,-16)*+{I_B}="6";
(2,-9)*+{M_B}="8";
(20,0)*+{V_B[1]}="10";
(-8,-8)*+_{_{\circlearrowright}}="12";
(5,-3)*+_{_{\circlearrowright}}="14";
(-10,5)*+_{_{\circlearrowright}}="14";
{\ar_{} "0";"2"};
{\ar_{} "2";"6"};
{\ar^{} "0";"4"};
{\ar_{} "2";"4"};
{\ar_{} "4";"8"};
{\ar^{} "4";"10"};
{\ar_{} "6";"8"};
{\ar_{} "8";"10"};
\endxy
\quad,\quad
\xy
(-14,18)*+{S_{M_B}[-1]}="0";
(-14.1,0)*+{M_B}="2";
(-3,0)*+{U_B}="4";
(-14,-16)*+{V_{M_B}}="6";
(2,-9)*+{U_{M_B}}="8";
(20,0)*+{B[1]}="10";
(-8,-8)*+_{_{\circlearrowright}}="12";
(5,-3)*+_{_{\circlearrowright}}="14";
(-10,5)*+_{_{\circlearrowright}}="14";
{\ar_{} "0";"2"};
{\ar_{} "2";"6"};
{\ar^{} "0";"4"};
{\ar_{} "2";"4"};
{\ar_{u_B} "4";"8"};
{\ar^{} "4";"10"};
{\ar_{} "6";"8"};
{\ar_{} "8";"10"};
\endxy&\\
&(Z_B\in\mathcal{Z},\, V_B,V_{M_B}\in\mathcal{V},\, I_B\in\mathcal{I},\, U_B,U_{M_B}\in\mathcal{U},\, S_{M_B}\in\mathcal{S})&
\end{eqnarray*}

Since $V_{M_B}\in\mathcal{V}\subseteq\mathcal{I}\ast\mathcal{V}[1]$, we can decompose $V_{M_B}$ into a distinguished triangle
\[ I_0\to V_{M_B}\to V_0[1]\to I_0[1]\quad(I_0\in\mathcal{I},V_0\in\mathcal{V}) \]
in $\mathcal{C}$.
By the octahedron axiom, we have the following commutative diagram made of distinguished triangles
\begin{equation}\label{Octa_ExtInj0}
\xy
(-14,18)*+{U_{M_B}}="0";
(-14.1,0)*+{B_0[1]}="2";
(-3,0)*+{B[1]}="4";
(-14,-16)*+{I_0[1]}="6";
(2,-9)*+{V_{M_B}[1]}="8";
(20,0)*+{V_0[2]}="10";
(-8,-8)*+_{_{\circlearrowright}}="12";
(5,-3)*+_{_{\circlearrowright}}="14";
(-10,5)*+_{_{\circlearrowright}}="14";
{\ar_{} "0";"2"};
{\ar_{} "2";"6"};
{\ar^{} "0";"4"};
{\ar_{} "2";"4"};
{\ar_{} "4";"8"};
{\ar^{} "4";"10"};
{\ar_{} "6";"8"};
{\ar_{} "8";"10"};
\endxy
\end{equation}
for some $B_0\in V_0\ast B\subseteq\mathcal{B}\ast\mathcal{B}\subseteq\mathcal{B}$. This satisfies
\[ \omega_{\mathcal{U}}(B_0[1])\cong U_{M_B}\cong\omega_{\mathcal{U}}(B[1]) \]
in $\mathcal{U}/\mathcal{I}$. Moreover, for any $U\in\mathcal{U}$, if we apply $\mathcal{C}(U,-)$ to $(\ref{Octa_ExtInj0})$, we obtain an exact sequence
\[ 0\to\mathcal{C}(U,B_0[1])\to\mathcal{C}(U,B[1])\overset{0}{\longrightarrow}\mathcal{C}(U,V_0[2]). \]
This shows {\rm (1)}.

Let us construct $(\ref{Tria_ExtInj})$.
Applying Lemma \ref{LemCondC2} {\rm (2)} to $U_{M_B}$, we obtain the following commutative diagram made of distinguished triangles in $\mathcal{C}$
\begin{equation}\label{Octa_ExtInj1}
\xy
(-20,0)*+{S^{\prime}[-1]}="0";
(-1,8)*+{N}="2";
(16,15)*+{I^{\prime}}="4";
(3,0)*+{U_{M_B}}="6";
(16.1,0)*+{Z^{\prime}}="8";
(16,-19)*+{T^{\prime}[1]}="10";
(9,6)*+_{_{\circlearrowright}}="12";
(-5,3)*+_{_{\circlearrowright}}="14";
(12,-6)*+_{_{\circlearrowright}}="14";
{\ar^{} "0";"2"};
{\ar^{} "2";"4"};
{\ar_{} "0";"6"};
{\ar^{} "2";"6"};
{\ar^{} "6";"8"};
{\ar_{} "6";"10"};
{\ar^{} "4";"8"};
{\ar^{} "8";"10"};
\endxy
\end{equation}
with $S^{\prime}\in\mathcal{S},\, Z^{\prime}\in\mathcal{Z},\, T^{\prime}\in\mathcal{T},\, I^{\prime}\in\mathcal{I},\, N\in\mathcal{N}^i$, dually to $(\ref{2TriaT_1})$. Remark that this gives $\sigma(U_{M_B})\cong Z^{\prime}$ in $\mathcal{Z}/\mathcal{I}$. By the octahedron axiom and $\mathcal{C}(I_0,T^{\prime}[1])=0$, we obtain
\begin{equation}\label{Octa_ExtInj2}
\xy
(14,18)*+{I_0}="0";
(14.1,0)*+{T^{\prime}[1]}="2";
(14,-16)*+{{}^{\exists}R[1]}="4";
(3,0)*+{U_{M_B}}="6";
(-2,-9)*+{B_0[1]}="8";
(-20,0)*+{N}="10";
(8,-8)*+_{_{\circlearrowright}}="12";
(-5,-3)*+_{_{\circlearrowright}}="14";
(10,5)*+_{_{\circlearrowright}}="14";
{\ar^{0} "0";"2"};
{\ar^{} "2";"4"};
{\ar_{} "0";"6"};
{\ar_{} "6";"2"};
{\ar^{} "6";"8"};
{\ar^{} "10";"6"};
{\ar_{} "8";"4"};
{\ar_{} "10";"8"};
\endxy
\end{equation}
with $R\in\mathcal{T}\cap(\mathcal{B}\ast\mathcal{N}^i)=\mathcal{T}\cap(\mathcal{B}\ast\mathcal{V}[1])$, by Proposition \ref{PropInverseImage}.

It remains to show {\rm (2)}. Let $U\in\mathcal{U}$ be any object.
Applying $\mathcal{C}(U,-)$ to $(\ref{Octa_ExtInj1})$ and $(\ref{Octa_ExtInj2})$, we obtain exact sequences
\begin{eqnarray*}
&\mathcal{C}(U,I^{\prime})\to\mathcal{C}(U,Z^{\prime})\to\mathcal{C}(U,T^{\prime}[1])\to0,&\\
&\mathcal{C}(U,I_0)\overset{0}{\longrightarrow}\mathcal{C}(U,T^{\prime}[1])\to\mathcal{C}(U,R[1])\to0,&
\end{eqnarray*}
which show
\begin{eqnarray}\label{Eq_ExtInj1}
\mathrm{Ext}^1_{\mathcal{C}}(U,R)&=&\mathcal{C}(U,R[1])\ \cong\ \mathcal{C}(U,T^{\prime}[1])\\
&\cong&\mathcal{C}(U,Z^{\prime})/\mathcal{C}(U,I^{\prime})\ \cong\ (\mathcal{U}/\mathcal{I})(U,Z^{\prime}).\nonumber
\end{eqnarray}
Since $Z^{\prime}\cong\sigma(U_{M_B})\cong\sigma(\omega_{\mathcal{U}}(B[1]))\cong\Sigma(\omega B)$ holds in $\mathcal{Z}/\mathcal{I}$ by Proposition \ref{PropShiftCompatT}, we obtain
\begin{eqnarray}\label{Eq_ExtInj2}
(\mathcal{U}/\mathcal{I})(U,Z^{\prime})&\cong&(\mathcal{U}/\mathcal{I})(U,\Sigma(\omega B))\\
&\cong&(\mathcal{Z}/\mathcal{I})(\sigma(U),\Sigma(\omega B))\ =\ \mathrm{Ext}^1_{\mathcal{Z}/\mathcal{I}}(\sigma(U),\omega(B)).\nonumber
\end{eqnarray}
Now {\rm (2)} follows from $(\ref{Eq_ExtInj1})$ and $(\ref{Eq_ExtInj2})$.
\end{proof}

\begin{prop}\label{PropExtInj}
Let $\mathcal{P}$ be a concentric TCP satisfying {\rm (I)+(II)}. For any $U\in\mathcal{U}$ and $T\in\mathcal{T}$, there is a monomorphism
\[ \mathrm{Ext}^1_{\mathcal{C}}(U,T)\hookrightarrow\mathrm{Ext}^1_{\mathcal{Z}/\mathcal{I}}(\sigma(U),\omega(T)). \]
\end{prop}
\begin{proof}
Apply Lemma \ref{LemExtInj} to $\mathcal{B}=\mathcal{T}$ and $B=T\in\mathcal{T}$, to obtain a distinguished triangle
\begin{equation}\label{Tria_PropExtInj}
T_0\to R\to N\overset{n}{\longrightarrow}T_0[1]\quad(T_0\in\mathcal{T},\, N\in\mathcal{N}^i)
\end{equation}
satisfying $\mathrm{Ext}^1_{\mathcal{C}}(U,T_0)\cong\mathrm{Ext}^1_{\mathcal{C}}(U,T)$ and $\mathrm{Ext}^1_{\mathcal{C}}(U,R)\cong\mathrm{Ext}^1_{\mathcal{Z}/\mathcal{I}}(\sigma(U),\omega(T))$.

Applying $\mathcal{C}(U,-)$ to $(\ref{Tria_PropExtInj})$, we obtain an exact sequence
\[ \mathcal{C}(U,N)\overset{\mathcal{C}(U,n)}{\longrightarrow}\mathcal{C}(U,T_0[1])\to\mathcal{C}(U,R[1]). \]
Since the image of $\mathcal{C}(U,n)$ is $0$ by Remark \ref{RemTCP_UNTN}, it follows $\mathcal{C}(U,T_0[1])\to\mathcal{C}(U,R[1])$ is monomorphic. This gives $\mathrm{Ext}^1_{\mathcal{C}}(U,T)\hookrightarrow\mathrm{Ext}^1_{\mathcal{Z}/\mathcal{I}}(\sigma(U),\omega(T))$.
\end{proof}

\begin{prop}\label{PropABtoLR}
Let $\mathcal{P}$ be a concentric TCP satisfying {\rm (I)+(II)}. Let $\mathcal{S}\subseteq\mathcal{A}\subseteq\mathcal{U}$ and $\mathcal{V}\subseteq\mathcal{B}\subseteq\mathcal{T}$ be full additive subcategories closed under isomorphisms and direct summands.
If we define $\mathscr{L},\mathscr{R}\subseteq\mathcal{Z}/\mathcal{I}$ by
\[ \mathscr{L}=\overline{\sigma}(\mathcal{A})=\sigma(\overline{\mathcal{A}}) \quad\text{and}\quad \mathscr{R}=\overline{\omega}(\mathcal{B})=\omega(\overline{\mathcal{B}}), \]
then $\mathcal{C}=\mathcal{A}\ast\mathcal{B}[1]$ implies $\mathcal{Z}/\mathcal{I}=\mathscr{L}\ast\Sigma\mathscr{R}$.
\end{prop}
\begin{proof}
If $\mathcal{C}=\mathcal{A}\ast\mathcal{B}[1]$, then in particular for any $Z\in\mathcal{Z}$, there is a distinguished triangle in $\mathcal{C}$
\[ A\overset{a}{\longrightarrow}Z\to B[1]\to A[1]\quad(A\in\mathcal{A},B\in\mathcal{B}). \]
Decompose $B[1]$ into a distinguished triangle in $\mathcal{C}$
\[ U_B\to B[1]\to V_B[1]\to U_B[1]\quad(U_B\in\mathcal{U},V_B\in\mathcal{V}), \]
which gives $U_B\cong\omega_{\mathcal{U}}(B[1])$ in $\mathcal{U}/\mathcal{I}$.
By the octahedron axiom and $\mathcal{C}(Z,V_B[1])=0$, we have
\[
\xy
(-8,23)*+{V_B}="-12";
(8,23)*+{V_B}="-14";
(-24,8)*+{A}="0";
(-8,8)*+{Z\oplus V_B}="2";
(8,8)*+{U_B}="4";
(24,8)*+{A[1]}="6";
(-24,-7)*+{A}="10";
(-8,-7)*+{Z}="12";
(8,-7)*+{B[1]}="14";
(24,-7)*+{A[1]}="16";
(-8,-22)*+{V_B[1]}="22";
(8,-22)*+{V_B[1]}="24";
{\ar@{=} "-12";"-14"};
{\ar_{} "-12";"2"};
{\ar^{} "-14";"4"};
{\ar^{\spmatrix{a}{g}} "0";"2"};
{\ar^{} "2";"4"};
{\ar^{} "4";"6"};
{\ar@{=} "0";"10"};
{\ar_{[1\ 0]} "2";"12"};
{\ar^{} "4";"14"};
{\ar@{=} "6";"16"};
{\ar_{a} "10";"12"};
{\ar^{} "12";"14"};
{\ar^{} "14";"16"};
{\ar_{0} "12";"22"};
{\ar^{} "14";"24"};
{\ar@{=} "22";"24"};
{\ar@{}|\circlearrowright "-12";"4"};
{\ar@{}|\circlearrowright "0";"12"};
{\ar@{}|\circlearrowright "2";"14"};
{\ar@{}|\circlearrowright "4";"16"};
{\ar@{}|\circlearrowright "12";"24"};
\endxy
\]
for some $g\in\mathcal{C}(A,V_B)$.
Thus $Z\oplus V_B\in A\ast U_B\subseteq\mathcal{U}$, which shows $V_B\in\mathcal{V}\cap\mathcal{U}=\mathcal{I}$.
Applying Lemma \ref{LemUUUTria} to
\[ A\to Z\oplus V_B\to U_B\to A[1], \]
we obtain a distinguished right triangle
\[ \sigma(A)\to Z\to\sigma(U_B)\to\Sigma(\sigma A). \]
This satisfies $\sigma (A)\in\mathscr{L}$, and $\sigma(U_B)\cong\sigma(\omega_{\mathcal{U}}(B[1]))\cong\Sigma(\omega B)\in\Sigma\mathscr{R}$ by Proposition \ref{PropShiftCompatT}.
\end{proof}

\begin{prop}\label{PropLRtoAB}
Let $\mathcal{P}$ be a concentric TCP satisfying {\rm (I)+(II)}.
Let $\mathscr{L},\mathscr{R}\subseteq\mathcal{Z}/\mathcal{I}$ be full additive subcategories, closed under isomorphisms and direct summands.
If we define $\mathcal{A}\subseteq\mathcal{U}$ and $\mathcal{B}\subseteq\mathcal{T}$ by
\[ \mathcal{A}=\overline{\sigma}^{-1}(\mathscr{L})=\widetilde{\sigma^{-1}(\mathscr{L})}\quad\text{and}\quad \mathcal{B}=\overline{\omega}^{-1}(\mathscr{R})=\widetilde{\omega^{-1}(\mathscr{R})}, \]
then $\mathcal{Z}/\mathcal{I}=\mathscr{L}\ast\Sigma\mathscr{R}$ implies $\mathcal{C}=\mathcal{A}\ast\mathcal{B}[1]$.
\end{prop}
\begin{proof}
Put $\mathcal{L}=\widetilde{\mathscr{L}}, \mathcal{R}=\widetilde{\mathscr{R}}$ as in Remark \ref{RemQuot2}. We proceed in the following 3 steps.
\begin{itemize}
\item[{\rm (i)}] $\mathcal{Z}/\mathcal{I}=\mathscr{L}\ast\Sigma\mathscr{R}\ \Rightarrow\ \mathcal{Z}\subseteq\mathcal{L}\ast\mathcal{B}[1]$.
\item[{\rm (ii)}] $\mathcal{Z}\subseteq\mathcal{L}\ast\mathcal{B}[1]\ \Rightarrow\ \mathcal{U}\subseteq\mathcal{A}\ast\mathcal{B}[1]$.
\item[{\rm (iii)}] $\mathcal{U}\subseteq\mathcal{A}\ast\mathcal{B}[1]\ \Rightarrow\ \mathcal{C}=\mathcal{A}\ast\mathcal{B}[1]$.
\end{itemize}

\bigskip

{\rm (i)} By $\mathcal{Z}/\mathcal{I}=\mathscr{L}\ast\Sigma\mathscr{R}$, for any $Z\in\mathcal{Z}$, there exists a distinguished triangle in $\mathcal{Z}/\mathcal{I}$
\begin{equation}\label{LZSR}
L\overset{\alpha}{\longrightarrow} Z\to\Sigma R\to\Sigma L
\end{equation}
for some $L\in\mathcal{L}$ and $R\in\mathcal{R}$. Take $f\in\mathcal{Z}(L,Z)$ satisfying $\alpha=\underline{f}$, take a distinguished triangle
\[ L\overset{\iota}{\longrightarrow}I\overset{\wp}{\longrightarrow}U_L\overset{\gamma}{\longrightarrow}L[1]\quad(I\in\mathcal{I},\, U_L\in\mathcal{U}) \]
in $\mathcal{C}$. As in Lemma \ref{LemInvariantStan} {\rm (1)}, complete $\begin{bmatrix}f\\ \iota\end{bmatrix}$ to obtain a $\mathcal{U}$-conic triangle
\begin{equation}\label{CanonL}
L\overset{\scriptstyle\begin{bmatrix}f\\ \iota\end{bmatrix}}{\longrightarrow}Z\oplus I\overset{a}{\longrightarrow}U\overset{b}{\longrightarrow}L[1]
\end{equation}
in $\mathcal{C}$. Since $\mathcal{Z}/\mathcal{I}$ is triangulated, the standard right triangle associated to $(\ref{CanonL})$ is isomorphic to $(\ref{LZSR})$. In particular, we have $\sigma(U)\cong\Sigma R\in\Sigma\mathscr{R}$.

By $\Omega\circ\Sigma\cong\mathrm{Id}_{\mathcal{Z}/\mathcal{I}}$ (Corollary \ref{CorSigmaOmega}) and the dual of Proposition \ref{PropShiftCompatT}, we have
\[ \omega(\sigma_{\mathcal{T}}(U[-1]))\cong\Omega(\sigma(U))\in\Omega\Sigma\mathscr{R}=\mathscr{R}. \]

Namely, if we take a distinguished triangle in $\mathcal{C}$
\[ S_U[-1]\to U[-1]\to T_U\to S_U\quad(S_U\in\mathcal{S},T_U\in\mathcal{T}), \]
which gives $\sigma_{\mathcal{T}}(U[-1])\cong T_U$ in $\mathcal{T}/\mathcal{I}$, then this $T_U$ satisfies $T_U\in\overline{\omega}^{-1}(\mathscr{R})=\mathcal{B}$.

By the octahedron axiom and $\mathcal{C}(S_U,L[1])=0$, we obtain a commutative diagram made of distinguished triangles in $\mathcal{C}$
\[
\xy
(-8,23)*+{T_U}="-12";
(8,23)*+{T_U}="-14";
(-24,8)*+{L}="0";
(-8,8)*+{L\oplus S_U}="2";
(8,8)*+{S_U}="4";
(24,8)*+{L[1]}="6";
(-24,-7)*+{L}="10";
(-8,-7)*+{Z\oplus I}="12";
(8,-7)*+{U}="14";
(24,-7)*+{L[1]}="16";
(-8,-22)*+{T_U[1]}="22";
(8,-22)*+{T_U[1]}="24";
{\ar@{=} "-12";"-14"};
{\ar_{} "-12";"2"};
{\ar^{} "-14";"4"};
{\ar^{} "0";"2"};
{\ar^{} "2";"4"};
{\ar^{0} "4";"6"};
{\ar@{=} "0";"10"};
{\ar_{} "2";"12"};
{\ar^{} "4";"14"};
{\ar@{=} "6";"16"};
{\ar_{} "10";"12"};
{\ar^{} "12";"14"};
{\ar^{} "14";"16"};
{\ar_{{}^{\exists}[z\ i]} "12";"22"};
{\ar^{} "14";"24"};
{\ar@{=} "22";"24"};
{\ar@{}|\circlearrowright "-12";"4"};
{\ar@{}|\circlearrowright "0";"12"};
{\ar@{}|\circlearrowright "2";"14"};
{\ar@{}|\circlearrowright "4";"16"};
{\ar@{}|\circlearrowright "12";"24"};
\endxy
\]
which implies $L\oplus S_U\in T_U\ast(Z\oplus I)\subseteq\mathcal{T}$, and thus $S_U\in\mathcal{S}\cap\mathcal{T}=\mathcal{I}$. It follows $L\oplus S_U\in\mathcal{L}$.

Moreover, since $i\in\mathcal{C}(I,T_U[1])=0$, this gives a distinguished triangle in $\mathcal{C}$
\begin{equation}\label{TriaLRtoAB1}
T_U\to L\oplus S_U\to Z\oplus I\overset{[z\ 0]}{\longrightarrow}T_U[1].
\end{equation}
It follows that $I$ is a direct summand of $L\oplus S_U\in\mathcal{L}$, and thus there exists $L^{\prime}\in\mathcal{L}$ satisfying $L^{\prime}\oplus I\cong L\oplus S_U$.
Then from $(\ref{TriaLRtoAB1})$, we obtain a distinguished triangle in $\mathcal{C}$
\[ T_U\to L^{\prime}\to Z\overset{z}{\longrightarrow} T_U[1] \]
satisfying $L^{\prime}\in\mathcal{L}$ and $T_U\in\mathcal{B}$.

\medskip

{\rm (ii)} Let $U\in\mathcal{U}$ be any object. Decompose it into a distinguished triangle in $\mathcal{C}$
\[ S[-1]\to U\to Z\to S\quad(S\in\mathcal{S},Z\in\mathcal{Z}), \]
which gives $\sigma(U)\cong Z$ in $\mathcal{Z}/\mathcal{I}$.

By assumption, there is a distinguished triangle in $\mathcal{C}$
\[ L\to Z\to B[1]\to L[1]\quad(L\in\mathcal{L},B\in\mathcal{B}). \]
By the octahedron axiom, we have
\[
\xy
(-8,23)*+{S[-1]}="-12";
(8,23)*+{S[-1]}="-14";
(-24,8)*+{B}="0";
(-8,8)*+{{}^{\exists}E}="2";
(8,8)*+{U}="4";
(24,8)*+{B[1]}="6";
(-24,-7)*+{B}="10";
(-8,-7)*+{L}="12";
(8,-7)*+{Z}="14";
(24,-7)*+{B[1]}="16";
(-8,-22)*+{S}="22";
(8,-22)*+{S}="24";
{\ar@{=} "-12";"-14"};
{\ar_{} "-12";"2"};
{\ar^{} "-14";"4"};
{\ar^{} "0";"2"};
{\ar^{} "2";"4"};
{\ar^{} "4";"6"};
{\ar@{=} "0";"10"};
{\ar_{} "2";"12"};
{\ar^{} "4";"14"};
{\ar@{=} "6";"16"};
{\ar_{} "10";"12"};
{\ar^{} "12";"14"};
{\ar^{} "14";"16"};
{\ar_{} "12";"22"};
{\ar^{} "14";"24"};
{\ar@{=} "22";"24"};
{\ar@{}|\circlearrowright "-12";"4"};
{\ar@{}|\circlearrowright "0";"12"};
{\ar@{}|\circlearrowright "2";"14"};
{\ar@{}|\circlearrowright "4";"16"};
{\ar@{}|\circlearrowright "12";"24"};
\endxy.
\]
Decompose $E$ into a distinguished triangle in $\mathcal{C}$
\[ U_E\to E\to V_E[1]\to U_E[1]\quad(U_E\in\mathcal{U},V_E\in\mathcal{V}). \]
Applying Condition {\rm (I)} to $E\in B\ast U\subseteq\mathcal{T}\ast\mathcal{U}$, we obtain an isomorphism in $\mathcal{Z}/\mathcal{I}$
\[ \sigma(\omega_{\mathcal{U}}(E))\cong\omega(\sigma_{\mathcal{T}}(E)), \]
namely, $\sigma(U_E)\cong\omega(L)\cong L$ in $\mathcal{Z}/\mathcal{I}$.

This implies $U_E\in\overline{\sigma}^{-1}(\mathscr{L})=\mathcal{A}$. By the octahedron axiom, we obtain the following commutative diagram made of distinguished triangles.
\[
\xy
(-14,18)*+{U_E}="0";
(-14.1,0)*+{E}="2";
(-14,-16)*+{V_E[1]}="4";
(-3,0)*+{U}="6";
(2,-9)*+{{}^{\exists}F[1]}="8";
(20,0)*+{B[1]}="10";
(-8,-8)*+_{_{\circlearrowright}}="12";
(5,-3)*+_{_{\circlearrowright}}="14";
(-10,5)*+_{_{\circlearrowright}}="14";
{\ar_{} "0";"2"};
{\ar_{} "2";"4"};
{\ar^{} "0";"6"};
{\ar_{} "2";"6"};
{\ar_{} "6";"8"};
{\ar^{} "6";"10"};
{\ar_{} "4";"8"};
{\ar_{} "8";"10"};
\endxy
\]

By Lemma \ref{LemSUU} {\rm (2)}, $F\in\mathcal{T}$ satisfies $\omega(F)\cong\omega(B)\in\mathscr{R}$, i.e., $F\in\overline{\omega}^{-1}(\mathscr{R})=\mathcal{B}$.
Thus the distinguished triangle in $\mathcal{C}$
\[ U_E\to U\to F[1]\to U_E[1] \]
satisfies $U_E\in\mathcal{A},F\in\mathcal{B}$.

\medskip

{\rm (iii)} Let $C\in\mathcal{C}$ be any object. Decompose it into a distinguished triangle in $\mathcal{C}$
\[ U_C\to C\to V_C[1]\to U_C[1]\quad(U_C\in\mathcal{U},V_C\in\mathcal{V}). \]
By the assumption, there is a distinguished triangle in $\mathcal{C}$
\[ A\to U_C\to B[1]\to A[1]\quad(A\in\mathcal{A},B\in\mathcal{B}). \]
By the octahedron axiom, we have the following commutative diagram made of distinguished triangles.
\[
\xy
(-20,0)*+{A}="0";
(-1,8)*+{U_C}="2";
(16,15)*+{B[1]}="4";
(3,0)*+{C}="6";
(16.1,0)*+{{}^{\exists}G[1]}="8";
(16,-19)*+{V_C[1]}="10";
(9,6)*+_{_{\circlearrowright}}="12";
(-5,3)*+_{_{\circlearrowright}}="14";
(12,-6)*+_{_{\circlearrowright}}="14";
{\ar^{} "0";"2"};
{\ar^{} "2";"4"};
{\ar_{} "0";"6"};
{\ar^{} "2";"6"};
{\ar^{} "6";"8"};
{\ar_{} "6";"10"};
{\ar^{} "4";"8"};
{\ar^{} "8";"10"};
\endxy
\]
It remains to show $G\in\mathcal{B}$.

Remark that we have $G\in B\ast V_C\subseteq\mathcal{T}$.
Since $V_C\in\mathcal{V}\subseteq\mathcal{I}\ast\mathcal{V} [1]$, there is a distinguished triangle in $\mathcal{C}$
\[ I^{\prime}\to V_C\to V^{\prime}[1]\to I^{\prime}[1]\quad(I^{\prime}\in\mathcal{I},V^{\prime}\in\mathcal{V}). \]
By the octahedron axiom and $\mathcal{C}(I^{\prime},B[1])=0$, we obtain
\[
\xy
(-8,23)*+{V^{\prime}}="-12";
(8,23)*+{V^{\prime}}="-14";
(-24,8)*+{B}="0";
(-8,8)*+{B\oplus I^{\prime}}="2";
(8,8)*+{I^{\prime}}="4";
(24,8)*+{B[1]}="6";
(-24,-7)*+{B}="10";
(-8,-7)*+{G}="12";
(8,-7)*+{V_C}="14";
(24,-7)*+{B[1]}="16";
(-8,-22)*+{V^{\prime}[1]}="22";
(8,-22)*+{V^{\prime}[1]}="24";
{\ar@{=} "-12";"-14"};
{\ar_{} "-12";"2"};
{\ar^{} "-14";"4"};
{\ar^{} "0";"2"};
{\ar^{} "2";"4"};
{\ar^{0} "4";"6"};
{\ar@{=} "0";"10"};
{\ar_{} "2";"12"};
{\ar^{} "4";"14"};
{\ar@{=} "6";"16"};
{\ar_{} "10";"12"};
{\ar^{} "12";"14"};
{\ar^{} "14";"16"};
{\ar_{} "12";"22"};
{\ar^{} "14";"24"};
{\ar@{=} "22";"24"};
{\ar@{}|\circlearrowright "-12";"4"};
{\ar@{}|\circlearrowright "0";"12"};
{\ar@{}|\circlearrowright "2";"14"};
{\ar@{}|\circlearrowright "4";"16"};
{\ar@{}|\circlearrowright "12";"24"};
\endxy.
\]
Applying Lemma \ref{LemSUU} {\rm (2)} to $V^{\prime}\to B\oplus I^{\prime}\to G\to V^{\prime}[1]$, we obtain
\[ \omega(G)\cong\omega(B\oplus I^{\prime})\cong\omega(B)\in\mathscr{R} \]
in $\mathcal{Z}/\mathcal{I}$, which shows $G\in\overline{\omega}^{-1}(\mathscr{R})=\mathcal{B}$.
\end{proof}

\begin{thm}\label{ThmBij}
Let $\mathcal{P}$ be a concentric TCP satisfying {\rm (I)+(II)}.
Then we have the following mutually inverse bijections.
\begin{eqnarray*}
\mathbb{R}\colon\mathfrak{M}_{\mathcal{P}}\to\mathfrak{CP}(\mathcal{Z}/\mathcal{I})&;&(\mathcal{A},\mathcal{B})\mapsto(\overline{\sigma}(\mathcal{A}),\overline{\omega}(\mathcal{B})),\\
\mathbb{I}\colon\mathfrak{CP}(\mathcal{Z}/\mathcal{I})\to\mathfrak{M}_{\mathcal{P}}&;&(\mathscr{L},\mathscr{R})\mapsto(\overline{\sigma}^{-1}(\mathscr{L}),\overline{\omega}^{-1}(\mathscr{R})).
\end{eqnarray*}
\end{thm}
\begin{proof}
First we show the well-definedness of these maps.

Let $(\mathcal{A},\mathcal{B})\in\mathfrak{M}_{\mathcal{P}}$ be any element. Put $\mathscr{L}=\overline{\sigma}(\mathcal{A})$, $\mathscr{R}=\overline{\omega}(\mathcal{B})$. Then,
\begin{itemize}
\item[-] $\mathrm{Ext}^1_{\mathcal{Z}/\mathcal{I}}(\mathscr{L},\mathscr{R})=0$ follows from the definition of $\mathfrak{M}_{\mathcal{P}}$, 
\item[-] $\mathcal{Z}/\mathcal{I}=\mathscr{L}\ast\Sigma\mathscr{R}$ follows from $\mathcal{C}=\mathcal{A}\ast\mathcal{B}[1]$ and Proposition \ref{PropABtoLR}.
\end{itemize}
It follows $(\mathscr{L},\mathscr{R})\in\mathfrak{CP}(\mathcal{Z}/\mathcal{I})$.

\smallskip

Let $(\mathscr{L},\mathscr{R})\in\mathfrak{CP}(\mathcal{Z}/\mathcal{I})$ be any element. Put $\mathcal{A}=\overline{\sigma}^{-1}(\mathscr{L})$, $\mathcal{B}=\overline{\omega}^{-1}(\mathscr{R})$. Then,
\begin{itemize}
\item[-] $\mathcal{S}\subseteq\mathcal{A}\subseteq\mathcal{U}$ and $\mathcal{V}\subseteq\mathcal{B}\subseteq\mathcal{T}$ are satisfied,
\item[-] $\mathrm{Ext}^1_{\mathcal{C}}(\mathcal{A},\mathcal{B})=0$ follows from $\mathrm{Ext}^1_{\mathcal{Z}/\mathcal{I}}(\mathscr{L},\mathscr{R})=0$ and Proposition \ref{PropExtInj}, 
\item[-] $\mathcal{C}=\mathcal{A}\ast\mathcal{B}[1]$ follows from $\mathcal{Z}/\mathcal{I}=\mathscr{L}\ast\Sigma\mathscr{R}$ and Proposition \ref{PropLRtoAB},
\item[-] $\mathrm{Ext}^1_{\mathcal{Z}/\mathcal{I}}(\overline{\sigma}(\overline{\sigma}^{-1}(\mathscr{L}),\overline{\omega}(\overline{\omega}^{-1}(\mathscr{R})))=0$ immediately follows from $\mathrm{Ext}^1_{\mathcal{Z}/\mathcal{I}}(\mathscr{L},\mathscr{R})=0$, since we have $\overline{\sigma}(\overline{\sigma}^{-1}(\mathscr{L}))\subseteq\mathscr{L}$ and $\overline{\omega}(\overline{\omega}^{-1}(\mathscr{R}))\subseteq\mathscr{R}$.
\end{itemize}
It follows $(\mathcal{A},\mathcal{B})\in\mathfrak{M}_{\mathcal{P}}$.

Thus these maps are well-defined.
It remains to show that they are mutually inverses.

Let $(\mathcal{A},\mathcal{B})\in\mathfrak{M}_{\mathcal{P}}$ be any element. Then it obviously satisfies
\begin{equation}\label{InclBij1}
\overline{\sigma}^{-1}(\overline{\sigma}(\mathcal{A}))\supseteq\mathcal{A}\quad\text{and}\quad \overline{\omega}^{-1}(\overline{\omega}(\mathcal{B}))\supseteq\mathcal{B}.
\end{equation}
Since both $(\mathcal{A},\mathcal{B})$ and $\mathbb{I}\circ\mathbb{R}((\mathcal{A},\mathcal{B}))$ are cotorsion pairs in $\mathcal{C}$, $(\ref{InclBij1})$ implies $(\mathcal{A},\mathcal{B})=\mathbb{I}\circ\mathbb{R}((\mathcal{A},\mathcal{B}))$.
This means $\mathbb{I}\circ\mathbb{R}=\mathrm{id}$.

Similarly, for any $(\mathscr{L},\mathscr{R})\in\mathfrak{CP}(\mathcal{Z}/\mathcal{I})$,
\[ \overline{\sigma}(\overline{\sigma}^{-1}(\mathscr{L}))\subseteq\mathscr{L} \quad\text{and}\quad \overline{\omega}(\overline{\omega}^{-1}(\mathscr{R}))\subseteq\mathscr{R} \]
imply $\mathbb{R}\circ\mathbb{I}((\mathscr{L},\mathscr{R}))=(\mathscr{L},\mathscr{R})$. This means $\mathbb{R}\circ\mathbb{I}=\mathrm{id}$.

\end{proof}

\begin{cor}\label{CorABBij1}
Let $\mathcal{P}$ be a concentric TCP satisfying {\rm (I)+(II)}. For any $(\mathcal{A},\mathcal{B})\in\mathfrak{CP}(\mathcal{C})$ satisfying $\mathcal{S}\subseteq\mathcal{A}$ and $\mathcal{V}\subseteq\mathcal{B}$, the following are equivalent.
\begin{enumerate}
\item $(\mathcal{A},\mathcal{B})\in\mathfrak{M}_{\mathcal{P}}$.
\item $\mathcal{U}\cap(\mathcal{S}[-1]\ast\mathcal{A})\subseteq\mathcal{A}$.
\item $\mathcal{T}\cap(\mathcal{B}\ast\mathcal{V}[1])\subseteq\mathcal{B}$.
\end{enumerate}
\end{cor}
\begin{proof}
Since $\mathbb{I}\circ\mathbb{R}=\mathrm{id}$, {\rm (1)} implies {\rm (2)} and {\rm (3)}. In fact, we have $\mathcal{U}\cap(\mathcal{S}[-1]\ast\mathcal{A})=\mathcal{A}$ and $\mathcal{T}\cap(\mathcal{B}\ast\mathcal{V}[1])=\mathcal{B}$.

Conversely, suppose {\rm (3)} holds. To show {\rm (1)}, it suffices to show $\mathrm{Ext}^1_{\mathcal{Z}/\mathcal{I}}(\sigma(\mathcal{A}/\mathcal{I}),\omega(\mathcal{B}/\mathcal{I}))=0$.
Let $A\in\mathcal{A},B\in\mathcal{B}$ be any pair of objects. By Lemma \ref{LemExtInj}, there is a distinguished triangle in $\mathcal{C}$
\[ B_0\to R\to N\to B_0[1] \]
satisfying $R\in\mathcal{T}\cap(\mathcal{B}\ast\mathcal{V}[1])$ and $\mathrm{Ext}^1_{\mathcal{C}}(A,R)\cong\mathrm{Ext}^1_{\mathcal{Z}/\mathcal{I}}(\sigma(A),\omega(B))$. By {\rm (3)}, we have $R\in\mathcal{B}$ and thus $\mathrm{Ext}^1_{\mathcal{C}}(A,R)=0$.

$(2)\Rightarrow(1)$ can be shown dually.
\end{proof}

\begin{cor}\label{CorABBij2}
For any concentric $\mathcal{P}$ satisfying {\rm (I)+(II)}, we have
\[ \mathfrak{M}_{\mathcal{P}}=\{(\mathcal{A},\mathcal{B})\in\mathfrak{CP}(\mathcal{C})\mid\mathcal{A}=\mathcal{U}\cap(\mathcal{S}[-1]\ast\mathcal{A}),\ \mathcal{B}=\mathcal{T}\cap(\mathcal{B}\ast\mathcal{V}[1])\}. \]
\end{cor}
\begin{proof}
If $\mathcal{A}=\mathcal{U}\cap(\mathcal{S}[-1]\ast\mathcal{A})$ and $\mathcal{B}=\mathcal{T}\cap(\mathcal{B}\ast\mathcal{V}[1])$ hold, then $\mathcal{S}\subseteq\mathcal{A}$ and $\mathcal{V}\subseteq\mathcal{B}$ follows automatically. Thus the equation follows from Corollary \ref{CorABBij1}.
\end{proof}

As a consequence of Theorem \ref{ThmBij}, we can define {\it mutation} of cotorsion pairs in $\mathfrak{M}_{\mathcal{P}}$ as a $\mathbb{Z}$-action as below.
\begin{dfn}\label{DefGeneralMut}
Let $\mathcal{P}$ be a concentric TCP satisfying {\rm (I)+(II)}.
Define the {\it mutation} on $\mathfrak{M}_{\mathcal{P}}$ to be a $\mathbb{Z}$-action, obtained by pulling back the shift $\mathbb{Z}$-action on $\mathfrak{CP}(\mathcal{Z}/\mathcal{I})$.

Namely, for any $n\in\mathbb{Z}$, we define as $\mu_n=\mathbb{I}\circ [n]\circ \mathbb{R}$. This gives a $\mathbb{Z}$-action $\mathbb{Z}\times\mathfrak{M}_{\mathcal{P}}\to\mathfrak{M}_{\mathcal{P}}$.
\end{dfn}

\section{Typical cases} \label{section_Typical}

Let us see the meaning of Theorem \ref{ThmBij} in the following particular cases (\S \ref{section_HE}, \S \ref{section_HTCP}).
\[
\xy
(0,24)*+{\left[\begin{array}{c}\text{Concentric TCP}\end{array}\right]}="-2";
(0,8)*+{\left[\begin{array}{c}\text{Concentric TCP}\\ \text{satisfying {\rm (I)+(II)}}\end{array}\right]}="0";
(-24,-8)*+{\left[\begin{array}{c}\text{Heart-equivalent}\\ \text{case}\ (\text{\S \ref{section_HE}})\end{array}\right]}="2";
(-24,-30)*+{\left[\begin{array}{c}\text{TCP appearing}\\ \text{in mutation}\\ (\text{Example \ref{ExZZ}}) \end{array}\right]}="4";
(24,-8)*+{\left[\begin{array}{c}\text{Hovey}\\ \text{TCP}\ (\text{\S \ref{section_HTCP}})\end{array}\right]}="6";
(24,-30)*+{\left[\begin{array}{c}\text{Recollement of}\\ \text{triangulated categories}\\ (\text{Corollary \ref{CorRecolltoHTCP}})\end{array}\right]}="8";
(-48,28)*+{}="10";
(-48,-36)*+{}="11";
(-56,24)*+{\text{general}}="20";
(-56,-32)*+{\text{special}}="20";
{\ar@{<=} "-2";"0"};
{\ar@{<=} "0";"2"};
{\ar@{<=} "2";"4"};
{\ar@{<=} "0";"6"};
{\ar@{<=} "6";"8"};
{\ar "11";"10"};
\endxy
\]

\subsection{Heart-equivalent case}\label{section_HE}

Instead of the condition {\rm (I)+(II)}, let us consider the following, a bit stronger condition.
\begin{cond}\label{CondHeartEq}
Let $\mathcal{P}=((\mathcal{S},\mathcal{T}),(\mathcal{U},\mathcal{V}))$ be a concentric TCP.
We consider the following condition.
\begin{itemize}
\item[{\rm (III)}] $(\mathcal{S},\mathcal{T})$ and $(\mathcal{U},\mathcal{V})$ are heart-equivalent. Namely, $\mathcal{P}$ satisfies $H_{(\mathcal{S},\mathcal{T})}(\mathcal{U})=0$ and $H_{(\mathcal{U},\mathcal{V})}(\mathcal{T})=0$ (Remark \ref{RemTCPAdj}).
\end{itemize}
\end{cond}

\begin{prop}\label{PropIIItoIandII}
Let $\mathcal{P}$ be a concentric TCP satisfying {\rm (III)}. Then it satisfies {\rm (I)} and {\rm (II)}.
\end{prop}
\begin{proof}
{\rm (I)} Let $X\in\mathcal{T}\ast\mathcal{U}$ be any object. Since $H_{(\mathcal{U},\mathcal{V})}(\mathcal{T})=H_{(\mathcal{U},\mathcal{V})}(\mathcal{U})=0$, we have $H_{(\mathcal{U},\mathcal{V})}(X)=0$. This means that there exists a diagram
\[
\xy
(-24,0)*+{V_X}="0";
(-8,0)*+{U_X}="2";
(8,0)*+{X}="4";
(28,0)*+{V_X[1]}="6";
(18,-12)*+{V_0}="8";
(18,-5)*+{_{\circlearrowright}}="10";
{\ar^{} "0";"2"};
{\ar^{u} "2";"4"};
{\ar^{v} "4";"6"};
{\ar_{} "4";"8"};
{\ar_{} "8";"6"};
\endxy
\quad(U_X\in\mathcal{U},\, V_X,V_0\in\mathcal{V})
\]
where $V_X\to U_X\overset{u}{\longrightarrow}X\overset{v}{\longrightarrow}V_X[1]$ is a distinguished triangle (Fact \ref{FactCPAdj} {\rm (4)}). Decompose $X$ into a distinguished triangle
\[ S_X[-1]\overset{s}{\longrightarrow}X\overset{t}{\longrightarrow}T_X\to S_X\quad(S_X\in\mathcal{S},T_X\in\mathcal{T}). \]
Since $v\circ s=0$ by $\mathcal{C}(S_X[-1],V_0)=0$, this $s$ factors through $u$. By the octahedron axiom, we obtain
\[
\xy
(-8,23)*+{V_X}="-12";
(8,23)*+{V_X}="-14";
(-25,8)*+{S_X[-1]}="0";
(-8,8)*+{U_X}="2";
(8,8)*+{{}^{\exists}Z}="4";
(24,8)*+{S_X}="6";
(-25,-7)*+{S_X[-1]}="10";
(-8,-7)*+{X}="12";
(8,-7)*+{T_X}="14";
(24,-7)*+{S_X}="16";
(-8,-22)*+{V_X[1]}="22";
(8,-22)*+{V_X[1]}="24";
{\ar@{=} "-12";"-14"};
{\ar_{} "-12";"2"};
{\ar^{} "-14";"4"};
{\ar^{} "0";"2"};
{\ar^{{}^{\exists}x} "2";"4"};
{\ar^{} "4";"6"};
{\ar@{=} "0";"10"};
{\ar_{u} "2";"12"};
{\ar^{{}^{\exists}y} "4";"14"};
{\ar@{=} "6";"16"};
{\ar_(0.56){s} "10";"12"};
{\ar_{t} "12";"14"};
{\ar^{} "14";"16"};
{\ar_{v} "12";"22"};
{\ar^{} "14";"24"};
{\ar@{=} "22";"24"};
{\ar@{}|\circlearrowright "-12";"4"};
{\ar@{}|\circlearrowright "0";"12"};
{\ar@{}|\circlearrowright "2";"14"};
{\ar@{}|\circlearrowright "4";"16"};
{\ar@{}|\circlearrowright "12";"24"};
\endxy
\]
which shows $Z\in\mathcal{U}\cap\mathcal{T}=\mathcal{Z}$. The commutative square in the middle shows that $\mu_X\colon\sigma(\omega_{\mathcal{U}}(X))\to\omega(\sigma_{\mathcal{T}}(X))$ is isomorphism in $\mathcal{Z}/\mathcal{I}$.

\medskip

{\rm (II)} Obviously, we always have $\mathcal{T}\cap(\mathcal{S}\ast\mathcal{V}[1])=\mathcal{T}\cap(\mathcal{I}\ast\mathcal{V}[1])$. Suppose $H_{(\mathcal{U},\mathcal{V})}(\mathcal{T})=0$ holds. Then for any $T\in\mathcal{T}$ and any distinguished triangle
\[ V\to U\to T\overset{v}{\longrightarrow}V[1]\quad(U\in\mathcal{U},\, V\in\mathcal{V}), \]
$v$ factors through some $V_0\in\mathcal{V}$ (Fact \ref{FactCPAdj} {\rm (4)}). In particular if $T\in\mathcal{T}\cap(\mathcal{I}\ast\mathcal{V}[1])$, we obtain a diagram
\[
\xy
(-24,0)*+{V}="0";
(-8,0)*+{I}="2";
(8,0)*+{T}="4";
(28,0)*+{V[1]}="6";
(18,-12)*+{V_0}="8";
(18,-5)*+{_{\circlearrowright}}="10";
{\ar^{} "0";"2"};
{\ar^{} "2";"4"};
{\ar^{v} "4";"6"};
{\ar_{} "4";"8"};
{\ar_{} "8";"6"};
\endxy
\quad(I\in\mathcal{I},\, V,V_0\in\mathcal{V}).
\]
This shows $T\in\mathcal{U}[-1]^{\perp}=\mathcal{V}$. Thus it follows
\[ \mathcal{T}\cap\mathcal{N}^i=\mathcal{T}\cap(\mathcal{I}\ast\mathcal{V}[1])\subseteq\mathcal{V}. \]
Since $\mathcal{V}\subseteq\mathcal{T}\cap\mathcal{N}^i$ follows from Remark \ref{RemConcentric} {\rm (ii)}, we obtain $\mathcal{T}\cap\mathcal{N}^i=\mathcal{V}$.
Dually, $H_{(\mathcal{S},\mathcal{T})}(\mathcal{U})=0$ implies $\mathcal{U}\cap\mathcal{N}^f=\mathcal{S}$.
\end{proof}

\begin{prop}\label{PropIIItoMut}
If $\mathcal{P}$ is concentric with {\rm (III)}, then $(\mathcal{Z},\mathcal{Z})$ becomes an $\mathcal{I}$-mutation pair, in the sense of Fact \ref{FactIYo}.
\end{prop}
\begin{proof}
First, let us show that $H_{(\mathcal{S},\mathcal{T})}(\mathcal{U})=0$ implies $\mathcal{U}\cap(\mathcal{I}\ast\mathcal{Z}[1])\subseteq\mathcal{Z}$. Remark that $H_{(\mathcal{S},\mathcal{T})}(\mathcal{U})=0$ means $\tau_{(\mathcal{S},\mathcal{T})}^+(\overline{\mathcal{U}})\subseteq\overline{\mathcal{T}}$.
In particular, for any $U\in\mathcal{U}\cap(\mathcal{I}\ast\mathcal{Z}[1])$, we have the following commutative diagram made of distinguished triangles in $\mathcal{C}$ (Fact \ref{FactCPAdj} {\rm (2)}).
\begin{eqnarray*}
&\xy
(-14,18)*+{S[-1]}="0";
(-14.1,0)*+{I}="2";
(-14,-16)*+{T}="4";
(-3,0)*+{U}="6";
(2,-9)*+{T^{\prime}}="8";
(20,0)*+{Z[1]}="10";
(-8,-8)*+_{_{\circlearrowright}}="12";
(5,-3)*+_{_{\circlearrowright}}="14";
(-10,5)*+_{_{\circlearrowright}}="14";
{\ar_{} "0";"2"};
{\ar_{} "2";"4"};
{\ar^{} "0";"6"};
{\ar_{} "2";"6"};
{\ar_{} "6";"8"};
{\ar^{} "6";"10"};
{\ar_{} "4";"8"};
{\ar_{} "8";"10"};
\endxy&\\
&(S\in\mathcal{S},\, T,T^{\prime}\in\mathcal{T},\, I\in\mathcal{I},\, Z\in\mathcal{Z})&
\end{eqnarray*}
Since $\mathcal{C}(S[-1],I)=0$, it follows that $U$ is a direct summand of $T^{\prime}$, and thus satisfies $U\in\mathcal{U}\cap\mathcal{T}=\mathcal{Z}$. Thus $\mathcal{U}\cap(\mathcal{I}\ast\mathcal{Z}[1])\subseteq\mathcal{Z}$ is shown.

Then, for any $Z\in\mathcal{Z}$, the distinguished triangle
\[ Z\to I_Z\to U_Z\to Z[1]\quad(I_Z\in\mathcal{I},U_Z\in\mathcal{U}) \]
should satisfy $U_Z\in\mathcal{U}\cap(\mathcal{I}\ast\mathcal{Z}[1])\subseteq\mathcal{Z}$.
This means
\[ \mathcal{Z}\subseteq(\mathcal{Z}[-1]\ast\mathcal{I})\cap\mathcal{I}[-1]^{\perp}. \]
Dually, $H_{(\mathcal{U},\mathcal{V})}(\mathcal{T})=0$ implies $\mathcal{Z}\subseteq(\mathcal{I}\ast\mathcal{Z}[1])\cap {}^{\perp}\mathcal{I}[1]$.
\end{proof}

\begin{cor}\label{CorIIItoMut}
If $\mathcal{P}$ is concentric with {\rm (III)}, then we have isomorphisms of functors 
\[ \Sigma\cong\langle1\rangle,\quad \Omega\cong\langle-1\rangle. \]
Thus the triangulation of $\mathcal{Z}/\mathcal{I}$ agrees with the one given in \cite[Theorem 4.2]{IYo}.
\end{cor}
\begin{proof}
This immediately follows from Proposition \ref{PropIIItoMut}.
\end{proof}

\begin{ex}\label{ExZZ}
Let $\mathcal{P}=((\mathcal{S},\mathcal{T}),(\mathcal{U},\mathcal{V}))$ be a TCP. Suppose $\mathcal{S}=\mathcal{V}(=\mathcal{I})$ and $\mathcal{U}=\mathcal{T}(=\mathcal{Z})$, as in Example \ref{Thefollowingareexamples} {\rm (5)}. Then $\mathcal{P}$ is a concentric TCP satisfying {\rm (III)}.
In this case, the map $\mu_1$ defined in Definition \ref{DefGeneralMut} agrees with the mutation $\mu^{-1}$ given in \cite[Theorem 3.7]{ZZ1}.
\end{ex}
\begin{proof}
Concentricity is trivial. Moreover, we have
\[ H_{(\mathcal{S},\mathcal{T})}(\mathcal{U})=H_{(\mathcal{S},\mathcal{T})}(\mathcal{T})=0,\ \ H_{(\mathcal{U},\mathcal{V})}(\mathcal{T})=H_{(\mathcal{U},\mathcal{V})}(\mathcal{U})=0. \]
The latter part follows from Corollary \ref{CorIIItoMut}.
\end{proof}

\subsection{Hovey TCP, as a generalization of recollement}\label{section_HTCP}

In view of \cite{Ho1}, \cite{Ho2}, the following has been defined in \cite[Definition 5.1]{NP}.
\begin{dfn}\label{DefHTCP}
We say that TCP $\mathcal{P}=((\mathcal{S},\mathcal{T}),(\mathcal{U},\mathcal{V}))$ is {\it Hovey TCP} %\footnote{The reason for this terminology will be given in \cite{NP}.}
if it satisfies $\mathcal{N}^i=\mathcal{N}^f\ (=:\mathcal{N})$.
\end{dfn}

\begin{ex}
If $\mathcal{P}$ is degenerated to a single cotorsion pair, then it is a Hovey TCP, which satisfies $\mathcal{N}=\mathcal{C}$.

A more interesting example is the one arising from a recollement. See Corollary \ref{CorRecolltoHTCP}.
\end{ex}

\begin{prop}\label{PropHTCP}
If $\mathcal{P}$ is Hovey TCP, then $\mathcal{N}\subseteq\mathcal{C}$ is a thick triangulated subcategory.
\end{prop}
\begin{proof}
$\mathcal{N}\subseteq\mathcal{C}$ is closed under direct summands by Corollary \ref{CorInverseImage}. By the definition of a Hovey TCP, it is also closed under shifts.

Let us show $\mathcal{N}\subseteq\mathcal{C}$ is closed under extensions. To show $\mathcal{N}\ast\mathcal{N}\subseteq\mathcal{N}$, obviously it suffices to show $\mathcal{V}[1]\ast\mathcal{S}\subseteq\mathcal{N}$.

For any $X\in\mathcal{V}[1]\ast\mathcal{S}$, take a distinguished triangle
\[ S[-1]\overset{s}{\longrightarrow}V[1]\to X\to S\quad(S\in\mathcal{S},V\in\mathcal{V}) \]
and then decompose $V\in\mathcal{V}\subseteq\mathcal{N}$ into a distinguished triangle
\[ S^{\prime}[-1]\to V[1]\overset{v}{\longrightarrow}V^{\prime}\to S^{\prime}\quad(S^{\prime}\in\mathcal{S},V^{\prime}\in\mathcal{V}). \]
Since $v\circ s=0$, the octahedron axiom gives
\[
\xy
(-20,16)*+{S[-1]}="0";
(0.5,15)*+{S^{\prime}[-1]}="2";
(17,14)*+{{}^{\exists}M}="4";
(-2,4)*+{V[1]}="6";
(6.1,-0.6)*+{X}="8";
(-5.8,-14.2)*+{V^{\prime}}="10";
(7.5,8.5)*+_{_{\circlearrowright}}="12";
(-5.5,11.5)*+_{_{\circlearrowright}}="14";
(-0.5,-4.3)*+_{_{\circlearrowright}}="14";
{\ar^{} "0";"2"};
{\ar^{} "2";"4"};
{\ar_{s} "0";"6"};
{\ar^{} "2";"6"};
{\ar^{} "6";"8"};
{\ar_{v} "6";"10"};
{\ar^{} "4";"8"};
{\ar^{} "8";"10"};
\endxy,
\]
which shows $X\in\mathcal{S}[-1]\ast\mathcal{S}\ast\mathcal{V}$. By Remark \ref{RemConcentric} {\rm (ii)}, it follows $X\in\mathcal{S}[-1]\ast(\mathcal{S}[-1]\ast\mathcal{V})\ast\mathcal{V}=\mathcal{N}$.
\end{proof}

In particular, we obtain the Verdier quotient $\ell\colon\mathcal{C}\to\mathcal{C}_{\mathcal{N}}$. We call morphism $f\in\mathcal{C}(X,Y)$ a {\it quasi-isomorphism} if its cone belongs to $\mathcal{N}$.
\begin{dfn}\label{DefEquivZIandCN}
Let $\mathcal{P}$ be a Hovey TCP. Since $\mathcal{I}\subseteq\mathcal{N}$, we have a functor $\Phi_{\mathcal{C}}\colon\mathcal{C}/\mathcal{I}\to\mathcal{C}_{\mathcal{N}}$ which makes the following diagram commutative (on the nose), by the universal property of the ideal quotient.
\[
\xy
(-10,8)*+{\mathcal{C}}="0";
(0,-4)*+{}="1";
(-10,-8)*+{\mathcal{C}/\mathcal{I}}="2";
(8,8)*+{\mathcal{C}_{\mathcal{N}}}="4";
{\ar_{p} "0";"2"};
{\ar^{\ell} "0";"4"};
{\ar_{\Phi_{\mathcal{C}}} "2";"4"};
{\ar@{}|\circlearrowright "0";"1"};
\endxy
\]
We denote the restrictions of $\Phi_{\mathcal{C}}$ as
\[ \Phi_{\mathcal{U}}=(\Phi_{\mathcal{C}})|_{\mathcal{U}/\mathcal{I}},\ \ \Phi_{\mathcal{T}}=(\Phi_{\mathcal{C}})|_{\mathcal{T}/\mathcal{I}},\ \ \Phi=(\Phi_{\mathcal{C}})|_{\mathcal{Z}/\mathcal{I}} \]
in the rest.
\end{dfn}

\begin{prop}\label{PropEquivZIandCN}
Let $\mathcal{P}$ be a Hovey TCP. Then $\Phi\colon\mathcal{Z}/\mathcal{I}\to\mathcal{C}_{\mathcal{N}}$ is an equivalence
\begin{equation}\label{DiagEquivZIandCN}
\xy
(-8,6)*+{\mathcal{Z}}="0";
(8,6)*+{\mathcal{C}}="2";
(-8,-6)*+{\mathcal{Z}/\mathcal{I}}="4";
(8,-6)*+{\mathcal{C}_{\mathcal{N}}}="6";
{\ar@{^(->} "0";"2"};
{\ar_{p} "0";"4"};
{\ar^{\ell} "2";"6"};
{\ar_{\Phi}^{\simeq} "4";"6"};
{\ar@{}|\circlearrowright "0";"6"};
\endxy,
\end{equation}
which makes the following diagrams commutative up to natural isomorphisms.
\begin{equation}\label{Diag_ZIandCN}
\xy
(0,14)*+{\mathcal{U}}="2";
(18,14)*+{\mathcal{C}}="4";
(0,0)*+{\mathcal{U}/\mathcal{I}}="12";
(18,0)*+{\mathcal{C}/\mathcal{I}}="14";
(-10,0)*+{}="13";
(10,-8)*+{}="15";
(28,0)*+{}="17";
(0,-14)*+{\mathcal{Z}/\mathcal{I}}="22";
(10,-6)*+{}="23";
(18,-14)*+{\mathcal{C}_{\mathcal{N}}}="24";
{\ar@{^(->} "2";"4"};
{\ar^{p_{\mathcal{U}}} "2";"12"};
{\ar_{p} "4";"14"};
{\ar@{^(->} "12";"14"};
{\ar^{\sigma} "12";"22"};
{\ar^{\Phi_{\mathcal{C}}} "14";"24"};
{\ar^(0.56){\Phi_{\mathcal{U}}} "12";"24"};
{\ar^{\simeq}_{\Phi} "22";"24"};
{\ar@/_1.80pc/_{\overline{\sigma}} "2";"22"};
{\ar@/^1.80pc/^{\ell} "4";"24"};
{\ar@{}|\circlearrowright "2";"14"};
{\ar@{}|\circlearrowright "12";"13"};
{\ar@{}|\circlearrowright "14";"17"};
{\ar@{}|\circlearrowright "14";"15"};
{\ar@{}|\circlearrowright "22";"23"};
\endxy
\quad,\quad
\xy
(0,14)*+{\mathcal{T}}="2";
(18,14)*+{\mathcal{C}}="4";
(0,0)*+{\mathcal{T}/\mathcal{I}}="12";
(18,0)*+{\mathcal{C}/\mathcal{I}}="14";
(-10,0)*+{}="13";
(10,-8)*+{}="15";
(28,0)*+{}="17";
(0,-14)*+{\mathcal{Z}/\mathcal{I}}="22";
(10,-6)*+{}="23";
(18,-14)*+{\mathcal{C}_{\mathcal{N}}}="24";
{\ar@{^(->} "2";"4"};
{\ar^{p_{\mathcal{T}}} "2";"12"};
{\ar_{p} "4";"14"};
{\ar@{^(->} "12";"14"};
{\ar^{\omega} "12";"22"};
{\ar^{\Phi_{\mathcal{C}}} "14";"24"};
{\ar^(0.56){\Phi_{\mathcal{T}}} "12";"24"};
{\ar^{\simeq}_{\Phi} "22";"24"};
{\ar@/_1.80pc/_{\overline{\omega}} "2";"22"};
{\ar@/^1.80pc/^{\ell} "4";"24"};
{\ar@{}|\circlearrowright "2";"14"};
{\ar@{}|\circlearrowright "12";"13"};
{\ar@{}|\circlearrowright "14";"17"};
{\ar@{}|\circlearrowright "14";"15"};
{\ar@{}|\circlearrowright "22";"23"};
\endxy
\end{equation}
Here, $p_{\mathcal{U}}$ and $p_{\mathcal{T}}$ denote the residue functors.
\end{prop}
\begin{proof}
For any $X\in\mathcal{C}$, take diagram $(\ref{PentaNat_inC})$
\begin{eqnarray*}
&\xy
(-16,-1)*+{U_X}="0";
(-8,8)*+{Z_U}="2";
(8,8)*+{Z_T}="4";
(16,-1)*+{T_X}="6";
(0,-9)*+{X}="8";
(0,10)*+{}="9";
{\ar^{z_U} "0";"2"};
{\ar^{z} "2";"4"};
{\ar^{z_T} "4";"6"};
{\ar_{u_X} "0";"8"};
{\ar_{t_X} "8";"6"};
{\ar@{}|\circlearrowright "8";"9"};
\endxy&\\
&(U_X\in\mathcal{U},\, T_X\in\mathcal{T},\, Z_U,Z_T\in\mathcal{Z})&\\
&(\underline{t}_X=(\eta_{\mathcal{T}})_X,\ %
\underline{z}_U=\eta_{(U_X)},\ %
\underline{u}_X=(\varepsilon_{\mathcal{U}})_X,\ %
\underline{z}_T=\varepsilon_{(T_X)}
).&
\end{eqnarray*}
Then $t_X,z_U,u_X,z_T$ are quasi-isomorphisms. (As a consequence, so is $z$.) In particular, we have an isomorphism
\[ \ell(z_U)\circ\ell(u_X)^{-1}\colon X\overset{\cong}{\longrightarrow} Z_U \]
in $\mathcal{C}_{\mathcal{N}}$. Thus $\Phi$ is essentially surjective.

Moreover, for any $U\in\mathcal{U}$, the unit morphism $\eta_U\in(\mathcal{U}/\mathcal{I})(U,\sigma(U))$ is sent to an isomorphism
\[ \Phi_{\mathcal{U}}(\eta_U)\colon\Phi_{\mathcal{U}}(U)\overset{\cong}{\longrightarrow}\Phi(\sigma(U)) \]
in $\mathcal{C}_{\mathcal{N}}$. This gives a natural isomorphism $\Phi_{\mathcal{U}}\cong\Phi\circ\sigma$, and thus the left diagram in $(\ref{Diag_ZIandCN})$ is commutative up to natural isomorphism. Commutativity of the right diagram in $(\ref{Diag_ZIandCN})$ can be shown dually.

It remains to show that $\Phi$ is fully faithful. Let $X,Y\in\mathcal{Z}$ be any pair of objects.

\medskip

Let $f\in\mathcal{Z}(X,Y)$ be any morphism. Remark that $\ell(f)=0$ holds if and only if there exists $C\in\mathcal{C}$ and a quasi-isomorphism $s\in\mathcal{C}(Y,C)$ satisfying $s\circ f=0$. This means that $f$ factors through some $N\in\mathcal{N}=\mathcal{S}[-1]\ast\mathcal{V}$. Since $\mathcal{C}(\mathcal{S}[-1],Y)=0$, it follows that $f$ factors through $V\in\mathcal{V}$.
Since $X\in\mathcal{Z}$, there is a distinguished triangle
\[ U_X[-1]\to X\to I_X\to U_X \]
satisfying $U_X\in\mathcal{U}$ and $I_X\in\mathcal{I}$. By $\mathcal{C}(U_X[-1],V)=0$, we can show that $f$ factors through $I_X$, as follows.
\[
\xy
(-18,8)*+{U_X[-1]}="0";
(0,8)*+{X}="2";
(17,8)*+{I_X}="4";
(-2,-1)*+{}="5";
(10,-2)*+{V}="6";
(9,10)*+{}="7";
(0,-10)*+{Y}="8";
{\ar^{} "0";"2"};
{\ar_{} "2";"4"};
{\ar^{} "2";"6"};
{\ar^{} "4";"6"};
{\ar_{} "6";"8"};
{\ar_{f} "2";"8"};
{\ar@{}|\circlearrowright "5";"6"};
{\ar@{}|\circlearrowright "6";"7"};
\endxy
\]
Thus if $\Phi(\underline{f})=\ell(f)=0$, then $\underline{f}=0$ follows. This means $\Phi$ is faithful.

Let us show that $\Phi$ is full. Remark that any morphism $\phi\in\mathcal{C}_{\mathcal{N}}(X,Y)$ can be expressed as a roof
\begin{equation}\label{Roof1}
\xy
(-11,-3)*+{X}="0";
(0,3)*+{C}="2";
(11,-3)*+{Y}="4";
{\ar^{x} "0";"2"};
{\ar_{s} "4";"2"};
\endxy
\end{equation}
with $C\in\mathcal{C}$ and a quasi-isomorphism $s$. Decompose $C$ into a distinguished triangle
\[ S[-1]\to C\overset{t}{\longrightarrow}T\to S\quad(S\in\mathcal{S},T\in\mathcal{T}) \]
in $\mathcal{C}$. Then, since $t$ is a quasi-isomorphism, the roof $(\ref{Roof1})$ gives the same morphism as
\begin{equation}\label{Roof2}
\xy
(-11,-3)*+{X}="0";
(0,3)*+{T}="2";
(11,-3)*+{Y}="4";
{\ar^{t\circ x} "0";"2"};
{\ar_{t\circ s} "4";"2"};
\endxy
\end{equation}
in $\mathcal{C}_{\mathcal{N}}$. Decompose $T$ into a distinguished triangle
\[ V\to Z\overset{z}{\longrightarrow}T\to V[1]\quad(V\in\mathcal{V}, Z\in\mathcal{Z}). \]
By $\mathcal{C}(X,V[1])=0$ and $\mathcal{C}(Y,V[1])=0$, there exists $x^{\prime}\in\mathcal{C}(X,Z)$ and $s^{\prime}\in\mathcal{C}(Y,Z)$ which makes the following diagram commutative in $\mathcal{C}$.
\[
\xy
(-14,0)*+{X}="0";
(4,0)*+{}="1";
(0,8)*+{Z}="2";
(14,0)*+{Y}="4";
(-4,0)*+{}="5";
(0,-8)*+{T}="6";
{\ar^{x^{\prime}} "0";"2"};
{\ar_{t\circ x} "0";"6"};
{\ar_{s^{\prime}} "4";"2"};
{\ar^{t\circ s} "4";"6"};
{\ar^{z} "2";"6"};
{\ar@{}|\circlearrowright "0";"1"};
{\ar@{}|\circlearrowright "4";"5"};
\endxy
\]
Since $\mathcal{N}\subseteq\mathcal{C}$ is thick and since the morphisms $z,\, t\circ s$ are quasi-isomorphisms, it follows that $s^{\prime}$ is also quasi-isomorphism. Thus $(\ref{Roof2})$ gives the same morphism as
\[
\xy
(-11,-3)*+{X}="0";
(0,3)*+{Z}="2";
(11,-3)*+{Y}="4";
{\ar^{x^{\prime}} "0";"2"};
{\ar_{s^{\prime}} "4";"2"};
\endxy
\]
in $\mathcal{C}_{\mathcal{N}}$. By Corollary \ref{Cor3-2}, it follows that $\underline{s}^{\prime}\in(\mathcal{Z}/\mathcal{I})(Y,Z)$ is isomorphism.
Thus $(\underline{s}^{\prime})^{-1}\circ\underline{x}^{\prime}\in(\mathcal{Z}/\mathcal{I})(X,Y)$ satisfies $\Phi((\underline{s}^{\prime})^{-1}\circ\underline{x}^{\prime})=\ell(s^{\prime})^{-1}\circ\ell(x^{\prime})=\phi$.

\end{proof}

\begin{rem}\label{RemBijHTCP}
Let $\mathcal{P}$ be a Hovey TCP. Let $\mathcal{D}\subseteq\mathcal{C}_{\mathcal{N}}$ be a full additive subcategory, closed under isomorphisms and direct summands. Then the commutativity of $(\ref{Diag_ZIandCN})$ shows the following.
\begin{enumerate}
\item $\mathcal{U}\cap\ell^{-1}(\mathcal{D})=\overline{\sigma}^{-1}\Phi^{-1}(\mathcal{D})$.
\item $\mathcal{T}\cap\ell^{-1}(\mathcal{D})=\overline{\omega}^{-1}\Phi^{-1}(\mathcal{D})$.
\end{enumerate}
\end{rem}

\begin{cor}\label{CorHTCP}
Any Hovey TCP is a concentric TCP satisfying {\rm (I)+(II)}.
\end{cor}
\begin{proof}
By Remark \ref{RemTCP_UNTN}, we have
\[ \mathcal{U}\cap\mathcal{N}=\mathcal{S},\quad \mathcal{N}\cap\mathcal{T}=\mathcal{V}, \]
which gives $\mathcal{S}\cap\mathcal{T}=\mathcal{U}\cap\mathcal{N}\cap\mathcal{T}=\mathcal{U}\cap\mathcal{V}$. Thus $\mathcal{P}$ is concentric TCP satisfying {\rm (II)}.

Let us show {\rm (I)}. In fact, we will show that if $\mathcal{P}$ is Hovey TCP, then $\mu_X$ in $(\ref{PentaNat})$ becomes isomorphism for any $X\in\mathcal{C}$.

Let $X\in\mathcal{C}$ be any object. In $(\ref{PentaNat_inC})$, those $u_X,t_X,z_U,z_T$ are quasi-isomorphisms. Thus $\ell(z)$ also becomes isomorphism in $\mathcal{C}_{\mathcal{N}}$. Corollary \ref{Cor3-2} shows that $\mu_X=\underline{z}$ is isomorphism in $\mathcal{Z}/\mathcal{I}$.
\end{proof}

\begin{cor}\label{CorTriaEquivZIandCN}
Let $\mathcal{P}$ be a Hovey TCP. Then the equivalence $\Phi\colon\mathcal{Z}/\mathcal{I}\overset{\simeq}{\longrightarrow}\mathcal{C}_{\mathcal{N}}$ obtained in  Proposition \ref{PropEquivZIandCN} is triangle equivalence.
\end{cor}
\begin{proof}
Denote the shift functor on $\mathcal{C}_{\mathcal{N}}$ by $[1]$, the same symbol as that of $\mathcal{C}$. For any $X\in\mathcal{Z}$, its shift $\Sigma X=Z_X$ in $\mathcal{Z}/\mathcal{I}$ is given by the distinguished triangles
\[
\xy
(-24,8)*+{X}="0";
(-8,8)*+{I_X}="2";
(8,8)*+{U_X}="4";
(26,8)*+{X[1]}="6";
(8,24)*+{S_{U_X}[-1]}="12";
(8,-8)*+{Z_X}="16";
(8,-24)*+{S_{U_X}}="18";
{\ar^{\iota_X} "0";"2"};
{\ar^{\wp_X} "2";"4"};
{\ar^{\gamma_X} "4";"6"};
{\ar "12";"4"};
{\ar^{z_{U_X}} "4";"16"};
{\ar "16";"18"};
\endxy\quad (I_X\in\mathcal{I},\, U_X\in\mathcal{U},\, Z_X\in\mathcal{Z},\, S_{U_X}\in\mathcal{S})
\]
as in Definition \ref{DefShifts}.
Since $\gamma_X$ and $z_{U_X}$ are quasi-isomorphisms, we obtain an isomorphism in $\mathcal{C}_{\mathcal{N}}$
\[ \lambda_X=\ell(z_{U_X})\circ\ell(\gamma_X)^{-1}=\Phi_{\mathcal{U}}(\eta_{U_X})\circ\Phi(\underline{\gamma}_X)^{-1}\colon X[1]\overset{\cong}{\longrightarrow}Z_X=\Phi(\Sigma X). \]
This gives a natural isomorphism $\lambda\colon [1]\circ\Phi\overset{\cong}{\Longrightarrow}\Phi\circ\Sigma$.

Let us show that $\Phi$ sends any standard right triangle to a distinguished triangle in $\mathcal{C}_{\mathcal{N}}$.
Let $X\overset{\underline{f}}{\longrightarrow}Y\overset{\eta_U\circ\underline{a}}{\longrightarrow}\sigma(U)\overset{\sigma(\underline{c})}{\longrightarrow}\Sigma X$ be the standard right triangle in $\mathcal{Z}/\mathcal{I}$, associated to a $\mathcal{U}$-conic triangle
\[ X\overset{f}{\longrightarrow}Y\overset{a}{\longrightarrow}U\overset{b}{\longrightarrow}X[1]. \]
By definition, this is given by a morphism of triangles in $\mathcal{C}$
\[
\xy
(-23,7)*+{X}="0";
(-7,7)*+{Y}="2";
(8,7)*+{U}="4";
(23,7)*+{X[1]}="6";
(-23,-7)*+{X}="10";
(-7,-7)*+{I_X}="12";
(8,-7)*+{U_X}="14";
(23,-7)*+{X[1]}="16";
{\ar^{f} "0";"2"};
{\ar^{a} "2";"4"};
{\ar^{b} "4";"6"};
{\ar@{=} "0";"10"};
{\ar "2";"12"};
{\ar^{c} "4";"14"};
{\ar@{=} "6";"16"};
{\ar_{\iota_X} "10";"12"};
{\ar_{\wp_X} "12";"14"};
{\ar_{\gamma_X} "14";"16"};
{\ar@{}|\circlearrowright "0";"12"};
{\ar@{}|\circlearrowright "2";"14"};
{\ar@{}|\circlearrowright "4";"16"};
\endxy
\]
and the following commutative diagram in $\mathcal{U}/\mathcal{I}$.
\[
\xy
(-9,6)*+{U}="0";
(8,6)*+{\sigma(U)}="2";
(-9,-6)*+{U_X}="4";
(8,-6)*+{\sigma(U_X)}="6";
%(19.5,-6)*+{Z_X}="8";
%(29,-6)*+{\Sigma X}="10";
%
{\ar^(0.46){\eta_U} "0";"2"};
{\ar_{\underline{c}} "0";"4"};
{\ar^(0.46){\sigma(\underline{c})} "2";"6"};
{\ar_{\eta_{U_X}} "4";"6"};
%{\ar@{=} "6";"8"};
%{\ar@{=} "8";"10"};
{\ar@{}|\circlearrowright "0";"6"};
\endxy
\]
This gives the following commutative diagram in $\mathcal{C}_{\mathcal{N}}$.
\[
\xy
(-31,7)*+{X}="0";
(-14,7)*+{Y}="2";
(-18,7)*+{}="3";
(8,7)*+{U}="4";
(27,7)*+{X[1]}="6";
(-31,-7)*+{X}="10";
(-14,-7)*+{Y}="12";
(8,-7)*+{\sigma(U)}="14";
(27,-7)*+{\Phi(\Sigma X)}="16";
{\ar^{\ell(f)} "0";"2"};
{\ar^{\ell(a)} "2";"4"};
{\ar^{\ell(b)} "4";"6"};
{\ar@{=} "0";"10"};
{\ar_{\cong} "2";"12"};
{\ar_{\Phi_{\mathcal{U}}(\eta_U)}^{\cong} "4";"14"};
{\ar_{\cong}^{\lambda_X} "6";"16"};
{\ar_{\Phi(\underline{f})} "10";"12"};
{\ar_(0.46){\Phi(\eta_U\circ\underline{a})} "12";"14"};
{\ar_{\Phi(\sigma(\underline{c}))} "14";"16"};
{\ar@{}|\circlearrowright "0";"12"};
{\ar@{}|\circlearrowright "3";"14"};
{\ar@{}|\circlearrowright "4";"16"};
\endxy
\]
The upper row is a distinguished triangle in $\mathcal{C}_{\mathcal{N}}$.
\end{proof}

\begin{rem}
Proposition \ref{PropEquivZIandCN} and Corollary \ref{CorTriaEquivZIandCN} generalize the triangle equivalence \cite[Theorem 4.1, Theorem 4.7]{IYa} which has been established for pair of co-$t$-structures arising from a presilting subcategory\footnote{The author whishes to thank Professor Osamu Iyama for informing him of this.}.
\end{rem}

Theorem \ref{ThmBij} gives the following.
\begin{cor}\label{CorBijHTCP}
Let $\mathcal{P}$ be a Hovey TCP. Then we have the following mutually inverse bijections.
\begin{eqnarray*}
\mathbf{R}_{\mathcal{P}}\colon\mathfrak{M}_{\mathcal{P}}\to\mathfrak{CP}(\mathcal{C}_{\mathcal{N}})&;&(\mathcal{A},\mathcal{B})\mapsto(\ell(\mathcal{A}),\ell(\mathcal{B})),\\
\mathbf{I}_{\mathcal{P}}\colon\mathfrak{CP}(\mathcal{C}_{\mathcal{N}})\to\mathfrak{M}_{\mathcal{P}}&;&(\mathscr{E},\mathscr{F})\mapsto(\mathcal{U}\cap\ell^{-1}(\mathscr{E}),\mathcal{T}\cap\ell^{-1}(\mathscr{F})).
\end{eqnarray*}
\end{cor}
\begin{proof}
Since $\Phi\colon\mathcal{Z}/\mathcal{I}\overset{\simeq}{\longrightarrow}\mathcal{C}_{\mathcal{N}}$ is a triangle equivalence, it induces a $\mathbb{Z}$-equivariant isomorphism (i.e., a bijection compatible with the shifts)
\[ \mathfrak{CP}(\mathcal{Z}/\mathcal{I})\overset{\cong}{\longleftrightarrow}\mathfrak{CP}(\mathcal{C}_{\mathcal{N}}) \]
which associates $(\Phi(\mathscr{L}),\Phi(\mathscr{R}))$ to $(\mathscr{L},\mathscr{R})\in\mathfrak{CP}(\mathcal{Z}/\mathcal{I})$ and $(\Phi^{-1}(\mathscr{E}),\Phi^{-1}(\mathscr{F}))$ to $(\mathscr{E},\mathscr{F})\in\mathfrak{CP}(\mathcal{C}_{\mathcal{N}})$.

Since $\ell|_{\mathcal{U}}\cong\Phi\circ\overline{\sigma}$ and $\ell|_{\mathcal{T}}\cong\Phi\circ\overline{\omega}$ by the commutativity of $(\ref{Diag_ZIandCN})$, this follows from Theorem \ref{ThmBij} and Remark \ref{RemBijHTCP}.
\end{proof}

\begin{rem}
Since $\mathcal{N}\subseteq\mathcal{C}$ is a triangulated subcategory, $(\mathcal{S},\mathcal{V})$ becomes a cotorsion pair on $\mathcal{N}$, satisfying $\mathcal{S}\cap\mathcal{V}=\mathcal{S}\cap\mathcal{T}=\mathcal{U}\cap\mathcal{V}$.
Its heart $\mathscr{A}_{(\mathcal{S},\mathcal{V})}$ is a full subcategory of both $\mathscr{A}_{(\mathcal{S},\mathcal{T})}$ and $\mathscr{A}_{(\mathcal{U},\mathcal{V})}$, which satisfies $\mathscr{A}_{(\mathcal{S},\mathcal{V})}=\mathscr{A}_{(\mathcal{S},\mathcal{T})}\cap\mathscr{A}_{(\mathcal{U},\mathcal{V})}$.
\end{rem}

\begin{claim}\label{ClaimTC2}
If $\mathcal{P}$ is a Hovey TCP, we have the following.
\begin{enumerate}
\item $(\mathcal{S},\mathcal{T})$ is a $t$-structure on $\mathcal{C}$ $\Leftrightarrow$ $(\mathcal{S},\mathcal{V})$ is a $t$-structure on $\mathcal{N}$ $\Leftrightarrow$ $(\mathcal{U},\mathcal{V})$ is a $t$-structure on $\mathcal{C}$ .
\item $\mathcal{S}=\mathcal{T}$ $(\Leftrightarrow\ \mathcal{U}=\mathcal{V})$ $\Rightarrow$ $\mathcal{S}=\mathcal{V}$, i.e., $\mathcal{S}$ is a cluster-tilting subcategory of $\mathcal{N}$.
\item $(\mathcal{S},\mathcal{T})$ is a co-$t$-structure on $\mathcal{C}$ $\Leftrightarrow$ $(\mathcal{S},\mathcal{V})$ is a co-$t$-structure on $\mathcal{N}$ $\Leftrightarrow$ $(\mathcal{U},\mathcal{V})$ is a co-$t$-structure on $\mathcal{C}$.
\end{enumerate}
\end{claim}
\begin{proof}
This immediately follows from Claims \ref{ClaimConcentric} and \ref{ClaimCot}.
\end{proof}

\begin{fact}\label{FactChen}(\cite[Theorem 3.1, Theorem 3.3]{C})
For any recollement of triangulated categories
\begin{equation}\label{DiagRecoll}
\xy
(-20,0)*+{\mathcal{N}}="0";
(0,0)*+{\mathcal{C}}="2";
(20,0)*+{\mathcal{C}_{\mathcal{N}}}="4";
(-10,4)*+{\perp}="11";
(-10,-5)*+{\perp}="12";
(10,4)*+{\perp}="13";
(10,-5)*+{\perp}="14";
{\ar@{^(->}_{i} "0";"2"};
{\ar@/^1.80pc/^{i^{!}} "2";"0"};
{\ar@/_1.80pc/_{i^{\ast}} "2";"0"};
{\ar_{\ell} "2";"4"};
{\ar@/^1.80pc/^{j_{\ast}} "4";"2"};
{\ar@/_1.80pc/_{j_{!}} "4";"2"};
\endxy,
\end{equation}
the following holds.
\begin{enumerate}
\item For any $(\mathcal{S},\mathcal{V})\in\mathfrak{CP}(\mathcal{N})$ and any $(\mathscr{E},\mathscr{F})\in\mathfrak{CP}(\mathcal{C}_{\mathcal{N}})$, we can {\it glue} them by
\[ \mathcal{A}=i^{\ast-1}(\mathcal{S})\cap\ell^{-1}(\mathscr{E}),\ \ \mathcal{B}=i^{!-1}(\mathcal{V})\cap\ell^{-1}(\mathscr{F}) \]
to obtain $(\mathcal{A},\mathcal{B})\in\mathfrak{CP}(\mathcal{C})$, which satisfies $(\mathcal{A},\mathcal{B})|_{\mathcal{N}}=(\mathcal{S},\mathcal{V})$ and $(\ell(\mathcal{A}),\ell(\mathcal{B}))=(\mathscr{E},\mathscr{F})$.
\item If $(\mathcal{A},\mathcal{B})\in\mathfrak{CP}(\mathcal{C})$ satisfies
\begin{equation}\label{ABCond1}
ii^{\ast}(\mathcal{A})\subseteq\mathcal{A},\ \ ii^!(\mathcal{B})\subseteq\mathcal{B}\ \ \text{and}\ \ j_{\ast}\ell(\mathcal{B})\subseteq\mathcal{B},
\end{equation}
then
\begin{itemize}
\item[{\rm (i)}] $\mathcal{S}=i^{-1}(\mathcal{A})=\mathcal{A}\cap\mathcal{N}$ and $\mathcal{V}=i^{-1}(\mathcal{B})=\mathcal{B}\cap\mathcal{N}$ form a cotorsion pair $(\mathcal{S},\mathcal{V})\in\mathfrak{CP}(\mathcal{N})$.
\item[{\rm (ii)}] $\mathscr{E}=\ell(\mathcal{A})$ and $\mathscr{F}=\ell(\mathcal{B})$ (i.e. the essential images of $\mathcal{A}$ and $\mathcal{B}$ in $\mathcal{C}_{\mathcal{N}}$) form a cotorsion pair $(\mathscr{E},\mathscr{F})\in\mathfrak{CP}(\mathcal{C}_{\mathcal{N}})$.
\item[{\rm (iii)}] $(\mathcal{S},\mathcal{V})$ and $(\mathscr{E},\mathscr{F})$ glue to recover $(\mathcal{A},\mathcal{B})$.
\end{itemize}
\end{enumerate}
\end{fact}

\medskip

If we fix $(\mathcal{S},\mathcal{V})\in\mathfrak{CP}(\mathcal{N})$, then Fact \ref{FactChen} can be rephrased as follows.
\begin{cor}\label{CorChen}
Let $(\ref{DiagRecoll})$ be a recollement of triangulated categories, and let $(\mathcal{S},\mathcal{V})\in\mathfrak{CP}(\mathcal{N})$ be a fixed cotorsion pair. Put
\[ \mathfrak{M}^{\prime}_{(\mathcal{S},\mathcal{V})}=\Set{ (\mathcal{A},\mathcal{B})\in\mathfrak{CP}(\mathcal{C})| \begin{array}{c}\mathcal{S}\subseteq\mathcal{A}\subseteq {}^{\perp}\mathcal{V}[1]\\ \mathcal{V}\subseteq\mathcal{B}\subseteq\mathcal{S}[-1]^{\perp}\end{array},\ \mathrm{Ext}^1_{\mathcal{C}_{\mathcal{N}}}(\ell(\mathcal{A}),\ell(\mathcal{B}))=0 }. \]
Here, $\mathcal{S}[-1]^{\perp}$ and ${}^{\perp}\mathcal{V}[1]$ denotes the right and left orthogonal subcategories taken in $\mathcal{C}$.

Then {\rm (1)} and {\rm (2)} in Fact \ref{FactChen} give the following mutually inverse bijections.
\begin{eqnarray*}
\mathbf{R}_{(\mathcal{S},\mathcal{V})}\colon \mathfrak{M}^{\prime}_{(\mathcal{S},\mathcal{V})}&\to&\mathfrak{CP}(\mathcal{C}_{\mathcal{N}})\\
(\mathcal{A},\mathcal{B})&\mapsto&(\ell(\mathcal{A}),\ell(\mathcal{B}))\\
\mathbf{I}_{(\mathcal{S},\mathcal{V})}\colon \mathfrak{CP}(\mathcal{C}_{\mathcal{N}})&\to&\mathfrak{M}^{\prime}_{(\mathcal{S},\mathcal{V})}\\
(\mathscr{E},\mathscr{F})&\mapsto&(i^{\ast-1}(\mathcal{S})\cap\ell^{-1}(\mathscr{E}),i^{!-1}(\mathcal{V})\cap\ell^{-1}(\mathscr{F}))
\end{eqnarray*}
\end{cor}
\begin{proof}
It suffices to show that $(\mathcal{A},\mathcal{B})\in\mathfrak{CP}(\mathcal{C})$ satisfies
\begin{equation}\label{EQUIV1}
(\ref{ABCond1})\ \ \text{and}\ \ i^{-1}(\mathcal{A})=\mathcal{S},\ i^{-1}(\mathcal{B})=\mathcal{V}
\end{equation}
if and only if it satisfies
\begin{equation}\label{EQUIV2}
\mathcal{S}\subseteq\mathcal{A}\subseteq {}^{\perp}\mathcal{V}[1],\ \mathcal{V}\subseteq\mathcal{B}\subseteq\mathcal{S}[-1]^{\perp}\ \ \text{and}\ \ \mathrm{Ext}^1_{\mathcal{C}_{\mathcal{N}}}(\ell(\mathcal{A}),\ell(\mathcal{B}))=0.
\end{equation}

\smallskip

Suppose $(\mathcal{A},\mathcal{B})$ satisfies $(\ref{EQUIV1})$. Then we have
\[ i^{\ast}(\mathcal{A})\subseteq i^{-1}(\mathcal{A})=\mathcal{S},\ i^!(\mathcal{B})\subseteq i^{-1}(\mathcal{B})=\mathcal{V},\ j_{\ast}\ell(\mathcal{B})\subseteq\mathcal{B}. \]
Since $(\mathcal{A},\mathcal{B})\in\mathfrak{CP}(\mathcal{C})$ and $(\mathcal{S},\mathcal{V})\in\mathfrak{CP}(\mathcal{N})$, this implies
\[ \mathrm{Ext}^1_{\mathcal{N}}(i^{\ast}(\mathcal{A}),\mathcal{V})=0,\ \mathrm{Ext}^1_{\mathcal{N}}(\mathcal{S},i^!(\mathcal{B}))=0,\ \mathrm{Ext}^1_{\mathcal{C}}(\mathcal{A},j_{\ast}\ell(\mathcal{B}))=0. \]
By the adjointness, this means
\[ \mathrm{Ext}^1_{\mathcal{C}}(\mathcal{A},\mathcal{V})=0,\ \mathrm{Ext}^1_{\mathcal{C}}(\mathcal{S},\mathcal{B})=0,\ \mathrm{Ext}^1_{\mathcal{C}_{\mathcal{N}}}(\ell(\mathcal{A}),\ell(\mathcal{B}))=0, \]
which is equivalent to $(\ref{EQUIV2})$.

Conversely, suppose $(\mathcal{A},\mathcal{B})$ satisfies $(\ref{EQUIV2})$. Then $i^{-1}(\mathcal{A})=\mathcal{N}\cap\mathcal{A}$ satisfies
\[ \mathcal{S}\subseteq \mathcal{N}\cap\mathcal{A}\subseteq \mathcal{N}\cap {}^{\perp}\mathcal{V}[1]=\mathcal{S}, \]
and thus $i^{-1}(\mathcal{A})=\mathcal{S}$. Similarly, we have $i^{-1}(\mathcal{B})=\mathcal{V}$. Thus $ii^{\ast}(\mathcal{A})\subseteq\mathcal{A}\ (\Leftrightarrow\ i^{\ast}(\mathcal{A})\subseteq\mathcal{S})$ follows from $\mathrm{Ext}^1_{\mathcal{N}}(i^{\ast}(\mathcal{A}),\mathcal{V})=\mathrm{Ext}^1_{\mathcal{C}}(\mathcal{A},\mathcal{V})=0$. Dually for $ii^!(\mathcal{B})\subseteq\mathcal{B}$. Similarly, $j_{\ast}\ell(\mathcal{B})\subseteq\mathcal{B}$ follows from $\mathrm{Ext}^1_{\mathcal{C}_{\mathcal{N}}}(\ell(\mathcal{A}),\ell(\mathcal{B}))=0$ and the adjointness.
\end{proof}

As a corollary, we can associate a Hovey TCP to any recollement with specified $(\mathcal{S},\mathcal{V})\in\mathfrak{CP}(\mathcal{N})$, as follows.
\begin{cor}\label{CorRecolltoHTCP}
Let $(\ref{DiagRecoll})$ be a recollement, and let $(\mathcal{S},\mathcal{V})\in\mathfrak{CP}(\mathcal{N})$ be any cotorsion pair. If we put $\mathcal{U}=i^{\ast-1}(\mathcal{S})$ and $\mathcal{T}=i^{!-1}(\mathcal{V})$, then $\mathcal{P}=((\mathcal{S},\mathcal{T}),(\mathcal{U},\mathcal{V}))$ is a Hovey TCP on $\mathcal{C}$. In particular, we have $\mathcal{U}={}^{\perp}\mathcal{V}[1]$ and $\mathcal{T}=\mathcal{S}[-1]^{\perp}$ in $\mathcal{C}$.
\end{cor}
\begin{proof}
Remark that we have
\[ (\mathcal{S},\mathcal{T})=\mathbf{I}_{(\mathcal{S},\mathcal{V})}((0,\mathcal{C}_{\mathcal{N}}))\ \ \text{and}\ \ (\mathcal{U},\mathcal{V})=\mathbf{I}_{(\mathcal{S},\mathcal{V})}((\mathcal{C}_{\mathcal{N}},0)). \]
Thus by Fact \ref{FactChen}, we have $(\mathcal{S},\mathcal{T}),(\mathcal{U},\mathcal{V})\in\mathfrak{CP}(\mathcal{C})$.
Since $(\mathcal{S},\mathcal{V})\in\mathfrak{CP}(\mathcal{N})$, we also have $\mathrm{Ext}^1_{\mathcal{C}}(\mathcal{S},\mathcal{V})=\mathrm{Ext}^1_{\mathcal{N}}(\mathcal{S},\mathcal{V})=0$, which means $\mathcal{P}$ is a TCP.
%Concentricity follows from\footnote{See also Remark \ref{RemarkAdditional} (2).}
%\[ \mathcal{S}\cap\mathcal{T}=\mathcal{S}\cap\mathcal{N}\cap\mathcal{T}=\mathcal{S}\cap\mathcal{V}, \]
%\[ \mathcal{U}\cap\mathcal{V}=\mathcal{U}\cap\mathcal{N}\cap\mathcal{V}=\mathcal{S}\cap\mathcal{V}. \]
Since $\mathcal{S}\ast\mathcal{V}[1]=\mathcal{N}\subseteq\mathcal{C}$ is triangulated subcategory, $\mathcal{P}$ is a Hovey TCP.
\end{proof}

Thus a Hovey TCP can be regarded as a generalization of a recollement (together with a chosen cotorsion pair $(\mathcal{S},\mathcal{V})\in\mathfrak{CP}(\mathcal{N})$). 
\begin{prop}\label{PropGeneralRecollement}
Let $(\ref{DiagRecoll})$ be a recollement, let $(\mathcal{S},\mathcal{V})\in\mathfrak{CP}(\mathcal{N})$ be a cotorsion pair, and let $\mathcal{P}=((\mathcal{S},\mathcal{T}),(\mathcal{U},\mathcal{V}))$ be the associated Hovey TCP obtained in Corollary \ref{CorRecolltoHTCP}. Then we have $\mathfrak{M}_{\mathcal{P}}=\mathfrak{M}^{\prime}_{(\mathcal{S},\mathcal{V})}$, and the maps $\mathbf{R}_{(\mathcal{S},\mathcal{V})}$ and $\mathbf{I}_{(\mathcal{S},\mathcal{V})}$ in Fact \ref{FactChen} agree with $\mathbf{R}_{\mathcal{P}}$ and $\mathbf{I}_{\mathcal{P}}$ in Corollary \ref{CorBijHTCP}.

With this view, Corollary \ref{CorBijHTCP} can be regarded as a generalization of Corollary \ref{CorChen} to Hovey TCP.
\end{prop}
\begin{proof}
This immediately follows from $\mathcal{U}=i^{\ast-1}(\mathcal{S})$ and $\mathcal{T}=i^{!-1}(\mathcal{V})$.
\end{proof}

\section*{Acknowledgement}
This article has been written when the author was staying at LAMFA, l'Universit\'{e} de Picardie-Jules Verne, by the support of JSPS Postdoctoral Fellowships for Research Abroad. He wishes to thank the hospitality of Professor Serge Bouc, Professor Radu Stancu and the members of LAMFA.

\end{document}